\documentclass[12pt,a4]{amsart}
\usepackage[a4paper, left=26mm, right=26mm, top=28mm, bottom=34mm]{geometry}
\usepackage[all]{xy}
\usepackage{amsmath}
\usepackage{amscd}
\usepackage{amssymb}
\usepackage{amsthm}
\usepackage{url}
\usepackage{mathrsfs}
\usepackage{comment}
\theoremstyle{definition}
\newtheorem{thm}{Theorem}[section]
\newtheorem{lem}[thm]{Lemma}
\newtheorem{cor}[thm]{Corollary}
\newtheorem{prop}[thm]{Proposition}

\theoremstyle{definition}

\newtheorem{rem}[thm]{Remark}
\newtheorem{claim}[thm]{Claim}
\newtheorem{defn}[thm]{Definition}

\newtheorem{ex}[thm]{Example}

\def\R{{\mathbb R}}
\def\Z{{\mathbb Z}}

\def\O{{\mathcal O}}

\def\Stab{\mathop{\mathrm{Stab}}\nolimits}

\def\Aut{\mathop{\mathrm{Aut}}\nolimits}

\def\Gal{\mathop{\mathrm{Gal}}\nolimits}

\def\Spa{\mathop{\rm Spa}}
\def\Spec{\mathop{\rm Spec}}
\def\Spf{\mathop{\rm Spf}}

\def\Supp{\mathop{\rm Supp}}

\def\sep{\text{\rm sep}}

\def\B{\mathbb{B}}
\def\D{\mathbb{D}}

\newcommand{\plim}[1][]{\mathop{\varprojlim}\limits_{#1}}
\newcommand{\ilim}[1][]{\mathop{\varinjlim}\limits_{#1}}
\renewcommand{\labelenumi}{(\arabic{enumi})}

\newcommand{\et}{\mathrm{\acute{e}t}}

\newcommand{\ad}{\mathrm{ad}}

\def\cosq{\mathop{\mathrm{cosq}}\nolimits}

\def\Spa{\mathop{\mathrm{Spa}}\nolimits}
\def\rig{\mathop{\mathrm{rig}}\nolimits}
\numberwithin{equation}{section}
\begin{document}

\title[Local constancy of \'etale cohomology of rigid analytic varieties]{Uniform local constancy of \'etale cohomology of rigid analytic varieties}

\author{Kazuhiro Ito}
\address{Department of Mathematics, Faculty of Science, Kyoto University, Kyoto 606-8502, Japan}
\email{kito@math.kyoto-u.ac.jp}




\maketitle

\begin{abstract}
We prove some $\ell$-independence results on local constancy of \'etale cohomology of rigid analytic varieties.
As a result,
we show that
a closed subscheme of a proper scheme over
an algebraically closed complete non-archimedean field
has a small open neighborhood in the analytic topology
such that,
for every prime number $\ell$ different from the residue characteristic,
the closed subscheme and the open neighborhood have the same \'etale cohomology with $\Z/\ell\Z$-coefficients.
The existence of such an open neighborhood for each $\ell$ was proved by Huber.
A key ingredient in the proof is
a uniform refinement of a theorem of Orgogozo on the
compatibility of the nearby cycles over general bases with base change.
\end{abstract}

\setcounter{tocdepth}{1}
\tableofcontents

\section{Introduction}\label{Section:Introduction}

Let $K$ be an algebraically closed complete non-archimedean field
whose topology is given by a valuation
$ \vert \cdot \vert \colon K \to \R_{\geq 0}$
of rank $1$.
Let $\O=K^\circ$ be the ring of integers of $K$.
In this paper,
we study local constancy of \'etale cohomology of rigid analytic varieties over $K$, or more precisely,
of adic spaces of finite type over $\Spa(K, \O)$.

\subsection{A main result}\label{Subsection:A main result}
The theory of \'etale cohomology for adic spaces was developed by Huber; see \cite{Huber96}.
Huber obtained several finiteness results on \'etale cohomology of adic spaces in a series of papers \cite{Huber98a, Huber98b, Huber07}.
Let us recall one of the main results of \cite{Huber98b}; see \cite[Theorem 3.6]{Huber98b} for a more precise statement.

\begin{thm}[{Huber \cite[Theorem 3.6]{Huber98b}}]\label{Theorem:Huber's theorem}
We assume that $K$ is of characteristic zero.
Let $X$ be a separated adic space of finite type over $\Spa(K, \O)$
and $Z$ a closed adic subspace of $X$.
Let $n$ be a positive integer invertible in $\O$.
Then there exists an open subset $V$ of $X$ containing $Z$ such that
the restriction map
\[
H^i(V, \Z/n\Z) \to H^i(Z, \Z/n\Z)
\]
on \'etale cohomology groups is an isomorphism for every integer $i$.
Moreover, we can assume that $V$ is quasi-compact.
\end{thm}

It is a natural question to ask whether we can take an open subset $V$ as in Theorem \ref{Theorem:Huber's theorem}
independent of $n$.
In the present paper,
we answer this question in the affirmative for
adic spaces which arise from schemes of finite type over $\O$.

More precisely, we will prove the following theorem.
For a scheme $\mathcal{X}$ of finite type over $\O$,
let $\widehat{\mathcal{X}}$ denote the $\varpi$-adic formal completion of $\mathcal{X}$, where $\varpi \in K^\times$ is an element with
$\vert \varpi \vert<1$.
The Raynaud generic fiber of $\widehat{\mathcal{X}}$
is denoted by $(\widehat{\mathcal{X}})^{\rig}$ in this section,
which is an adic space of finite type over $\Spa(K, \O)$.
(It is denoted by $d(\widehat{\mathcal{X}})$ in \cite{Huber96} and in the main body of this paper.)

\begin{thm}[Theorem \ref{Theorem:tubular neighborhood direct image}]\label{Theorem:main result adic space}
Let
$\mathcal{Z} \hookrightarrow \mathcal{X}$
be a closed immersion of separated schemes of finite type over $\O$.
We have a closed embedding
$(\widehat{\mathcal{Z}})^{\rig} \hookrightarrow (\widehat{\mathcal{X}})^{\rig}$.
Then there exists an open subset $V$ of $(\widehat{\mathcal{X}})^{\rig}$ containing $(\widehat{\mathcal{Z}})^{\rig}$ such that,
for every positive integer $n$ invertible in $\O$,
the restriction map
\[
H^i(V, \Z/n\Z) \to H^i((\widehat{\mathcal{Z}})^{\rig}, \Z/n\Z)
\]
on \'etale cohomology groups is an isomorphism for every integer $i$.
Moreover, we can assume that $V$ is quasi-compact.
\end{thm}

A more precise statement is given in Theorem \ref{Theorem:tubular neighborhood direct image}.
In this paper, we will use de Jong's alterations in several ways.
This is the main reason why we restrict ourselves to the case where adic spaces arise from schemes of finite type over $\O$.
We remark that, in our case, we need not impose any conditions on the characteristic of $K$.
We will also prove an analogous statement for \'etale cohomology with compact support; see Theorem \ref{Theorem:tubular neighborhood compact support}.

\begin{rem}\label{Remark:motivation}
In \cite{Scholze},
Scholze proved
the weight-monodromy conjecture
for
a projective smooth variety $X$ over a non-archimedean local field $L$ of mixed characteristic $(0, p)$
which is a set-theoretic complete intersection in a projective smooth toric variety, by reduction to the function field case proved by Deligne.
In the proof,
Scholze used Theorem \ref{Theorem:Huber's theorem}
to construct,
for a fixed prime number $\ell \neq p$,
a projective smooth variety $Y$ over a function field of characteristic $p$
and
an appropriate mapping from
\'etale cohomology with $\Z/\ell\Z$-coefficients of $X$ to that of $Y$.
The initial motivation for the present study is,
following the method of Scholze,
to prove that
an analogue of the weight-monodromy conjecture
holds for
\'etale cohomology with $\Z/\ell\Z$-coefficients
of such a variety $X$
for all but finitely many $\ell \neq p$
by reduction to an ultraproduct variant of Weil II established by Cadoret \cite{Cadoret}.
For this, we shall use Theorem \ref{Theorem:main result adic space} instead of Theorem \ref{Theorem:Huber's theorem}.
The details are given in \cite{Ito}.
\end{rem}

\subsection{Local constancy of higher direct images with proper support}\label{Subsection:Local constancy of higher direct images with proper support introduction}

For the proof of Theorem \ref{Theorem:main result adic space},
we need to investigate local constancy of higher direct images with proper support for morphisms of adic spaces.
Before stating our results on higher direct images with proper support,
let us give an outline of the proof of Theorem \ref{Theorem:Huber's theorem}.

\textit{Sketch of the proof of Theorem \ref{Theorem:Huber's theorem}.}
We assume that $K$ is of characteristic zero.
For simplicity,
we assume that the closed embedding
$Z \hookrightarrow X$
is of the form
$(\widehat{\mathcal{Z}})^{\rig} \hookrightarrow (\widehat{\mathcal{X}})^{\rig}$
for a closed immersion of finite presentation
$\mathcal{Z} \hookrightarrow \mathcal{X}$
of separated schemes of finite type over $\O$.
By considering the blow-up of
$\mathcal{X}$ along $\mathcal{Z}$,
we may assume further that
the closed subscheme $\mathcal{Z}$ is defined by one global function
$f \in \O_{\mathcal{X}}(\mathcal{X})$.
Let
\[
f \colon \mathcal{X} \to \Spec \O[T]
\]
be the morphism defined by $T \mapsto f$.
The Raynaud generic fiber of the $\varpi$-adic formal completion of
$\Spec \O[T]$ is the unit disc
$
\B(1) := \Spa (K\langle T \rangle, \O \langle T \rangle).
$
The set of $K$-rational points of $\B(1)$ is identified with the set
\[
\B(1)(K)=\{ x \in K \, \vert \, \vert x \vert \leq 1 \}.
\]
The morphism $f$ induces the following morphism of adic spaces:
\[
f^{\rig} \colon (\widehat{\mathcal{X}})^{\rig} \to \B(1).
\]
The inverse image $(f^{\rig})^{-1}(0)$ of the origin $0 \in \B(1)$ is the closed subspace $(\widehat{\mathcal{Z}})^{\rig}$.

We fix a positive integer $n$ invertible in $\O$.
We want to take an open subset $V$ in Theorem \ref{Theorem:Huber's theorem} as the inverse image
\[
V=(f^{\rig})^{-1}(\B(\epsilon))
\]
of the disc $\B(\epsilon) \subset \B(1)$ of radius $\epsilon$ centered at $0$
for a small $\epsilon \in \vert K^\times \vert$.
Such a subset is called a
\textit{tubular neighborhood}
of $(\widehat{\mathcal{Z}})^{\rig}$.
For this, we have to compute \'etale cohomology with $\Z/n\Z$-coefficients
of $(f^{\rig})^{-1}(\B(\epsilon))$ for
a small $\epsilon \in \vert K^\times \vert$.
By the Leray spectral sequence for $f^{\rig}$,
it suffices to compute the cohomology group
\[
H^i(\B(\epsilon), R^jf^{\rig}_*\Z/n\Z)
\]
for all $i, j$.
The key steps are as follows.
\begin{itemize}
    \item By \cite[Theorem 2.1]{Huber98b},
    the \'etale sheaf
    $
    R^jf^{\rig}_*\Z/n\Z
    $
    is an oc-quasi-constructible \'etale sheaf of $\Z/n\Z$-modules in the sense of \cite[Definition 1.4]{Huber98b}.
    It follows that there exists an element $\epsilon_1 \in \vert K^\times \vert$ such that
    the restriction
    $(R^jf^{\rig}_*\Z/n\Z)\vert_{\B(\epsilon_1) \backslash \{ 0 \}}$
    is a locally constant $\Z/n\Z$-sheaf of finite type.
    \item By the $p$-adic Riemann existence theorem of L\"utkebohmert
    \cite[Theorem 2.2]{Lutkebohmert93},
    there exists an element $\epsilon_0 \in \vert K^\times \vert$ with
    $\epsilon_0 \leq \epsilon_1$
    such that
    $(R^jf^{\rig}_*\Z/n\Z)\vert_{\B(\epsilon_0) \backslash \{ 0 \}}$
    is trivialized by a Kummer covering $\varphi_m \colon \B(\epsilon^{1/m}_0) \backslash \{ 0 \} \to \B(\epsilon_0) \backslash \{ 0 \}$ defined by $T \mapsto T^m$.
\end{itemize}
Then the desired result can be obtained by explicit calculations.
\qed

In our case, the problem is to show
that
$\epsilon_0$ and $\epsilon_1$ in the above argument
can be taken independent of $n$.
To overcome this problem,
by de Jong's alterations and by cohomological descent,
we reduce to the case where
there exists an element $\epsilon \in \vert K^\times \vert$
with $\epsilon \leq 1$
such that the restriction
\[
(f^{\rig})^{-1}(\B(\epsilon) \backslash \{ 0 \}) \to \B(\epsilon) \backslash \{ 0 \}
\]
of $f^{\rig}$ is \textit{smooth}.
In this case, we will analyze
the higher direct image sheaf with proper support
\[
R^jf^{\rig}_!\Z/n\Z
\]
on $\B(1)$, which is defined in \cite[Definition 5.4.4]{Huber96}.
An important fact is that,
since $f^{\rig}$ is smooth over $\B(\epsilon) \backslash \{ 0 \}$,
the restriction
$
(R^jf^{\rig}_!\Z/n\Z)\vert_{\B(\epsilon) \backslash \{ 0 \}}
$
is a constructible \'etale sheaf of $\Z/n\Z$-modules
(in the sense of \cite[Definition 2.7.2]{Huber96})
for every positive integer $n$ invertible in $\O$ by \cite[Theorem 6.2.2]{Huber96}.

Our main result on local constancy of higher direct images with proper support is as follows.
We do not suppose that $K$ is of characteristic zero.
For elements $a, b \in \vert K^{\times} \vert$ with
$a < b \leq 1$,
let
$
\B(a, b) \subset \B(1)
$
be the annulus
with inner radius $a$ and outer radius $b$ centered at $0$.

\begin{thm}[Proposition \ref{Proposition:constant sheaves adapted} and Theorem \ref{Theorem:uniform local constancy}]\label{Theorem:uniform local constancy introduction}
Let $f \colon \mathcal{X} \to \Spec \O[T]$
be a separated morphism of finite presentation.
We assume that
there exists an element $\epsilon \in \vert K^\times \vert$
with $\epsilon \leq 1$
such that the induced morphism
\[
f^{\rig} \colon (\widehat{\mathcal{X}})^{\rig} \to \B(1)
\]
is smooth over $\B(\epsilon) \backslash \{ 0 \}$.
Then there exists an element
$\epsilon_0 \in \vert K^{\times} \vert$
with $\epsilon_0 \leq \epsilon$
such that,
for every 
positive integer $n$ invertible in $\O$,
the following two assertions hold:
\begin{enumerate}
    \item The restriction
    $
    (R^if^{\rig}_!\Z/n\Z)\vert_{\B(\epsilon_0) \backslash \{ 0 \}}
    $
    is a locally constant $\Z/n\Z$-sheaf of finite type
    for every $i$.
    \item 
    For elements $a, b \in \vert K^{\times} \vert$ with $a < b \leq \epsilon_0$,
    there exists a composition
    \[
    h \colon \B(c^{1/m}, d^{1/m}) \overset{\varphi_m}{\longrightarrow} \B(c, d) \overset{g}{\longrightarrow} \B(a, b)
    \]
    of
    a Kummer covering
    $
    \varphi_m
    $
    of degree $m$, where $m$ is invertible in $\O$,
    with
    a finite Galois \'etale morphism
    $g$,
    such that
    $(R^if^{\rig}_!\Z/n\Z)\vert_{\B(a, b)}$
    is trivialized by $h$ for every $i$.
    If $K$ is of characteristic zero, then we can take $g$ as a Kummer covering.
    (The morphism $g$ can be taken independent of $n$ although the integer $m$ possibly depends on $n$.)
\end{enumerate}
\end{thm}

\begin{rem}
We will prove Theorem \ref{Theorem:uniform local constancy introduction}
in a slightly more general setting involving certain sheaves on $\mathcal{X}$ which are not necessary constant; see Section \ref{Section:Local constancy of higher direct images with proper support for generically smooth morphisms} for details.
\end{rem}

Under the assumptions of Theorem \ref{Theorem:uniform local constancy introduction},
the same results hold for
the higher direct image sheaf $R^if^{\rig}_*\Z/n\Z$
by Poincar\'e duality \cite[Corollary 7.5.5]{Huber96},
which will imply Theorem \ref{Theorem:main result adic space}.

A key ingredient in the proof of Theorem \ref{Theorem:uniform local constancy introduction}
is the following uniform refinement of
a theorem of Orgogozo \cite[Th\'eor\`eme 2.1]{Orgogozo06} on the compatibility of
\textit{the sliced nearby cycles functors} with base change.
We also obtain a result on uniform unipotency of the sliced nearby cycles functors.
See Section \ref{Subsection:Sliced nearby cycles functor} for the definition of the sliced nearby cycles functors
and
see Definition \ref{Definition:base change, unipotent} for the terminology used in the following theorem.

\begin{thm}[Corollary \ref{Corollary:uniform base change for constant sheaves}]\label{Theorem:uniform base change for constant sheaves intro}
Let $S$ be a Noetherian excellent scheme and
$g \colon Y \to S$ a separated morphism of finite type.
There exists an alteration $S' \to S$ such that,
for every positive integer $n$ invertible on $S$,
the following assertions hold:
\begin{enumerate}
    \item The sliced nearby cycles complexes for
the base change $g_{S'} \colon Y_{S'} \to S'$ of $g$ and the constant sheaf $\Z/n\Z$ are \textit{compatible with any base change}.
    \item The sliced nearby cycles complexes for
$g_{S'} \colon Y_{S'} \to S'$ and the constant sheaf $\Z/n\Z$ are \textit{unipotent}.
\end{enumerate}
\end{thm}

Theorem \ref{Theorem:uniform base change for constant sheaves intro}
is a corollary of a more general result
(Theorem \ref{Theorem:uniform base change}),
which may be of independent interest.
The proof of
Theorem \ref{Theorem:uniform base change for constant sheaves intro} is
quite similar to that of \cite[Th\'eor\`eme 3.1.1]{Orgogozo19}.
In the proof, we need de Jong's alteration.

By using a comparison theorem of Huber \cite[Theorem 5.7.8]{Huber96},
we will deduce Theorem \ref{Theorem:uniform local constancy introduction} from Theorem \ref{Theorem:uniform base change for constant sheaves intro}.
Roughly speaking, Theorem \ref{Theorem:uniform local constancy introduction} (1) can be deduced from Theorem \ref{Theorem:uniform base change for constant sheaves intro} (1) by considering a specialization map from an adic space of finite type over $\Spa(K, \O)$ to its reduction;
see Section  \ref{Subsection:Specialization maps for higher direct image sheaves with proper support} and Section \ref{Subsection:Local constancy of higher direct images with proper support} for details.
Theorem \ref{Theorem:uniform local constancy introduction} (2)
can be deduced from
Theorem \ref{Theorem:uniform base change for constant sheaves intro} (2)
and
some properties of the \textit{discriminant function}
\[
\delta_h \colon [0, \infty) \to \R_{\geq 0}
\]
associated with a finite Galois \'etale covering
$h \colon Y \to \B(1) \backslash \{ 0 \}$ defined in \cite{Lutkebohmert93, Ramero05, LS05}.
See
Section \ref{Subsection:Tame sheaves on annuli} and
Appendix \ref{Appendix:finite etale coverings of annuli} for details.
In the proofs of both parts of Theorem \ref{Theorem:uniform local constancy introduction},
points of rank $2$ of (finite Galois \'etale coverings of) $\B(1)$ play important roles.

\subsection{The organization of this paper}\label{Subsection:The organization of this paper}
This paper is organized as follows.
In Section \ref{Section:Nearby cycles over general bases},
we first recall the definition of the sliced nearby cycles functors.
Then we formulate our main result
(Theorem \ref{Theorem:uniform base change})
on the sliced nearby cycles functors.
In Section \ref{Section:Proof of theorem uniform base change},
we prove Theorem \ref{Theorem:uniform base change}.

In Section \ref{Section:Tubular neighborhoods and main results},
we recall the definition of
tubular neighborhoods,
and then we state our main results
(Theorem \ref{Theorem:tubular neighborhood compact support} and
Theorem \ref{Theorem:tubular neighborhood direct image})
on \'etale cohomology of tubular neighborhoods.
In Section \ref{Section:Etale cohomology with compact support of adic spaces and nearby cycles},
we recall a comparison theorem of Huber and use it
to study the relation between
higher direct images with proper support for morphisms of adic spaces
and the sliced nearby cycles functors.
In Section \ref{Section:Local constancy of higher direct images with proper support for generically smooth morphisms},
we prove Theorem \ref{Theorem:uniform local constancy introduction} in a slightly more general setting.
In Section \ref{Section:proofs of main theorems},
we prove
Theorem \ref{Theorem:tubular neighborhood compact support} and
Theorem \ref{Theorem:tubular neighborhood direct image}
(and hence Theorem \ref{Theorem:main result adic space})
by using Theorem \ref{Theorem:uniform local constancy introduction}.

Finally,
in Appendix \ref{Appendix:finite etale coverings of annuli},
we prove
two theorems
(Theorem \ref{Theorem:split into annuli} and Theorem \ref{Theorem:trivialization of tame sheaf})
on finite \'etale coverings of annuli in the unit disc,
which are essentially proved in \cite{Lutkebohmert93, Ramero05, LS05}.

\section{Nearby cycles over general bases}\label{Section:Nearby cycles over general bases}

In this section, we formulate our main results on nearby cycles over general bases.
We will use the following notation.
Let $f \colon X \to S$ be a morphism of schemes.
For a morphism $T \to S$ of schemes,
the base change
$X \times_S T$
of $X$ is denoted by $X_T$
and
the base change of $f$
is denoted by
$f_T \colon X_T \to T$.
For a commutative ring $\Lambda$,
let
$D^+(X, \Lambda)$
be the derived category of bounded below complexes of
\'etale sheaves of $\Lambda$-modules on $X$.
For a complex $\mathcal{K} \in D^+(X, \Lambda)$,
the pull-back of $\mathcal{K}$ to $X_T$ is denoted by $\mathcal{K}_T$.
We often call an \'etale sheaf on $X$ simply a sheaf on $X$.

\subsection{Sliced nearby cycles functor}\label{Subsection:Sliced nearby cycles functor}

In this paper,
a scheme is called
a \textit{strictly local scheme}
if it is isomorphic to
an affine scheme
$\Spec R$ where $R$ is a strictly Henselian local ring.
Let $f \colon X \to S$
be a morphism of schemes.
Let $q \colon U \to S$ be a morphism
from
a strictly local scheme $U$.
The closed point of $U$ is denoted by $u$.
Let $\eta \in U$ be a point.
Let $\overline{\eta} \to U$ be an algebraic geometric point lying above $\eta$,
i.e.\
it is a geometric point lying above $\eta$ such that the residue field $\kappa(\overline{\eta})$ is a separable closure of the residue field $\kappa(\eta)$ of $\eta$.
The strict localization of $U$ at $\overline{\eta} \to U$
is denoted by $U_{(\overline{\eta})}$.
We have the following commutative diagram:
\[
\xymatrix{ X_{U_{(\overline{\eta})}} \ar[r]^-{j} \ar[d]^-{} & X_U \ar[d]^-{f_U}  & \ar[l]_-{i} X_u \ar[d]^-{} \\
U_{(\overline{\eta})} \ar[r]_-{} & U & \ar[l]^-{} u.
}
\]
Let $\Lambda$ be a commutative ring.
We have the following functor:
\[
R\Psi_{f_U, \overline{\eta}}:=i^{*}Rj_*j^* \colon D^{+}(X_U, \Lambda) \to D^{+}(X_u, \Lambda).
\]
This functor is called the \textit{sliced nearby cycles functor} in \cite{Illusie17}.
For a complex
$\mathcal{K}  \in D^{+}(X_U, \Lambda)$,
we have an action of
the absolute Galois group
$\Gal(\kappa(\overline{\eta})/\kappa(\eta))$ on
$R\Psi_{f_U, \overline{\eta}}(\mathcal{K})$
via the canonical isomorphism
\[
\Aut(U_{(\overline{\eta})}/\Spec (\O_{U, \eta})) \cong \Gal(\kappa(\overline{\eta})/\kappa(\eta)).
\]

Let
$q \colon V \to U$
be a local morphism
of strictly local schemes over $S$,
i.e.\ a morphism over $S$ which sends the closed point
$v$ of $V$ to
the closed point $u$ of $U$.
Let $\xi \in V$ be a point with image $\eta=q(\xi) \in U$.
For an algebraic geometric point
$\overline{\xi} \to V$
lying above $\xi$,
we have an algebraic geometric point
$\overline{\eta} \to U$
lying above $\eta$ by taking
the separable closure of $\kappa(\eta)$ in $\kappa(\overline{\xi})$.
We call $\overline{\eta} \to U$ the image of $\overline{\xi} \to V$
under the morphism $q$.
We have the following commutative diagram:
\[
\xymatrix{ X_{V_{(\overline{\xi})}} \ar[r]^-{} \ar[d]^-{q} & X_V \ar[d]^-{q}  & \ar[l]_-{} X_{v} \ar[d]^-{q} \\
X_{U_{(\overline{\eta})}} \ar[r]^-{} & X_U & \ar[l]_-{} X_u,
}
\]
where the vertical morphisms are induced by $q$.
For a complex
$\mathcal{K} \in D^{+}(X_U, \Lambda)$,
we have the following base change map:
\[
q^*R\Psi_{f_U, \overline{\eta}}(\mathcal{K}) \to R\Psi_{f_V, \overline{\xi}}(\mathcal{K}_V).
\]

We will use the following terminology.

\begin{defn}\label{Definition:G-unipotent}
Let $G$ be a group and $X$ a scheme.
We say that a sheaf $\mathcal{F}$ of $\Lambda$-modules on $X$ with a $G$-action is
\textit{$G$-unipotent}
if $\mathcal{F}$ has a finite filtration which is stable by the action of $G$ such that the action of $G$ on each successive quotient is trivial.
We say that a complex
$\mathcal{K} \in D^+(X, \Lambda)$
with a $G$-action is \textit{$G$-unipotent} if its cohomology sheaves are $G$-unipotent.
\end{defn}

\begin{defn}\label{Definition:base change, unipotent}
Let $f \colon X \to S$ be a morphism of schemes.
Let $\Lambda$ be a commutative ring and $\mathcal{K} \in D^{+}(X, \Lambda)$ a complex.
\begin{enumerate}
    \item We say that
    \textit{the sliced nearby cycles complexes for $f$ and $\mathcal{K}$ are compatible with any base change}
    (or simply that \textit{the nearby cycles for $f$ and $\mathcal{K}$ are compatible with any base change})
    if
    for every local morphism $q \colon V \to U$
    of strictly local schemes over $S$ and every algebraic geometric point
    $\overline{\xi} \to V$ with image $\overline{\eta} \to U$,
    the base change map
    \[
    q^*R\Psi_{f_U, \overline{\eta}}(\mathcal{K}_U) \to R\Psi_{f_V, \overline{\xi}}(\mathcal{K}_V).
    \]
    is an isomorphism.
    \item We say that
    \textit{the sliced nearby cycles complexes for $f$ and $\mathcal{K}$ are unipotent}
    (or simply that \textit{the nearby cycles for $f$ and $\mathcal{K}$ are unipotent})
    if for every morphism
    $q \colon U \to S$ from a strictly local scheme $U$, a point $\eta \in U$, and an algebraic geometric point $\overline{\eta} \to U$ lying above $\eta$,
    the complex
    $R\Psi_{f_{U}, \overline{\eta}}(\mathcal{K}_U)$
    is
    $\Gal(\kappa(\overline{\eta})/\kappa(\eta))$-unipotent.
\end{enumerate}
\end{defn}

\begin{rem}\label{Remark:vanishing topos}
We can restate Definition \ref{Definition:base change, unipotent} (1) in terms of vanishing topoi as follows.
Let $f \colon X \to S$ be a morphism of schemes.
Let
\[
X \overset{\leftarrow}{\times}_S S
\]
be the vanishing topos,
where the \'etale topos of a scheme $X$ is also denoted by $X$ by abuse of notation.
See \cite[Expos\'e XI]{STG} and \cite{Illusie17} for the definition and basic properties of the vanishing topos
$X \overset{\leftarrow}{\times}_S S$.
Let $\Lambda$ be a commutative ring.
We have a morphism of topoi
$\Psi_f \colon X \to X \overset{\leftarrow}{\times}_S S$.
The direct image functor  
\[
R\Psi_f \colon D^{+}(X, \Lambda) \to D^{+}(X \overset{\leftarrow}{\times}_S S, \Lambda)
\]
defined by $\Psi_f$ is called the \textit{nearby cycles functor}.
For a morphism $q \colon T \to S$ of schemes, we have a morphism of topoi
$
{\overset{\leftarrow}{q}} \colon X_T \overset{\leftarrow}{\times}_T T \to X \overset{\leftarrow}{\times}_S S
$
and a $2$-commutative diagram
\[
\xymatrix{ X_T \ar[r]^-{} \ar[d]^-{R\Psi_{f_T}} & X \ar[d]^-{R\Psi_f} \\
X_T \overset{\leftarrow}{\times}_T T \ar[r]^-{{\overset{\leftarrow}{q}}} & X \overset{\leftarrow}{\times}_S S,
}
\]
where $X_T \to X$ is the projection.
For a complex $\mathcal{K} \in D^{+}(X, \Lambda)$, 
we have the base change map
\[
c_{f, q}(\mathcal{K}) \colon ({\overset{\leftarrow}{q}})^*R\Psi_{f}(\mathcal{K}) \to R\Psi_{f_T}(\mathcal{K}_T).
\]

For a morphism
$f \colon X \to S$
of schemes and a complex
$\mathcal{K} \in D^{+}(X, \Lambda)$,
the sliced nearby cycles complexes for $f$ and $\mathcal{K}$ are compatible with any base change in the sense of Definition \ref{Definition:base change, unipotent} (1)
if and only if,
for every morphism
$q \colon T \to S$ of schemes, 
the base change map
$
c_{f, q}(\mathcal{K})
$
is an isomorphism.
This follows from the following descriptions of the stalks of
the nearby cycles functor and the sliced nearby cycles functors.

Let $x \to X$ be a geometric point of $X$
and let $s \to S$ denote the composition $x \to X \to S$.
Let $t \to S$ be a geometric point with a specialization map
$\alpha \colon t \to s$, i.e.\ an $S$-morphism
$\alpha \colon S_{(t)} \to S_{(s)}$,
where $S_{(s)}$ (resp.\ $S_{(t)}$) is the strict localization of $S$ at $s \to S$ (resp.\ $t \to S$).
The triple $(x, t, \alpha)$ defines a point of the vanishing topos
$X \overset{\leftarrow}{\times}_S S$
and every point of
$X \overset{\leftarrow}{\times}_S S$
is of this form (up to equivalence).
The topos
$X \overset{\leftarrow}{\times}_S S$
has enough points.
For the stalk
$R\Psi_f(\mathcal{K})_{(x, t, \alpha)}$ of $R\Psi_f(\mathcal{K})$ at $(x, t, \alpha)$,
we have an isomorphism
\[
R\Psi_f(\mathcal{K})_{(x, t, \alpha)} \cong R\Gamma(X_{(x)}\times_{S_{(s)}} S_{(t)}, \mathcal{K});
\]
see \cite[(1.3.2)]{Illusie17}.
Here the pull-back of $\mathcal{K}$ to $X_{(x)}\times_{S_{(s)}} S_{(t)}$ is also denoted by $\mathcal{K}$ and
we will use this notation in this paper when there is no possibility of confusion.

We have a similar description of the stalks of the sliced nearby cycles functors.
More precisely,
let $q \colon U \to S$
be a morphism
from a strictly local scheme $U$
and
$\overline{\eta} \to U$
an algebraic geometric point.
Let $x \to X_u$
be a geometric point of the special fiber $X_u$ of $X_U$.
Then,
since the morphism
$X_{U_{(\overline{\eta})}} \to X_U$ is quasi-compact and quasi-separated,
we have
\begin{equation}\label{equation:vanishing cycle stalk}
    R\Psi_{f_U, \overline{\eta}}(\mathcal{K}_U)_x \cong
R\Gamma((X_U)_{(x)} \times_{U} U_{(\overline{\eta})}, \mathcal{K}_U).
\end{equation}
\end{rem}

\subsection{Main results on nearby cycles over general bases}\label{Subsection:Main results on nearby cycles over general bases}

A proper surjective morphism
$f \colon X \to Y$
of Noetherian schemes is called an
\textit{alteration}
if it sends every generic point of $X$ to a generic point of $Y$ and it is generically finite, i.e.\
there exists a dense open subset $U \subset Y$ such that
the restriction
$f^{-1}(U) \to U$
is a finite morphism.
If furthermore $X$ and $Y$ are integral schemes,
then $f$ is called an integral alteration.
An alteration
$f \colon X \to Y$
is called a \textit{modification}
if there exists a dense open subset $U \subset Y$ such that
the restriction
$f^{-1}(U) \to U$
is an isomorphism.

Let $f \colon X \to S$ be a morphism of finite type of
Noetherian excellent schemes.
In \cite{Orgogozo06},
Orgogozo proved the following result:

\begin{thm}[{Orgogozo \cite[Th\'eor\`eme 2.1]{Orgogozo06}}]\label{Theorem:Orgogozo}
For a positive integer $n$ invertible on $S$
and for a constructible sheaf $\mathcal{F}$ of $\Z/n\Z$-modules on $X$,
there exists a modification
$S' \to S$
such that
the sliced nearby cycles complexes for
$f_{S'}$ and $\mathcal{F}_{S'}$
are compatible with any base change in the sense of Definition \ref{Definition:base change, unipotent} (1).
\end{thm}
\begin{proof}
See \cite[Th\'eor\`eme 2.1]{Orgogozo06} for the proof and for a more general result.
(Actually, Orgogozo formulated his results in terms of vanishing topoi.
See Remark \ref{Remark:vanishing topos}.)
\end{proof}

To prove Theorem \ref{Theorem:main result adic space},
we need a uniform refinement of Theorem \ref{Theorem:Orgogozo}.
More precisely,
we need
a modification (or an alteration)
$S' \to S$
such that,
for every positive integer $n$ invertible on $S$,
the sliced nearby cycles complexes for
$f_{S'}$ and the constant sheaf $\Z/n\Z$
are compatible with any base change.

In order to prove the existence of such a modification, we will use the methods developed in a recent paper \cite{Orgogozo19} of Orgogozo.
In fact,
by the same methods,
we can also prove
that 
there exists
an \textit{alteration}
$S' \to S$
such that,
for every positive integer $n$ invertible on $S$,
the sliced nearby cycles complexes for
$f_{S'}$ and the constant sheaf $\Z/n\Z$ are unipotent in the sense of Definition \ref{Definition:base change, unipotent} (2).
Such an alteration is also needed in the proof of Theorem \ref{Theorem:main result adic space}.

To formulate our results,
we recall the definition of
a \textit{locally unipotent sheaf} on a Noetherian scheme from \cite{Orgogozo19}.
Let $X$ be a Noetherian scheme.
In this paper,
we call
a finite set
$\mathfrak{X}= \{ X_\alpha \}_\alpha$
of locally closed subsets
of $X$
a \textit{stratification}
if
we have
$X=\coprod_\alpha X_\alpha$
(set-theoretically).
We endow each $X_\alpha$ with the reduced subscheme structure.

\begin{defn}[{Orgogozo \cite[1.2.1]{Orgogozo19}}]\label{Definition:locally unipotent}
Let $X$ be a Noetherian scheme and $\mathfrak{X}$ a stratification of $X$.
We say that an abelian sheaf $\mathcal{F}$ on $X$ is
\textit{locally unipotent along} $\mathfrak{X}$
if, for every morphism $U \to X$ from a strictly local scheme $U$ and every $X_\alpha \in \mathfrak{X}$,
the pull-back of $\mathcal{F}$ to $U \times_X X_\alpha$ has a finite filtration whose successive quotients are constant sheaves.
\end{defn}

\begin{rem}\label{Remark:locally unipotent is constructible along a stratification}
If a constructible abelian sheaf $\mathcal{F}$ on
a Noetherian scheme $X$ is locally unipotent along a stratification $\mathfrak{X}$, then it is constructible along $\mathfrak{X}$, i.e.\ for every $X_\alpha \in \mathfrak{X}$, the pull-back of $\mathcal{F}$ to $X_\alpha$ is locally constant.
(See \cite[1.2.2]{Orgogozo19}.)
\end{rem}

The main result on nearby cycles over general bases is as follows.

\begin{thm}\label{Theorem:uniform base change}
Let $S$ be a Noetherian excellent scheme.
Let $f \colon X \to S$ be a proper morphism.
Let $\mathfrak{X}$ be a stratification of $X$.
Then there exists an alteration $S' \to S$ such that,
for every positive integer $n$ invertible on $S$ and
every complex
$\mathcal{K} \in D^{+}(X, \Z/n\Z)$
whose cohomology sheaves are constructible sheaves of $\Z/n\Z$-modules
and are locally unipotent along $\mathfrak{X}$,
the following two assertions hold.
\begin{enumerate}
    \item The sliced nearby cycles complexes for
    $f_{S'} \colon X_{S'} \to S'$ and $\mathcal{K}_{S'}$ are compatible with any base change.
    \item The sliced nearby cycles complexes for $f_{S'} \colon X_{S'} \to S'$ and $\mathcal{K}_{S'}$ are unipotent.
\end{enumerate}
\end{thm}

In fact, as in \cite{Orgogozo06},
we can show a more precise result
for the compatibility of the sliced nearby cycles functors with base change as a corollary of Theorem \ref{Theorem:uniform base change}:

\begin{cor}\label{Corollary:modification}
Under the assumptions of
Theorem \ref{Theorem:uniform base change},
there exists a
\textit{modification} $S' \to S$
such that,
for every positive integer $n$ invertible on $S$ and
every complex
$\mathcal{K} \in D^{+}(X, \Z/n\Z)$
whose cohomology sheaves are constructible sheaves of $\Z/n\Z$-modules
and are locally unipotent along $\mathfrak{X}$,
the sliced nearby cycles complexes for
$f_{S'} \colon X_{S'} \to S'$ and $\mathcal{K}_{S'}$ are compatible with any base change.
\end{cor}
\begin{proof}
This follows from Theorem \ref{Theorem:uniform base change}
together with \cite[Lemme 3.2 and Lemme 3.3]{Orgogozo06}.
\end{proof}

\section{Proof of Theorem \ref{Theorem:uniform base change}}\label{Section:Proof of theorem uniform base change}

\subsection{Nodal curves}\label{Subsection:Nodal curves}

In this subsection, we recall some results on nodal curves from \cite{deJong97, Orgogozo19}.
Let $f \colon X \to S$ be a morphism of Noetherian schemes.
We say that $f$ is a \textit{nodal curve}
if it is a flat projective morphism such that every geometric fiber of $f$ is a connected reduced curve having at most ordinary double points as singularities.
We say that $f$ is a
\textit{nodal curve adapted to a pair} $(X^{\circ}, S^{\circ})$ of dense open subsets
$X^{\circ}$ and $S^{\circ}$ of $X$ and $S$, respectively, if the following conditions are satisfied:
\begin{itemize}
    \item $f$ is a nodal curve which is smooth over $S^{\circ}$.
    \item There is a closed subscheme $D$ of $X$ which is \'etale over $S$
    and is contained in the smooth locus of $f$.
    Moreover we have $f^{-1}(S^{\circ}) \cap (X \backslash D)=X^{\circ}$.
\end{itemize}

The following proposition will be used in the proof of Theorem \ref{Theorem:uniform base change}, which is one of the main reasons why we introduce the notion of locally unipotent sheaves.

\begin{prop}[{Orgogozo \cite[Proposition 2.3.1]{Orgogozo19}}]\label{Proposition:nodal curve case}
Let $S$ be a Noetherian scheme and $f \colon X \to S$ a nodal curve
adapted to a pair $(X^{\circ}, S^{\circ})$ of dense open subsets
$X^{\circ}$ and $S^{\circ}$ of $X$ and $S$, respectively.
Let $u \colon X^{\circ} \hookrightarrow X$ denote the open immersion.
Assume that $S^{\circ}$ is normal.
Then,
for every positive integer $n$ invertible on $S$ and
every locally constant constructible sheaf $\mathcal{L}$ of $\Z/n\Z$-modules on $X^{\circ}$ such that
$u_{!}\mathcal{L}$ is locally unipotent along the stratification $\mathfrak{X}=\{ X^{\circ}, X \backslash X^{\circ} \}$ of $X$,
the sheaf
\[
R^{i}f_*(u_{!}\mathcal{L})
\]
is locally unipotent along the stratification $\mathfrak{S}=\{ S^{\circ}, S \backslash S^{\circ} \}$ of $S$ for every $i$.
\end{prop}
\begin{proof}
See \cite[Proposition 2.3.1]{Orgogozo19}.
\end{proof}

\begin{rem}\label{Remark:nodal curve cases}
The proof of Theorem \ref{Theorem:uniform base change} is inspired by that of Proposition \ref{Proposition:nodal curve case}.
In fact,
we can show that,
with the notation of Proposition \ref{Proposition:nodal curve case},
the nearby cycles for $f$ and $u_! \mathcal{L}$ are compatible with any base change and unipotent.
Since we will not use this fact in the proof of Theorem \ref{Theorem:uniform base change}, we omit the proof of it.
\end{rem}

We say that a morphism
$f \colon X \to S$
of Noetherian integral schemes
is
a \textit{pluri nodal curve adapted to a dense open subset}
$X^{\circ} \subset X$
if there are an integer $d \geq 0$, a sequence
\[
(X=X_d \underset{f_d}{\to} X_{d-1} \to \cdots \to X_1 \underset{f_1}{\to} X_0=S)
\]
of morphisms of Noetherian integral schemes, and dense open subsets $X^{\circ}_i \subset X_i$ for every $0 \leq i \leq d$ with $X^{\circ}_d=X^{\circ}$ such that
$f_i \colon X_i \to X_{i-1}$ is
a nodal curve adapted to the pair $(X^{\circ}_i, X^{\circ}_{i-1})$
for every $ 1 \leq i \leq d$.
If $d=0$, by convention, it means that $X=S$ and $f$ is the identity map.

The following theorem of de Jong plays an important role in the proof of Theorem \ref{Theorem:uniform base change}.

\begin{thm}[{de Jong \cite[Theorem 5.9]{deJong97}}]\label{Theorem:alteration}
Let $f \colon X \to S$ be a proper surjective morphism of Noetherian excellent integral schemes.
Let $X^{\circ} \subset X$ be a dense open subset.
We assume that the geometric generic fiber of $f$ is irreducible.
Then there is the following commutative diagram:
\[
\xymatrix{ X_0 \ar[r]^-{f'} \ar[d]^-{} & S' \ar[d]^-{}  \\
X\ar[r]^-{f} & S,
}
\]
where the vertical maps are integral alterations and
$f'$ is a pluri nodal curve adapted to a dense open subset
$X^{\circ\circ}_0 \subset X_0$ which is contained in the inverse image of 
$X^{\circ} \subset X$.
\end{thm}
\begin{proof}
See \cite[Theorem 5.9]{deJong97} and the proof of \cite[Theorem 5.10]{deJong97}.
We note that if the dimension of the generic fiber of $f$ is zero, then $f$ is an integral alteration.
Hence we can take $S'$ as $X$ and take $f'$ as the identity map on $X$ in this case.
\end{proof}

\subsection{Preliminary lemmas}\label{Subsection:Preliminary lemmas}

We shall give two lemmas,
which will be used in the proof of Theorem \ref{Theorem:uniform base change}.

We will need the following terminology.

\begin{defn}\label{Definition:truncated}
Let $f \colon X \to S$ be a morphism of schemes.
Let $\Lambda$ be a commutative ring and $\mathcal{K} \in D^{+}(X, \Lambda)$ a complex.
Let $\rho$ be an integer.
\begin{enumerate}
    \item We say that
    \textit{the sliced nearby cycles complexes for $f$ and $\mathcal{K}$ are $\rho$-compatible with any base change}
    (or simply that \textit{the nearby cycles for $f$ and $\mathcal{K}$ are $\rho$-compatible with any base change})
    if
    for every local morphism $q \colon V \to U$
    of strictly local schemes over $S$ and every algebraic geometric point
    $\overline{\xi} \to V$ with image $\overline{\eta} \to U$,
    we have
    $\tau_{\leq \rho} \Delta=0$
    for the cone $\Delta$ of the base change map:
    \[
    q^*R\Psi_{f_U, \overline{\eta}}(\mathcal{K}_U) \to R\Psi_{f_V, \overline{\xi}}(\mathcal{K}_V) \to \Delta \to.
    \]
    \item We say that
    \textit{the sliced nearby cycles complexes for $f$ and $\mathcal{K}$ are $\rho$-unipotent}
     (or simply that \textit{the nearby cycles for $f$ and $\mathcal{K}$ are $\rho$-unipotent})
    if for every morphism $q \colon U \to S$ from a strictly local scheme $U$, a point $\eta \in U$, and an algebraic geometric point $\overline{\eta} \to U$ lying above $\eta$,
    the complex
    \[
    \tau_{\leq \rho}R\Psi_{f_{U}, \overline{\eta}}(\mathcal{K}_U)
    \]
    is
    $\Gal(\kappa(\overline{\eta})/\kappa(\eta))$-unipotent.
\end{enumerate}
\end{defn}

\begin{lem}\label{Lemma:reduction, closed subscheme}
Let $f \colon X \to Z$ and $g \colon Z \to S$
be morphisms of schemes.
Let $h:= g \circ f$ denote the composition.
Let $\mathcal{K} \in D^{+}(X, \Z/n\Z)$ be a complex.
\begin{enumerate}
    \item Assume that $g$ is a closed immersion.
    If the nearby cycles for $f$ and $\mathcal{K}$ are $\rho$-compatible with any base change
        (resp.\ are $\rho$-unipotent), then so
        are the nearby cycles for $h$ and $\mathcal{K}$.
    \item Assume that $f$ is a closed immersion.
    If the nearby cycles for $h$ and $\mathcal{K}$ are $\rho$-compatible with any base change
        (resp.\ are $\rho$-unipotent), then so
        are the nearby cycles for $g$ and $f_*\mathcal{K}$.
\end{enumerate}
\end{lem}
\begin{proof}
(1) Let $q \colon U \to S$ be a morphism from a strictly local scheme $U$
and $\overline{\eta} \to U$ an algebraic geometric point.
If the image of $\overline{\eta}$ in $S$ is not contained in $Z$,
then we have
$R\Psi_{h_U, \overline{\eta}}(\mathcal{K}_U)=0$.
If the image of $\overline{\eta}$ in $S$ is contained in $Z$,
then $U':=Z \times_S U$ is a strictly local scheme and
$\overline{\eta}$ induces an algebraic geometric point
$\overline{\eta}' \to U'$.
We have $U'_{(\overline{\eta}')} \cong Z \times_S U_{(\overline{\eta})}$,
and hence
$R\Psi_{h_U, \overline{\eta}}(\mathcal{K}_U) \cong
R\Psi_{f_{U'}, \overline{\eta}'}(\mathcal{K}_{U'})$.
The assertion follows from this description.

(2) Let $q \colon U \to S$ and $\overline{\eta} \to U$ be as above.
Let $u \in U$ be the closed point.
Then we have
$(f_u)_*R\Psi_{h_U, \overline{\eta}}(\mathcal{K}_U) \cong
R\Psi_{g_{U}, \overline{\eta}}((f_*\mathcal{K})_{U})$,
where $f_u \colon X_u \to Z_u$ is the base change of $f$.
Since $(f_u)_*$ is exact, the assertion follows from this isomorphism.
\end{proof}

As in \cite{Orgogozo06},
we need some results on cohomological descent.
See \cite[Expos\'e Vbis]{SGA 4-II} and \cite[Section 5]{Deligne HodgeIII}
for the terminology used here.
Let $f \colon Y \to X$ be a morphism of schemes.
Let
\[
\beta \colon Y_{\bullet}:=\cosq_0(Y/X) \to X
\]
be the augmented simplicial object in the category of schemes
defined as in
\cite[(5.1.4)]{Deligne HodgeIII},
so $Y_m$ is the
$(m+1)$-times fiber product
$Y \times_X \cdots \times_X Y$ for $m \geq 0$.
We can associate
to the \'etale topoi of $Y_m$ ($m \geq 0$)
a topos
$(Y_\bullet)^{\sim}$; see \cite[(5.1.6)--(5.1.8)]{Deligne HodgeIII}.
Moreover,
as in \cite[(5.1.11)]{Deligne HodgeIII},
we have a morphism of topoi
\[
(\beta_*, \beta^*) \colon (Y_\bullet)^{\sim} \to X^{\sim}_\et
\]
from $(Y_\bullet)^{\sim}$ to the \'etale topos $X^{\sim}_\et$ of $X$.

\begin{lem}\label{Lemma:proper descent}
Let $f \colon X \to S$ be a morphism of schemes.
Let $\beta_0 \colon Y \to X$ be a proper surjective morphism.
We put
$\beta \colon Y_{\bullet}:=\cosq_0(Y/X) \to X$.
Let $\mathcal{F}$ be a sheaf of $\Z/n\Z$-modules on $X$ and $\mathcal{F}_m:=\beta_m^*\mathcal{F}$ the pull-back of $\mathcal{F}$ by
$\beta_m \colon Y_m \to X$.
The composition $f \circ \beta_m$ is denoted by $f_m$.
Let $\rho \geq -1$ be an integer.
    \begin{enumerate}
        \item If the nearby cycles for $f_m$ and $\mathcal{F}_m$ are $(\rho-m)$-compatible with any base change
        for every $0 \leq m \leq \rho+1$,
        then the nearby cycles for $f$ and $\mathcal{F}$ are $\rho$-compatible with any base change.
        \item If the nearby cycles for $f_m$ and $\mathcal{F}_m$ are $(\rho-m)$-unipotent for every $0 \leq m \leq \rho$,
        then the nearby cycles for $f$ and $\mathcal{F}$ are $\rho$-unipotent.
    \end{enumerate}
\end{lem}
\begin{proof}
The assertion (1) is \cite[Lemme 4.1]{Orgogozo06}.
(See also Remark \ref{Remark:vanishing topos}.)
Although it is stated for constant sheaves, the same proof works for sheaves of $\Z/n\Z$-modules (or more generally, for torsion abelian sheaves).

The assertion (2) can be proved by the same arguments as in the proof of \cite[Lemme 4.1]{Orgogozo06}.
We shall give a sketch here.
Let $q \colon U \to S$ be a morphism from a strictly local scheme $U$
and $\overline{\eta} \to U$ an algebraic geometric point with image $\eta \in U$.
Let $u \in U$ be the closed point.
We have the following diagram:
\[
\xymatrix{ (Y_{\bullet})_{U_{(\overline{\eta})}} \ar[r]^-{j_\bullet} \ar[d]^-{\beta} & (Y_{\bullet})_U \ar[d]^-{\beta}  & \ar[l]_-{i_\bullet} (Y_{\bullet})_u \ar[d]^-{\beta} \\
X_{U_{(\overline{\eta})}} \ar[r]^-{j} & X_U & \ar[l]_-{i} X_u,
}
\]
where $\beta \colon (Y_{\bullet})_U \to X_U$ is the base change of $\beta$, etc.
By \cite[Expos\'e Vbis, Proposition 4.3.2]{SGA 4-II},
the morphism $\beta_0 \colon Y \to X$ is universally of cohomological descent,
and hence
we have
$\mathcal{F}_U \cong R\beta_*\beta^*\mathcal{F}_U$.
Using this isomorphism and the proper base change theorem,
we obtain
\[
R\Psi_{f_{U}, \overline{\eta}}(\mathcal{F}_{U}) \cong R\beta_*(i_\bullet)^*R(j_\bullet)_*(j_\bullet)^*\beta^*\mathcal{F}_U.
\]
The pull-back
of
the complex
\[
(i_\bullet)^*R(j_\bullet)_*(j_\bullet)^*\beta^*\mathcal{F}_U
\]
to
$(Y_m)_u$
is isomorphic to
$R\Psi_{(f_m)_{U}, \overline{\eta}}((\mathcal{F}_m)_{U})$ for every $m \geq 0$.
Thus we have the following spectral sequence:
\[
E^{k, l}_{1}=R^l (\beta_k)_*R\Psi_{(f_k)_{U}, \overline{\eta}}((\mathcal{F}_k)_{U}) \Rightarrow R^{k+l}\Psi_{f_{U}, \overline{\eta}}(\mathcal{F}_{U}).
\]
(See \cite[(5.2.7.1)]{Deligne HodgeIII}.)
The assertion follows from this spectral sequence since
the sheaf
\[
R^l (\beta_k)_*R\Psi_{(f_k)_{U}, \overline{\eta}}((\mathcal{F}_k)_{U})
\cong R^l (\beta_k)_* \tau_{\leq l}R\Psi_{(f_k)_{U}, \overline{\eta}}((\mathcal{F}_k)_{U})
\]
is $\Gal(\kappa(\overline{\eta})/\kappa(\eta))$-unipotent if $k+l \leq \rho$ by our assumption.
\end{proof}

\subsection{Proof of Theorem \ref{Theorem:uniform base change}}\label{Subsection:proof of theorem uniform base change}

In this subsection, we prove Theorem \ref{Theorem:uniform base change}.
Let us stress that the proof is heavily inspired by the methods of \cite{Orgogozo06, Orgogozo19}.

In this section, we use the following terminology.

\begin{defn}\label{Definition:rho-adapted}
Let $S$ be a Noetherian scheme and
$f \colon X \to S$ a morphism of finite type.
Let $\rho$ be an integer.
\begin{enumerate}
    \item Let $\mathfrak{X}$ be a stratification of $X$.
We say that an alteration $S' \to S$ is \textit{$\rho$-adapted to
the pair $(f, \mathfrak{X})$} if,
for every positive integer $n$ invertible on $S$ and every constructible sheaf $\mathcal{F}$ of $\Z/n\Z$-modules on $X$ which is locally unipotent along $\mathfrak{X}$,
    the nearby cycles for $f_{S'} \colon X_{S'} \to S'$ and $\mathcal{F}_{S'}$ are $\rho$-compatible with any base change
    and $\rho$-unipotent.
    \item Let $u \colon U \hookrightarrow X$ be an open immersion.
    We say that an alteration $S' \to S$ is \textit{$\rho$-adapted to
the pair $(f, U)$} if,
for every positive integer $n$ invertible on $S$ and
every locally constant constructible sheaf $\mathcal{L}$ of $\Z/n\Z$-modules on $U$
such that $u_!\mathcal{L}$ is locally unipotent along the stratification $\{ U, X\backslash U \}$,
the nearby cycles for $f_{S'} \colon X_{S'} \to S'$ and $(u_!\mathcal{L})_{S'}$ are $\rho$-compatible with any base change and $\rho$-unipotent.
\end{enumerate}
\end{defn}

Let $S$ be a Noetherian excellent integral scheme.
Let $\rho$ and $d$ be two integers.
We shall consider the following statement
$\textbf{P}(S, \rho, d)$:
\begin{quote}
    $\textbf{P}(S, \rho, d)$:
    Let $T \to S$ be an integral alteration and $f \colon Y \to T$ a proper morphism such that the dimension of the generic fiber of $f$ is less than or equal to $d$.
    Let $\mathfrak{Y}$ be a stratification of $Y$.
    Then
    there exists an alteration $T' \to T$ which is $\rho$-adapted to
    $(f, \mathfrak{Y})$ in the sense of Definition \ref{Definition:rho-adapted} (1).
\end{quote}

\begin{rem}\label{Remark:rho=-2 and d=-1}\
\begin{enumerate}
    \item $\textbf{P}(S, -2, d)$ holds trivially for every Noetherian excellent integral scheme $S$ and every integer $d$.
    \item For an integral scheme $T$ and a proper morphism
    $f \colon Y \to T$, the condition that the dimension of the generic fiber is less than or equal to $-1$ means that $f$ is not surjective.
    The statement $\textbf{P}(S, \rho, -1)$ is not trivial.
\end{enumerate}
\end{rem}

\begin{lem}\label{Lemma:main results are induced}
To prove Theorem \ref{Theorem:uniform base change},
it is enough to prove that statement $\textbf{P}(S, \rho, d)$ holds for every triple
$(S, \rho, d)$,
where $S$ is a Noetherian excellent integral scheme,
and $\rho$ and $d$ are integers.
\end{lem}

\begin{proof}
Let $S$ be a Noetherian scheme and $f \colon X \to S$ a morphism of finite type.
Let $N$ be the supremum of dimensions of fibers of $f$.
Let $q \colon U \to S$ be a morphism from a strictly local scheme $U$
and $\overline{\eta} \to U$ an algebraic geometric point.
Then,
for every sheaf $\mathcal{F}$ of $\Z/n\Z$-modules on $X$, where $n$ is a positive integer,
we have for $i > 2N$
\[
R^i\Psi_{f_U, \overline{\eta}}(\mathcal{F}_U)=0
\]
by \cite[Proposition 3.1]{Orgogozo06}; see also Remark \ref{Remark:vanishing topos}.
By using this fact,
the assertion can be proved by standard arguments.
\end{proof}

We will prove $\textbf{P}(S, \rho, d)$ by induction on the triples
$(S, \rho, d)$.
For two Noetherian excellent integral schemes $S$ and $S'$,
we denote
\[
S' \prec S
\]
if $S'$ is isomorphic to a proper closed subscheme of an integral alteration of $S$.
For a Noetherian excellent integral scheme $S$ and an integer $\rho$,
we also consider the following statements.
\begin{itemize}
    \item $\textbf{P}(S, \rho, *)$:
    The statement $\textbf{P}(S, \rho, d')$ holds for every integer $d'$.
    \item $\textbf{P}(* \prec S, \rho, *)$:
    The statement $\textbf{P}(S', \rho, d')$ holds for every Noetherian excellent integral scheme $S'$ with $S' \prec S$ and every integer $d'$.
\end{itemize}

We begin with the following lemma.

\begin{lem}\label{Lemma:induction:Gabber's lemma}
Let $S$ be a Noetherian excellent integral scheme and $\rho$ an integer.
If $\textbf{P}(* \prec S, \rho, *)$
holds, then $\textbf{P}(S, \rho, -1)$
holds.
\end{lem}

\begin{proof}
This lemma can be proved by the same arguments as in \cite[Section 4.2]{Orgogozo06} by using \cite[Proposition 1.6.2]{Orgogozo19} instead of \cite[Lemme 4.3]{Orgogozo06}.
We recall the arguments for the reader's convenience.

We assume that $\textbf{P}(* \prec S, \rho, *)$ holds.
Let $T \to S$ be an integral alteration and
let $f \colon Y \to T$ be a proper morphism.
We assume that $f$ is not surjective.
Let $\mathfrak{Y}$ be a stratification of $Y$.
Let $Z:=f(Y)$ be the schematic image of $f$.
We write $g \colon Y \to Z$ for the induced morphism.
By applying $\textbf{P}(* \prec S, \rho, *)$ to each irreducible component of $Z$,
we can find an alteration $Z' \to Z$ which is $\rho$-adapted to $(g, \mathfrak{Y})$.
By \cite[Proposition 1.6.2]{Orgogozo19},
there is an alteration
$T' \to T$
such that every irreducible component of
$Z_{T'}:=Z \times_T T'$
endowed with the reduced closed subscheme structure
has a $Z$-morphism to $Z'$.

Let $\mathcal{F}$ be a constructible sheaf of $\Z/n\Z$-modules on $Y$ which is locally unipotent along $\mathfrak{Y}$,
where $n$ is a positive integer invertible on $T$.
We shall show that the nearby cycles for $f_{T'}$ and $\mathcal{F}_{T'}$ are $\rho$-compatible with any base change and $\rho$-unipotent.
By Lemma \ref{Lemma:reduction, closed subscheme} (1),
it suffices to show
that the same properties hold for
the nearby cycles for
$g_{T'} \colon Y_{T'} \to Z_{T'}$
and $\mathcal{F}_{T'}$.
We put
\[
W_0:= \coprod_{\alpha \in \Theta} (Y_{T'} \times_{Z_{T'}} Z_\alpha),
\]
where $\{  Z_\alpha  \}_{\alpha \in \Theta}$ is the set of the irreducible components of $Z_{T'}$,
and put
\[
\beta \colon W_\bullet:= \cosq_0(W_0/Y_{T'}) \to Y_{T'}.
\]
By the constructions of $Z'$ and $T'$, and by Lemma \ref{Lemma:reduction, closed subscheme} (1),
we see that
the nearby cycles for $g_m$ and $\mathcal{F}_{m}$ are $\rho$-compatible with any base change and $\rho$-unipotent for every $m \geq 0$,
where
we write
$\mathcal{F}_{m}:=\beta^{*}_m\mathcal{F}_{T'}$ and $g_m:=g_{T'} \circ \beta_m$.
By Lemma \ref{Lemma:proper descent}, it follows that
the nearby cycles for
$g_{T'} \colon Y_{T'} \to Z_{T'}$ and $\mathcal{F}_{T'}$
are
$\rho$-compatible with any base change and $\rho$-unipotent.
\end{proof}

Our next task is to show the following lemma.

\begin{lem}\label{Lemma:induction:alteration}
Let
$(S, \rho, d)$
be a triple of a Noetherian excellent integral scheme $S$ and two integers $\rho$ and $d$.
Assume that $d \geq 0$.
If $\textbf{P}(S, \rho, d-1)$, $\textbf{P}(S, \rho-1, *)$, and $\textbf{P}(* \prec S, \rho, *)$ hold, then $\textbf{P}(S, \rho, d)$ holds.
\end{lem}

The proof of Lemma \ref{Lemma:induction:alteration} is divided into two steps.
The first step is to prove the following lemma.

\begin{lem}\label{Lemma:open strata}
We assume that $\textbf{P}(S, \rho, d-1)$, $\textbf{P}(S, \rho-1, *)$, and $\textbf{P}(* \prec S, \rho, *)$ hold.
Under this assumption,
to prove $\textbf{P}(S, \rho, d)$,
it suffices to prove the following statement
$\textbf{P}_{\text{nd}}(S, \rho, d)$:
\begin{quote}
$\textbf{P}_{\text{nd}}(S, \rho, d)$:
Let $T \to S$ be an integral alteration
and $f \colon Y \to T$ a pluri nodal curve adapted to a dense open subset
$Y^{\circ} \subset Y$
such that the dimension of the generic fiber of $f$ is less than or equal to $d$.
Then there is an alteration
$T' \to T$ which is $\rho$-adapted to $(f, Y^{\circ})$ in the sense of Definition \ref{Definition:rho-adapted} (2).
\end{quote}
\end{lem}

\begin{proof}
We assume that
$\textbf{P}_{\text{nd}}(S, \rho, d)$
holds.
Let $T \to S$ be an integral alteration and
$f \colon Y \to T$ a proper morphism such that the dimension of the generic fiber of $f$ is less than or equal to $d$.
Let $\mathfrak{Y}$ be a stratification of $Y$.
We want to prove that
there is an alteration
$T' \to T$ which is $\rho$-adapted to $(f, \mathfrak{Y})$.

\textit{Step 1.}
It suffices to prove the following claim (I):
\begin{enumerate}
    \item[(I)] Let $u \colon Y^{\circ} \hookrightarrow Y$ be an open immersion. 
Then there is an alteration
$T' \to T$ which is $\rho$-adapted to $(f, Y^{\circ})$.
\end{enumerate}

Indeed,
by replacing $\mathfrak{Y}$ by a stratification refining it,
we may assume that $\mathfrak{Y}$ is a good stratification
in the sense of \cite[Section 1.1]{Orgogozo19} (it is called a bonne stratification in French).
Then every sheaf $\mathcal{F}$ of $\Z/n\Z$-modules on $Y$ which is constructible along $\mathfrak{Y}$
has a finite filtration
such that each successive quotient is of the form
$u_!\mathcal{L}$ where
$u \colon Y_\alpha \hookrightarrow Y$ is an immersion for some $Y_\alpha \in \mathfrak{Y}$ and $\mathcal{L}$ is a locally constant constructible sheaf of $\Z/n\Z$-modules on $Y_\alpha$; see \cite[Proposition 1.1.4]{Orgogozo19}.
If furthermore $\mathcal{F}$ is locally unipotent along $\mathfrak{Y}$, then so is every successive quotient of this filtration.
Since $\mathfrak{Y}$ consists of finitely many locally closed subsets,
by using Lemma \ref{Lemma:reduction, closed subscheme} (2),
we see that it suffices to prove the claim (I).

\textit{Step 2.}
To prove the claim (I),
we may assume that $Y$ is integral, the morphism $f$ is surjective, and the geometric generic fiber of $f$ is irreducible.

Indeed,
there is a field $L$ which is a finite extension of the function field of $T$ such that,
every irreducible component of
$Y \times_T \Spec L$
is geometrically irreducible.
Let $T' \to T$ be the normalization of $T$ in $L$.
We put
\[
\beta_0 \colon W_0:= \coprod_{\alpha \in \Theta} Y_\alpha \to Y_{T'},
\]
where $\{  Y_\alpha  \}_{\alpha \in \Theta}$ is the set of the irreducible components of $Y_{T'}$.
By $\textbf{P}(S, \rho-1, *)$ and Lemma \ref{Lemma:proper descent},
it suffices to show that there is
an alteration
$T'' \to T'$
which is $\rho$-adapted to $(f_0, \beta^{-1}_0(Y^{\circ}_{T'}))$,
where $f_0:=f_{T'} \circ \beta_0$.
It is enough to show that
the same assertion holds
after restricting to
each component
$Y_\alpha$.
By $\textbf{P}(* \prec S, \rho, *)$
and Lemma \ref{Lemma:induction:Gabber's lemma},
we may assume that $Y_\alpha \to T'$ is surjective.
Then,
by the construction of $T'$,
the geometric generic fiber of $Y_\alpha \to T'$ is irreducible.
This completes the proof of our claim.

\textit{Step 3.}
We may assume that $Y^{\circ}$ is non-empty.
We claim that we may assume further that
$f \colon Y \to T$ is a pluri nodal curve adapted to a dense open subset
$Y^{\circ\circ} \subset Y$
with $Y^{\circ\circ} \subset Y^{\circ}$.

Indeed, by Theorem \ref{Theorem:alteration},
there is the following commutative diagram:
\[
\xymatrix{ Y_0 \ar[r]^-{f'} \ar[d]^-{\beta} & T' \ar[d]^-{}  \\
Y\ar[r]^-{f} & T,
}
\]
where the vertical maps are integral alterations and
$f'$ is a pluri nodal curve adapted to a dense open subset
$Y^{\circ\circ}_0 \subset Y_0$ which is contained in
$\beta^{-1}(Y^{\circ})$.
The generic fiber of $Y_{T'} \to T'$ is irreducible.
Let $W_0$ be the disjoint union of $Y_0$ and
the irreducible components of $Y_{T'}$ which do not dominate $T'$.
Then the natural morphism
$W_0 \to Y_{T'}$
is a proper surjective morphism.
By the same arguments as in the proof of the previous step,
we see that it suffices to prove that 
there is an alteration
$T'' \to T'$
which is $\rho$-adapted to $(f', \beta^{-1}(Y^{\circ}))$.

\textit{Step 4.}
Finally, we complete the proof of the claim (I).

Let $u' \colon Y^{\circ\circ} \hookrightarrow Y$ and 
$u'' \colon Y^{\circ\circ} \hookrightarrow Y^{\circ}$
denote the open immersions.
Let $n$ be a positive integer invertible on $T$ and
$\mathcal{L}$ a locally constant constructible sheaf of $\Z/n\Z$-modules on $Y^{\circ}$ such that $u_!\mathcal{L}$ is locally unipotent along the stratification $\mathfrak{Y}=\{ Y^{\circ}, Y\backslash Y^{\circ} \}$.
We have an exact sequence
\[
0 \to u'_! u''^{*} \mathcal{L} \to u_!\mathcal{L} \to \mathcal{G} \to 0.
\]
The sheaf $u'_! u''^{*} \mathcal{L}$ is locally unipotent along the stratification $\{ Y^{\circ\circ}, Y\backslash Y^{\circ\circ} \}$.
The sheaf $\mathcal{G}$ is supported on $Y \backslash Y^{\circ\circ}$ and the restriction of $\mathcal{G}$ to
$Y \backslash Y^{\circ\circ}$ is locally unipotent along the stratification $\{ Y^{\circ} \backslash Y^{\circ\circ}, Y\backslash Y^{\circ} \}$.
By applying $\textbf{P}(S, \rho, d-1)$ to $Y\backslash Y^{\circ\circ} \to T$ and the stratification $\{ Y^{\circ} \backslash Y^{\circ\circ}, Y\backslash Y^{\circ} \}$
and using Lemma \ref{Lemma:reduction, closed subscheme} (2),
we see that
$\textbf{P}_{\text{nd}}(S, \rho, d)$ implies the claim (I) by d\'evissage.

The proof of Lemma \ref{Lemma:open strata} is now complete.
\end{proof}

Next, we prove $\textbf{P}_{\text{nd}}(S, \rho, d)$ in Lemma \ref{Lemma:open strata} under the assumptions:

\begin{lem}\label{Lemma:claim nodal case}
We assume that $\textbf{P}(S, \rho, d-1)$, $\textbf{P}(S, \rho-1, *)$, and $\textbf{P}(* \prec S, \rho, *)$ hold.
Then the statement $\textbf{P}_{\text{nd}}(S, \rho, d)$ in Lemma \ref{Lemma:open strata} is true.
\end{lem}

\begin{proof}
Let $u \colon Y^\circ \hookrightarrow Y$ denote the open immersion.
Let $T \to S$ be an integral alteration and $f \colon Y \to T$ a pluri nodal curve adapted to a dense open subset
$Y^{\circ} \subset Y$ such that the dimension of the generic fiber of $f$ is less than or equal to $d$.
If $f$ is an isomorphism, then there is nothing to prove.
Hence we may assume that $f$ is not an isomorphism,
and hence there are a factorization
\[
\xymatrix{
Y \ar[r]^-{h} \ar@/_12pt/[rr]_-{f} & X \ar[r]^-{g} & T
}
\]
and a dense open subset $X^\circ \subset X$ such that 
$h \colon Y \to X$ is a nodal curve adapted to the pair $(Y^{\circ}, X^{\circ})$ and
$g \colon X \to T$ is a pluri nodal curve adapted to $X^\circ$.
Since $\textbf{P}(S, \rho, d-1)$ holds,
we may assume that
the identity map
$T \to T$ is $\rho$-adapted to
the following two pairs
\[
(Y \backslash Y^\circ \to T, \{ Y \backslash Y^\circ \}) \quad \text{and} \quad (g, \{ X^\circ, X \backslash X^\circ  \}).
\]
By replacing $T$ with its normalization, we may assume that
$T$ is normal.

We claim that the identity map $T \to T$
is $\rho$-adapted to $(f, Y^\circ)$.
The proof is divided into two parts.
First, we prove the assertion after restricting to the smooth locus 
$Y' \subset Y$ of $h$.
Then, we prove our claim by using the results on the smooth locus $Y'$.

\begin{claim}\label{Claim:for smooth locus}
Let $a \colon Y' \to T$ denote the restriction of $f$ to $Y'$.
Let $n$ be a positive integer invertible on $T$
and
$\mathcal{L}$
a locally constant constructible sheaf $\mathcal{L}$ of $\Z/n\Z$-modules on $Y^{\circ}$
such that $u_!\mathcal{L}$ is locally unipotent along the stratification $\mathfrak{Y}:=\{ Y^{\circ}, Y\backslash Y^{\circ} \}$.
Let $\mathcal{F}$ be the pull-back of
$u_!\mathcal{L}$ to $Y'$.
Then the following assertions hold:
\begin{enumerate}
    \item The nearby cycles for $a \colon Y' \to T$ and $\mathcal{F}$ are $\rho$-compatible with any base change.
    \item The nearby cycles for $a \colon Y' \to T$ and $\mathcal{F}$ are $\rho$-unipotent.
\end{enumerate}
\end{claim}

\begin{proof}
(1) 
We fix a local morphism
$q \colon V \to U$
of strictly local schemes over $T$
and an algebraic geometric point
$\overline{\xi} \to V$ with image $\overline{\eta} \to U$.
In the following, for a morphism
$\phi \colon Z \to T$
and a complex $\mathcal{K} \in D^+(Z, \Z/n\Z)$,
the cone of the base change map
\[
q^*R\Psi_{\phi_U, \overline{\eta}}(\mathcal{K}_U) \to R\Psi_{\phi_V, \overline{\xi}}(\mathcal{K}_V)
\]
is denoted by $\Delta(\phi, \mathcal{K})$.
For a morphism $\phi \colon Z \to W$ of $T$-schemes
and a $T$-scheme $T'$,
the base change $Z_{T'} \to W_{T'}$
is often denoted by the same letter $\phi$ when there is no possibility of confusion.

We want to show
$\tau_{\leq \rho} \Delta(a, \mathcal{F})=0$.
It suffices to prove that
$\tau_{\leq \rho} \Delta(a, \mathcal{F})_{x}=0$
at every geometric point $x \to Y'_s$, where $s \in V$ is the closed point.
The morphism
\[
(q^*R\Psi_{a_U, \overline{\eta}}(\mathcal{F}_U))_x \to R\Psi_{a_V, \overline{\xi}}(\mathcal{F}_V)_x
\]
on the stalks induced by the base change map can be identified with the pull-back map
\[
R\Gamma((Y'_U)_{(x)}\times_{U} U_{(\overline{\eta})}, u_!\mathcal{L}) \to R\Gamma((Y'_V)_{(x)}\times_{V} V_{(\overline{\xi})}, u_!\mathcal{L}).
\]
(See also (\ref{equation:vanishing cycle stalk}) in Remark \ref{Remark:vanishing topos}.)
Since the sheaf $u_!\mathcal{L}$ is locally unipotent along $\mathfrak{Y}=\{ Y^{\circ}, Y\backslash Y^{\circ} \}$,
we may assume that $\mathcal{L}=\Lambda$ is a constant sheaf on $Y^\circ$
by d\'evissage.

Note that $Y^\circ$ is contained in $Y'$.
Since we have the following exact sequence of sheaves on $Y'$
\[
0 \to u_!\Lambda \to \Lambda \to v_*\Lambda \to 0,
\]
where
$v \colon Y' \backslash Y^\circ \hookrightarrow Y'$
is the closed immersion and
the open immersion
$u \colon Y^\circ \hookrightarrow Y'$ is denoted by the same letter $u$,
it suffices to prove that
$\tau_{\leq \rho}\Delta(a, \Lambda)=0$
and
$\tau_{\leq \rho}\Delta(a, v_*\Lambda)=0$.

It follows from the assumption on $T$ that the nearby cycles for $a \circ v$ and the constant sheaf $\Lambda$ are $\rho$-compatible with any base change.
Hence
we have
$\tau_{\leq \rho}\Delta(a, v_*\Lambda)=0$
by Lemma \ref{Lemma:reduction, closed subscheme} (2).
By the assumption on $T$ again,
the nearby cycles for $g$ and the constant sheaf $\Lambda$ are $\rho$-compatible with any base change.
Since the composition $b \colon Y' \hookrightarrow Y \to X$ is smooth,
we have
$
\Delta(a, \Lambda) \cong b^*\Delta(g, \Lambda)
$
by the smooth base change theorem.
Hence
we obtain that
\[
\tau_{\leq \rho}\Delta(a, \Lambda) \cong \tau_{\leq \rho}b^*\Delta(g, \Lambda) \cong b^*\tau_{\leq \rho}\Delta(g, \Lambda)=0.
\]

(2) Let $q \colon U \to T$
be a morphism from a strictly local scheme $U$, a point $\eta \in U$, and an algebraic geometric point $\overline{\eta} \to U$ lying above $\eta$.
Let $s \in U$ be the closed point.
We want to show that the complex
\[
\tau_{\leq \rho}R\Psi_{a_{U}, \overline{\eta}}(\mathcal{F}_U)
\]
is $\Gal(\kappa(\overline{\eta})/\kappa(\eta))$-unipotent.

We first claim that, for every $i \leq \rho$, the sheaf
$R^i\Psi_{a_{U}, \overline{\eta}}(\mathcal{F}_U)$
is constructible.
Since we have already shown that
the nearby cycles for $a$ and $\mathcal{F}$ are $\rho$-compatible with any base change,
we may assume that $U$ is the strict localization of $T$ at $s \to T$,
in particular, we may assume that $U$ is Noetherian.
Then,
by using \cite[Proposition 7.1.9]{EGA II},
we may assume that $U$ is the spectrum of strictly Henselian discrete valuation ring,
and in this case, the claim follows from \cite[Th.\ finitude, Th\'eor\`eme 3.2]{SGA4 1/2}.
(See also \cite[Section 8]{Orgogozo06}.)

Now, it suffices to prove that, for every geometric point
$x \to Y'_s$,
the complex
\[
\tau_{\leq \rho}R\Psi_{a_U, \overline{\eta}}(\mathcal{F}_U)_x \cong \tau_{\leq \rho}R\Gamma((Y'_U)_{(x)} \times_{U} U_{(\overline{\eta})}, u_!\mathcal{L})
\]
is $\Gal(\kappa(\overline{\eta})/\kappa(\eta))$-unipotent;
see \cite[Lemme 1.2.5]{Orgogozo19}.
Since the sheaf $u_!\mathcal{L}$ is locally unipotent along $\mathfrak{Y}=\{ Y^{\circ}, Y\backslash Y^{\circ} \}$,
we reduce to the case where $\mathcal{L}=\Lambda$ is a constant sheaf on $Y^\circ$
by d\'evissage.

By the exact sequence
$
0 \to u_!\Lambda \to \Lambda \to v_*\Lambda \to 0,
$
it suffices to prove that
the nearby cycles for $a$ and the sheaf $v_*\Lambda$ (resp.\ the constant sheaf $\Lambda$) are $\rho$-unipotent.
By using the assumption on $T$,
we conclude by the same argument as in the proof of (1).
\end{proof}

\begin{claim}\label{Claim:end of the proof}
Let $n$ be a positive integer invertible on $T$
and
$\mathcal{L}$
a locally constant constructible sheaf $\mathcal{L}$ of $\Z/n\Z$-modules on $Y^{\circ}$
such that $\mathcal{F}:=u_!\mathcal{L}$ is locally unipotent along the stratification $\mathfrak{Y}=\{ Y^{\circ}, Y\backslash Y^{\circ} \}$.
Then the following assertions hold:
\begin{enumerate}
    \item The nearby cycles for $f$ and
    $\mathcal{F}$ are $\rho$-compatible with any base change.
    \item The nearby cycles for $f$ and $\mathcal{F}$ are $\rho$-unipotent.
\end{enumerate}
\end{claim}

\begin{proof}
(1)
We fix a local morphism
$q \colon V \to U$
of strictly local schemes over $T$
and an algebraic geometric point
$\overline{\xi} \to V$ with image $\overline{\eta} \to U$.
We retain the notation of the proof of Claim \ref{Claim:for smooth locus} (1).
We write $\Delta:=\Delta(f, \mathcal{F})$.
We want to show
$\tau_{\leq \rho} \Delta =0$.
Let $c \colon Z \hookrightarrow Y$ be a closed immersion
whose complement is the smooth locus $Y'$ of $h$.
By Claim \ref{Claim:for smooth locus} (1),
we have
\[
\tau_{\leq \rho} \Delta \cong c_*c^* \tau_{\leq \rho}\Delta,
\]
and hence
it suffices to show that $c^* \tau_{\leq \rho}\Delta=0$.
Since the composition $d \colon Z \to Y \to X$ is a finite morphism,
it is enough to prove that
\[
d_*c^* \tau_{\leq \rho}\Delta=0.
\]
By using $\tau_{\leq \rho} \Delta \cong c_*c^* \tau_{\leq \rho}\Delta$,
we obtain an isomorphism
$d_*c^* \tau_{\leq \rho}\Delta \cong \tau_{\leq \rho}Rh_*\Delta$.
By the proper base change theorem, we have
$Rh_*\Delta \cong \Delta(g, Rh_*\mathcal{F})$.
Note that $X^\circ$ is normal since $T$ is normal.
Hence the cohomology sheaves of
$Rh_*\mathcal{F}$
are locally unipotent along the
the stratification $\{ X^\circ, X \backslash X^\circ  \}$ by Proposition \ref{Proposition:nodal curve case}.
By the assumption on $T$,
we have
$
\tau_{\leq \rho}\Delta(g, R^{i}h_*\mathcal{F})=0
$
for every $i$.
It follows that
$
\tau_{\leq \rho}\Delta(g, Rh_*\mathcal{F})=0.
$
This completes the proof of (1).

(2) 
Let $q \colon U \to T$
be a morphism from a strictly local scheme $U$, a point $\eta \in U$, and an algebraic geometric point $\overline{\eta} \to U$ lying above $\eta$.
We write
$\mathcal{K}:= R\Psi_{f_U, \overline{\eta}}(\mathcal{F}_U)$.
Let $e \colon Y' \to Y$ denote the open immersion.
We have the following distinguished triangle:
\[
e_!e^*\tau_{\leq \rho}\mathcal{K} \to \tau_{\leq \rho}\mathcal{K}  \to c_*c^*\tau_{\leq \rho}\mathcal{K} \to.
\]
By Claim \ref{Claim:for smooth locus} (2),
it suffices to prove that
$c^*\tau_{\leq \rho}\mathcal{K}$ is $\Gal(\kappa(\overline{\eta})/\kappa(\eta))$-unipotent.
Since $d$ is a finite morphism, it suffices to prove that
\[
d_*c^*\tau_{\leq \rho}\mathcal{K}
\]
is $\Gal(\kappa(\overline{\eta})/\kappa(\eta))$-unipotent.
We have the following distinguished triangle:
\[
Rh_*e_!e^*\tau_{\leq \rho}\mathcal{K} \to Rh_*\tau_{\leq \rho}\mathcal{K}  \to d_*c^*\tau_{\leq \rho}\mathcal{K} \to.
\]
Since
the complex
$Rh_*e_!e^*\tau_{\leq \rho}\mathcal{K}$
is $\Gal(\kappa(\overline{\eta})/\kappa(\eta))$-unipotent
by Claim \ref{Claim:for smooth locus} (2),
it is enough to show that
$
\tau_{\leq \rho}Rh_*\tau_{\leq \rho}\mathcal{K} \cong \tau_{\leq \rho}Rh_*\mathcal{K}
$
is $\Gal(\kappa(\overline{\eta})/\kappa(\eta))$-unipotent.
By the proper base change theorem, we have
\[
Rh_*\mathcal{K} \cong R\Psi_{g_{U}, \overline{\eta}}((Rh_*\mathcal{F})_U).
\]
As above,
by Proposition \ref{Proposition:nodal curve case} and
the assumption on $T$,
it follows that
the complex
$\tau_{\leq \rho}R\Psi_{g_{U}, \overline{\eta}}((Rh_*\mathcal{F})_U)$ is $\Gal(\kappa(\overline{\eta})/\kappa(\eta))$-unipotent, whence (2).
\end{proof}

The proof of Lemma \ref{Lemma:claim nodal case} is complete.
\end{proof}

Now Lemma \ref{Lemma:induction:alteration}
follows from Lemma \ref{Lemma:open strata} and Lemma \ref{Lemma:claim nodal case}.
Finally, we prove the following proposition which completes 
the proof Theorem \ref{Theorem:uniform base change}.

\begin{prop}\label{Proposition:completion of induction}
For every triple $(S, \rho, d)$ of a Noetherian excellent integral scheme $S$ and two integers $\rho$ and $d$, the statement $\textbf{P}(S, \rho, d)$ holds.
\end{prop}

\begin{proof}
We assume that $\textbf{P}(S, \rho, d)$ does not hold.
Then,
by Lemma \ref{Lemma:induction:Gabber's lemma} and Lemma \ref{Lemma:induction:alteration},
we can find infinitely many triples
$\{ (S_n, \rho_n, d_n) \}_{n \in \Z_{\geq 0}}$
with the following properties:
\begin{enumerate}
    \item $\textbf{P}(S_n, \rho_n, d_n)$ does not hold for every $n \in \Z_{\geq 0}$.
    \item $(S_0, \rho_0, d_0)=(S, \rho, d)$.
    \item For every $n \in \Z_{\geq 0}$, we have
    \begin{enumerate}
        \item $S_{n+1} \prec S_n$,
        \item $S_{n+1}=S_{n}$, $\rho_{n+1}=\rho_n-1$, and $d_n \geq 0$, or
        \item $S_{n+1}=S_{n}$, $\rho_{n+1}=\rho_n$, and $d_{n+1}=d_n-1 \geq -1$.
    \end{enumerate}
\end{enumerate}
By \cite[Lemme in 3.4.4]{Orgogozo19},
there is an integer $N \geq 0$ such that $S_{n+1}=S_n$ for every $n \geq N$.
Since $\textbf{P}(S', -2, d')$ holds trivially for every Noetherian excellent integral scheme $S'$ and every integer $d'$,
there is an integer $N' \geq N$ such that 
$d_{n+1}=d_n-1 \geq -1$
for every $n \geq N'$.
This leads to a contradiction.
\end{proof}

For future reference,
we state the following
immediate consequence
of Theorem \ref{Theorem:uniform base change}
as a corollary.

\begin{cor}\label{Corollary:uniform base change for constant sheaves}
Let $S$ be a Noetherian excellent scheme and $f \colon X \to S$ a separated morphism of finite type.
There exists an alteration $S' \to S$ such that,
for every positive integer $n$ invertible on $S$,
the sliced nearby cycles complexes for
$f_{S'} \colon X_{S'} \to S'$ and the constant sheaf $\Z/n\Z$ are compatible with any base change and are unipotent
\end{cor}
\begin{proof}
The morphism $f$ has a factorization $f = g \circ u$ where
$u \colon X \hookrightarrow P$
is an open immersion and
$g \colon P \to S$ is a proper morphism.
Let $q \colon U \to S$ be a morphism from a strictly local scheme $U$
and $\overline{\eta} \to U$ an algebraic geometric point.
Let $u \in U$ be the closed point.
Then
the restriction of
$R\Psi_{g_U, \overline{\eta}}(\Z/n\Z)$
to $X_u$
is isomorphic to $R\Psi_{f_U, \overline{\eta}}(\Z/n\Z)$.
Thus,
by applying Theorem \ref{Theorem:uniform base change} to $g \colon P \to S$
and the stratification $\{ P \}$ of $P$,
we obtain the desired conclusion.
\end{proof}

\section{Tubular neighborhoods and main results}\label{Section:Tubular neighborhoods and main results}

In this section, we will state our main results on \'etale cohomology of tubular neighborhoods.

\subsection{Adic spaces and pseudo-adic spaces}\label{Subsection:Adic spaces and pseudo-adic spaces}

In this paper, we will freely use the theory of adic spaces and pseudo-adic spaces developed by Huber.
Our basis references are \cite{Huber93, Huber94, Huber96}.
We shall recall the definitions very roughly.
We will use the terminology in \cite[Section 1.1]{Huber96}, such as a valuation of a ring, an affinoid ring, a Tate ring, or a strongly Noetherian Tate ring.

An \textit{adic space} is by definition a triple
\[
X=(X, \O_X, \{ v_{x} \}_{x \in X})
\]
where $X$ is a topological space, $\O_X$ is a sheaf of topological rings on the topological space $X$, and $v_x$ is an equivalence class of valuations of the stalk
$\O_{X, x}$ at $x \in X$
which is locally isomorphic to
the
\textit{affinoid adic space}
$\Spa (A, A^+)$
associated with an affinoid ring
$(A, A^+)$;
see \cite[Section 1.1]{Huber96} for details.
In this paper,
unless stated otherwise,
we assume that every adic space is locally isomorphic to
the affinoid adic space
$\Spa (A, A^+)$
associated with an affinoid ring
$(A, A^+)$
such that $A$ is a strongly Noetherian Tate ring.
So we can use the results in \cite{Huber96};
see \cite[(1.1.1)]{Huber96}.
In particular, we only treat analytic adic spaces; see \cite[Section 1.1]{Huber96} for the definition of an analytic adic space.

A \textit{pseudo-adic space} is a pair
\[
(X, S)
\]
where $X$ is an adic space and $S$ is a subset of $X$ satisfying certain conditions; see \cite[Definition 1.10.3]{Huber96}.
If $X$ is an adic space and $S \subset X$ is a locally closed subset,
then
$(X, S)$
is a pseudo-adic space.
Almost all pseudo-adic spaces which appear in this paper are of this form.
A morphism
$f \colon (X, S) \to (X', S')$
of pseudo-adic spaces is by definition a morphism $f \colon X \to X'$ of adic spaces with $f(S) \subset S'$.

We have a functor $X \mapsto (X, X)$ from the category of adic spaces to the category of pseudo-adic spaces.
We often consider an adic space as a pseudo-adic space via this functor.

A typical example of an adic space is the following.
Let $K$ be a non-archimedean field,
i.e.\
it is a topological field whose topology is induced by a valuation
$ \vert \cdot \vert \colon K \to \R_{\geq 0}$
of rank $1$.
We assume that $K$ is complete.
Let $\O=K^{\circ}$ be the valuation ring of $\vert \cdot \vert$.
We call $\O$ the ring of integers of $K$.
Let $\varpi \in K^{\times}$ be an element with $\vert \varpi \vert < 1$.
Let $\mathcal{X}$ be a scheme of finite type over $\O$.
The $\varpi$-adic formal completion of $\mathcal{X}$ is denoted by $\widehat{\mathcal{X}}$ or $\mathcal{X}^{\wedge}$.
Following \cite[Section 1.9]{Huber96},
the Raynaud generic fiber of $\widehat{\mathcal{X}}$
is denoted by $d(\widehat{\mathcal{X}})$, which is an adic space of finite type over $\Spa(K, \O)$.
In particular $d(\widehat{\mathcal{X}})$ is quasi-compact.
For example, we have
\[
d((\Spec \O[T])^{\wedge})=\Spa (K\langle T \rangle, \O \langle T \rangle)=:\B(1).
\]
We often identify $d((\Spec \O[T])^{\wedge})$ with $\B(1)$.
For a morphism
$f \colon \mathcal{Y} \to \mathcal{X}$
of schemes of finite type over $\O$, the induced morphism
$d(\widehat{\mathcal{Y}}) \to d(\widehat{\mathcal{X}})$ is denoted by $d(f)$ (rather than $d(\widehat{f})$).

Important examples of pseudo-adic spaces for us are tubular neighborhoods of adic spaces.
In the next subsection,
we will define them in the case where adic spaces arise from schemes of finite type over $\O$.

\subsection{Tubular neighborhoods}\label{Subsection:Tubular neighborhoods}

Let $X=(X, \O_X, \{ v_{x} \}_{x \in X})$ be an adic space.
Let $U \subset X$ be an open subset and $g \in \O_X(U)$ an element.
Following \cite{Huber96},
for a point $x \in U$,
we write
$\vert g(x) \vert:=v_x(g)$.
(Strictly speaking, we implicitly choose a valuation from the equivalence class $v_x$.)

As in the previous subsection,
let $K$ be a complete non-archimedean field
with ring of integers $\O$.

\begin{prop}\label{Proposition:Tubular neighborhoods}
Let $\mathcal{X}$ be a scheme of finite type over $\O$ and $\mathcal{Z} \hookrightarrow \mathcal{X}$ a closed immersion of finite presentation.
Let $\epsilon \in \vert K^{\times} \vert$ be an element.
\begin{enumerate}
    \item There exist subsets
\[
S(\mathcal{Z}, \epsilon) \subset d(\widehat{\mathcal{X}}) \quad \text{and} \quad T(\mathcal{Z}, \epsilon) \subset d(\widehat{\mathcal{X}})
\]
satisfying the following properties;
for any affine open subset
$\mathcal{U} \subset \mathcal{X}$
and any set
$\{ g_1, \dotsc, g_q \} \subset \O_\mathcal{U}(\mathcal{U})$
of elements
defining the closed subscheme
$\mathcal{Z} \cap \mathcal{U}$ of $\mathcal{U}$,
we have
\begin{align*}
S(\mathcal{Z}, \epsilon) \cap d(\widehat{\mathcal{U}}) &= \{ x \in d(\widehat{\mathcal{U}}) \, \vert \, \vert g_i(x) \vert < \epsilon \  \text{for every $1 \leq i \leq q$} \} \\
&:=\{ x \in d(\widehat{\mathcal{U}}) \, \vert \, \vert g_i(x) \vert < \vert \varpi(x) \vert \  \text{for every $1 \leq i \leq q$} \}
\end{align*}
and
\begin{align*}
T(\mathcal{Z}, \epsilon) \cap d(\widehat{\mathcal{U}}) &= \{ x \in d(\widehat{\mathcal{U}}) \, \vert \, \vert g_i(x) \vert \leq  \epsilon  \  \text{for every $1 \leq i \leq q$} \} \\
&:= \{ x \in d(\widehat{\mathcal{U}}) \, \vert \, \vert g_i(x) \vert \leq \vert \varpi(x) \vert \  \text{for every $1 \leq i \leq q$} \},
\end{align*}
where $\varpi \in K^{\times}$ is an element with $\epsilon = \vert \varpi \vert$ and the element of $\O_{d(\widehat{\mathcal{U}})}(d(\widehat{\mathcal{U}}))$ arising from $g_i$ is denoted by the same letter.
Moreover, they are characterized by the above properties. 
    \item 
    The subset $T(\mathcal{Z}, \epsilon)$ is a quasi-compact open subset of $d(\widehat{\mathcal{X}})$.
    The subset $S(\mathcal{Z}, \epsilon)$ is closed and constructible in $d(\widehat{\mathcal{X}})$.
    (See \cite[(1.1.13)]{Huber96} for the definition of a constructible subset.)
    \item For a morphism $f \colon \mathcal{Y} \to \mathcal{X}$ of finite type, we have
    \[
    S(\mathcal{Z} \times_{\mathcal{X}} \mathcal{Y} , \epsilon)= d(f)^{-1}(S(\mathcal{Z}, \epsilon))
    \quad
    \text{and}
    \quad
    T(\mathcal{Z} \times_{\mathcal{X}} \mathcal{Y} , \epsilon)= d(f)^{-1}(T(\mathcal{Z}, \epsilon)).
    \]
\end{enumerate}
\end{prop}
\begin{proof}
(1)
Let $\varpi \in K^{\times}$ be an element with $\epsilon = \vert \varpi \vert$.
Let
$\mathcal{U} \subset \mathcal{X}$
be an affine open subset.
It suffices to show that the subsets
\[
\{ x \in d(\widehat{\mathcal{U}}) \, \vert \, \vert g_i(x) \vert < \epsilon \ \text{for every $1 \leq i \leq q$} \}
\]
and
\[
\{ x \in d(\widehat{\mathcal{U}}) \, \vert \, \vert g_i(x) \vert \leq  \epsilon  \ \text{for every $1 \leq i \leq q$} \}
\]
are independent of the choice of
a set
$\{ g_1, \dotsc, g_q \} \subset \O_\mathcal{U}(\mathcal{U})$
of elements
defining the closed subscheme
$\mathcal{Z} \cap \mathcal{U}$ of $\mathcal{U}$.
Let $\{ h_1, \dotsc, h_r\} \subset \O_\mathcal{U}(\mathcal{U})$ be another set of such elements.
Then, for every $i$,
we have
\[
g_i = \Sigma_{1 \leq j \leq r} s_{ij} h_j
\]
for some elements $\{ s_{ij} \} \subset \O_\mathcal{U}(\mathcal{U})$.
Since we have $\vert s_{ij}(x) \vert \leq 1$ for every $x \in d(\widehat{\mathcal{U}})$ and every $s_{ij}$,
the assertion follows.

(2)
We may assume that $\mathcal{X}$ is affine.
The subset $T(\mathcal{Z}, \epsilon)$ is a rational subset of the affinoid adic space $d(\widehat{\mathcal{X}})$, and hence it is open and quasi-compact.
The subset $S(\mathcal{Z}, \epsilon)$ is
the complement of
the union of finitely many rational subsets.
It follows that $S(\mathcal{Z}, \epsilon)$ is closed and constructible.

(3) We may assume that $\mathcal{X}$ and $\mathcal{Y}$ are affine.
Then the assertion follows from the descriptions given in (1).
\end{proof}

The subsets
$T(\mathcal{Z}, \epsilon)$
and
$S(\mathcal{Z}, \epsilon)$ in Proposition \ref{Proposition:Tubular neighborhoods} are called an
\textit{open tubular neighborhood}
and
a \textit{closed tubular neighborhood} of $d(\widehat{\mathcal{Z}})$ in $d(\widehat{\mathcal{X}})$, respectively.
For an element $\epsilon \in \vert K^{\times} \vert$, we also consider the following subsets:
\[
Q(\mathcal{Z}, \epsilon):=d(\widehat{\mathcal{X}}) \backslash S(\mathcal{Z}, \epsilon).
\]
This is a quasi-compact open subset of $d(\widehat{\mathcal{X}})$.

For a locally closed subset $S$ of an adic space $X$,
the pseudo-adic space
$(X, S)$
is often denoted by $S$ for simplicity.
For example,
the pseudo-adic spaces
$(d(\widehat{\mathcal{X}}), S(\mathcal{Z}, \epsilon))$
and $(d(\widehat{\mathcal{X}}), T(\mathcal{Z}, \epsilon))$
are denoted by
$S(\mathcal{Z}, \epsilon)$ and $T(\mathcal{Z}, \epsilon)$, respectively.

\begin{rem}\label{Remark:formal scheme}
For a formal scheme $\mathscr{X}$ of finite type over $\Spf \O$ and a closed immersion
$\mathscr{Z} \hookrightarrow \mathscr{X}$ of finite presentation (in the sense of \cite[Chapter I, Definition 2.2.1]{Fujiwara-Kato}),
we can also define
tubular neighborhoods of $d(\mathscr{Z})$ in $d(\mathscr{X})$ in the same way.
However,
we will always work with algebraizable formal schemes of finite type over $\O$ in this paper.
\end{rem}

We end this subsection with the following lemma.

\begin{lem}\label{Lemma:tubular neighborhood small enough}
Let $\mathcal{X}$ be a scheme of finite type over $\O$ and $\mathcal{Z} \hookrightarrow \mathcal{X}$ a closed immersion of finite presentation.
For a constructible subset $W \subset d(\widehat{\mathcal{X}})$ containing $d(\widehat{\mathcal{Z}})$,
there is an element $\epsilon \in \vert K^{\times} \vert$ such that
$T(\mathcal{Z}, \epsilon) \subset W$.
\end{lem}
\begin{proof}
We may assume that $\mathcal{X}$ is affine.
Then the underlying topological space of $d(\widehat{\mathcal{X}})$ is a spectral space.
We have
\[
d(\widehat{\mathcal{Z}})=\bigcap_{\epsilon \in \vert K^{\times} \vert} T(\mathcal{Z}, \epsilon).
\]
Hence the intersection
\[
\bigcap_{\epsilon \in \vert K^{\times} \vert} T(\mathcal{Z}, \epsilon) \cap (d(\widehat{\mathcal{X}}) \backslash W)
\]
is empty.
In the constructible topology,
the subsets
$T(\mathcal{Z}, \epsilon)$
and
$d(\widehat{\mathcal{X}}) \backslash W$
are closed, and
$d(\widehat{\mathcal{X}})$ is quasi-compact.
It follows that there is an element $\epsilon \in \vert K^{\times} \vert$ such that
the intersection $T(\mathcal{Z}, \epsilon) \cap d(\widehat{\mathcal{X}}) \backslash W$ is empty, that is
$T(\mathcal{Z}, \epsilon) \subset W$.
\end{proof}

\subsection{Main results on tubular neighborhoods}\label{Subsection:Main results on tubular neighborhoods}

In this subsection,
let $K$ be an algebraically closed complete non-archimedean field
with ring of integers $\O$.

To state the main results on tubular neighborhoods, we need \'etale cohomology and \'etale cohomology with proper support of pseudo-adic spaces.
See \cite[Section 2.3]{Huber96} for definition of the \'etale site of a pseudo-adic space.
As shown in \cite[Proposition 2.3.7]{Huber96},
for an adic space $X$ and an open subset
$U \subset X$,
the \'etale topos of the adic space $U$ is naturally equivalent to
the \'etale topos of
the pseudo-adic space $(X, U)$.
For a commutative ring $\Lambda$,
let $D^+(X, \Lambda)$
denote the derived category of bounded below complexes of
\'etale sheaves of $\Lambda$-modules on a pseudo-adic space $X$.

Let $f \colon X \to Y$ be a morphism of analytic pseudo-adic spaces.
We assume that $f$ is separated, locally of finite type, and \textit{taut}.
(See \cite[Definition 5.1.2]{Huber96} for the definitions of a taut pseudo-adic space and a taut morphism of pseudo-adic spaces.
For example, if $f$ is separated and quasi-compact, then $f$ is taut.)
For such a morphism $f$,
the direct image functor with proper support
\[
Rf_! \colon D^+(X, \Lambda) \to D^+(Y, \Lambda)
\]
is defined in \cite[Definition 5.4.4]{Huber96},
where $\Lambda$ is a torsion commutative ring.
Moreover, if $Y=\Spa(K, \O)$,
we obtain for a complex $\mathcal{K} \in D^+(X, \Lambda)$ the cohomology group with proper support
\[
H^i_c(X, \mathcal{K}).
\]

\begin{ex}\label{Example:taut subsets}
Let $\mathcal{X}$ be a separated scheme of finite type over $\O$ and $\mathcal{Z} \hookrightarrow \mathcal{X}$ a closed immersion of finite presentation.
 \begin{enumerate}
    \item The adic spaces $d(\widehat{\mathcal{Z}})$ and $d(\widehat{\mathcal{X}})$ are separated and of finite type over $\Spa(K, \O)$.
    The morphism $d(\widehat{\mathcal{X}}) \backslash d(\widehat{\mathcal{Z}}) \to \Spa(K, \O)$
    is separated, locally of finite type, and taut; see \cite[Lemma 5.1.4]{Huber96}.
    \item The pseudo-adic spaces
$S(\mathcal{Z}, \epsilon)$, $T(\mathcal{Z}, \epsilon)$,
and $Q(\mathcal{Z}, \epsilon)$ are separated and of finite type (and hence taut) over $\Spa(K, \O)$.
    \item
    For a subset $S$ of an analytic adic space $X$,
    the interior of $S$ in $X$ is denoted by $S^\circ$.
    The morphism
    $S(\mathcal{Z}, \epsilon)^{\circ}  \to \Spa(K, \O)$
    is separated, locally of finite type, and taut
    \cite[Lemma 1.3 iii)]{Huber98a}.
\end{enumerate}
\end{ex}

Let us recall the following results due to Huber in our setting.

\begin{thm}[{Huber \cite[Theorem 2.5]{Huber98a}, \cite[Theorem 3.6]{Huber98b}}]\label{Theorem:Huber main body}
We assume that $K$ is of characteristic zero.
Let $\mathcal{X}$ be a separated scheme of finite type over $\O$ and $\mathcal{Z} \hookrightarrow \mathcal{X}$ a closed immersion of finite presentation.
Let $n$ be a positive integer invertible in $\O$
and let
$\mathcal{F}$ be a constructible \'etale sheaf of $\Z/n\Z$-modules on $d(\widehat{\mathcal{X}})$ in the sense of \cite[Definition 2.7.2]{Huber96}.
\begin{enumerate}
    \item
    There exists an element
    $\epsilon_0 \in \vert K^{\times} \vert$
    such that, for every
    $\epsilon \in \vert K^{\times} \vert$
    with
    $\epsilon \leq \epsilon_0$,
    the following natural maps are isomorphisms for every $i$:
    \begin{enumerate}
        \item $H^i_c(S(\mathcal{Z}, \epsilon), \mathcal{F}\vert_{S(\mathcal{Z}, \epsilon)}) \overset{\cong}{\to} H^i_c(d(\widehat{\mathcal{Z}}), \mathcal{F}\vert_{d(\widehat{\mathcal{Z}})})$.
        \item $H^i_c(T(\mathcal{Z}, \epsilon), \mathcal{F}) \overset{\cong}{\to} H^i_c(T(\mathcal{Z}, \epsilon_0), \mathcal{F})$.
        \item $H^i_c(Q(\mathcal{Z}, \epsilon), \mathcal{F}) \overset{\cong}{\to} H^i_c(d(\widehat{\mathcal{X}})\backslash d(\widehat{\mathcal{Z}}), \mathcal{F})$.
    \end{enumerate}
    \item 
    We assume further that $\mathcal{F}$ is locally constant.
    Then there exists an element
    $\epsilon_0 \in \vert K^{\times} \vert$
    such that, for every
    $\epsilon \in \vert K^{\times} \vert$
    with
    $\epsilon \leq \epsilon_0$,
    the restriction maps
    \[
    \begin{split}
    H^i(T(\mathcal{Z}, \epsilon), \mathcal{F}) \overset{\cong}{\to} 
& H^i(S(\mathcal{Z}, \epsilon), \mathcal{F}\vert_{S(\mathcal{Z}, \epsilon)}) \\ &\overset{\cong}{\to}
H^i(S(\mathcal{Z}, \epsilon)^{\circ}, \mathcal{F})  \overset{\cong}{\to}
H^i(d(\widehat{\mathcal{Z}}), \mathcal{F}\vert_{d(\widehat{\mathcal{Z}})})
\end{split}
    \]
    on \'etale cohomology groups
    are isomorphisms for every $i$.
\end{enumerate}
\end{thm}
\begin{proof}
See \cite[Theorem 2.5]{Huber98a} for the proof of (1) and a more general result.
(See \cite[Remark 5.5.11]{Huber96} for the constructions of the natural maps.)
See \cite[Theorem 3.6]{Huber98b} for the proof (2) and a more general result.
\end{proof}

\begin{rem}\label{Remark:Huber's result positive char}
For an algebraically closed complete non-archimedean field
$K$ of positive characteristic,
an analogous statement to Theorem \ref{Theorem:Huber main body} (1) is proved in \cite[Corollary 5.8]{Huber07}.
\end{rem}

\begin{rem}\label{Remark:mod p cohomology}
If the residue field of $\O$ is of positive characteristic $p >0$,
the assumption that $n$ is invertible in $\O$ in Theorem \ref{Theorem:Huber main body} is essential.
For example, the \'etale cohomology group
$H^1(\B(1), \Z/p\Z)$
is an infinite dimensional $\Z/p\Z$-vector space; see the computations in \cite[Remark 6.4.2]{Berkovich}.
However
we have
$H^1(\{ 0 \}, \Z/p\Z)=0$
for the origin $0 \in \B(1)$.
\end{rem}

The main objective of this paper is to prove uniform variants of Theorem \ref{Theorem:Huber main body} for constant sheaves.
The main result on \'etale cohomology groups with proper support of tubular neighborhoods is as follows.

\begin{thm}\label{Theorem:tubular neighborhood compact support}
Let $K$ be an algebraically closed complete non-archimedean field
with ring of integers $\O$.
Let $\mathcal{X}$ be a separated scheme of finite type over $\O$ and $\mathcal{Z} \hookrightarrow \mathcal{X}$ a closed immersion of finite presentation.
Then there exists an element
$\epsilon_0 \in \vert K^{\times} \vert$
such that, for every
$\epsilon \in \vert K^{\times} \vert$
with
$\epsilon \leq \epsilon_0$ and for every positive integer $n$ invertible in $\O$,
the following natural maps are isomorphisms for every $i$:
\begin{enumerate}
        \item $H^i_c(S(\mathcal{Z}, \epsilon), \Z/n\Z) \overset{\cong}{\to} H^i_c(d(\widehat{\mathcal{Z}}), \Z/n\Z)$.
        \item $H^i_c(T(\mathcal{Z}, \epsilon), \Z/n\Z) \overset{\cong}{\to} H^i_c(T(\mathcal{Z}, \epsilon_0), \Z/n\Z)$.
        \item $H^i_c(Q(\mathcal{Z}, \epsilon), \Z/n\Z) \overset{\cong}{\to} H^i_c(d(\widehat{\mathcal{X}})\backslash d(\widehat{\mathcal{Z}}), \Z/n\Z)$.
    \end{enumerate}
\end{thm}

The main result on \'etale cohomology groups of tubular neighborhoods is as follows.

\begin{thm}\label{Theorem:tubular neighborhood direct image}
Let $K$ be an algebraically closed complete non-archimedean field
with ring of integers $\O$.
Let $\mathcal{X}$ be a separated scheme of finite type over $\O$ and $\mathcal{Z} \hookrightarrow \mathcal{X}$ a closed immersion of finite presentation.
Then there exists an element
$\epsilon_0 \in \vert K^{\times} \vert$
such that, for every
$\epsilon \in \vert K^{\times} \vert$
with
$\epsilon \leq \epsilon_0$
and for every positive integer $n$ invertible in $\O$,
the restriction maps
\[
H^i(T(\mathcal{Z}, \epsilon), \Z/n\Z) \overset{\cong}{\to} 
H^i(S(\mathcal{Z}, \epsilon), \Z/n\Z) \overset{\cong}{\to}
H^i(S(\mathcal{Z}, \epsilon)^{\circ}, \Z/n\Z)  \overset{\cong}{\to}
H^i(d(\widehat{\mathcal{Z}}), \Z/n\Z)
\]
are isomorphisms for every $i$.
\end{thm}

Theorem \ref{Theorem:main result adic space} follows from Theorem \ref{Theorem:tubular neighborhood direct image}; see also the following remark.

\begin{rem}\label{Remark:finite presentation}
In Theorem \ref{Theorem:tubular neighborhood compact support} and Theorem \ref{Theorem:tubular neighborhood direct image},
the assumption that the closed immersion
$\mathcal{Z} \hookrightarrow \mathcal{X}$
is of finite presentation 
is not important in practice.
Indeed,
if we are only interested in the adic spaces
$d(\widehat{\mathcal{Z}})$ and $d(\widehat{\mathcal{X}})$,
then by replacing $\mathcal{Z}$ with the closed subscheme
$\mathcal{Z}' \hookrightarrow \mathcal{Z}$
defined by the sections killed by a power of a non-zero element of the maximal ideal of $\O$,
we can reduce to the case where $\mathcal{Z}$ is flat over $\O$ without changing $d(\widehat{\mathcal{Z}})$.
Then $\mathcal{Z}$ is of finite presentation over $\O$ by \cite[Premi\`ere partie, Corollaire 3.4.7]{Raynaud-Gruson}, and hence $\mathcal{Z} \hookrightarrow \mathcal{X}$ is also of finite presentation.
\end{rem}

The proofs of Theorem \ref{Theorem:tubular neighborhood compact support} and Theorem \ref{Theorem:tubular neighborhood direct image} will be given in Section \ref{Section:proofs of main theorems}.
In the rest of this section,
we will restate Theorem \ref{Theorem:tubular neighborhood direct image} for proper schemes over $K$.

Let $L \subset K$ be a subfield of $K$ which is a complete non-archimedean field with the induced topology.
Let $\O_L$ be the ring of integers of $L$.
For a scheme $X$ of finite type over $L$,
the adic space associated with $X$ is denoted by
\[
X^{\ad}:= X \times_{\Spec L} \Spa(L, \O_L);
\]
see \cite[Proposition 3.8]{Huber94}.
For an adic space $Y$ locally of finite type over $\Spa(L, \O_L)$,
we denote by 
\[
Y_{K}:=Y \times_{\Spa(L, \O_L)} \Spa(K, \O)
\]
the base change of $Y$ to $\Spa(K, \O)$, which exists by \cite[Proposition 1.2.2]{Huber96}.

\begin{cor}\label{Corollary:proper scheme}
Let $X$ be a proper scheme over $L$ and $Z \hookrightarrow X$ a closed immersion.
We have a closed immersion $Z^{\ad} \hookrightarrow X^{\ad}$ of adic spaces over $\Spa(L, \O_L)$.
Then,
there is a quasi-compact open subset $V$ of $X^{\ad}$
containing
$Z^{\ad}$
such that,
for every positive integer $n$ invertible in $\O$,
the restriction map
\[
H^i(V_K, \Z/n\Z) \to H^i((Z^{\ad})_K, \Z/n\Z)
\]
is an isomorphism for every $i$.
\end{cor}
\begin{proof}
There exist a proper scheme
$\mathcal{X}$
over $\Spec \O_L$
and a closed immersion
$\mathcal{Z} \hookrightarrow \mathcal{X}$
such that
the base change of it to $\Spec L$ is
isomorphic to the closed immersion $Z \hookrightarrow X$
by Nagata's compactification theorem; see \cite[Chapter II, Theorem F.1.1]{Fujiwara-Kato} for example.
As in Remark \ref{Remark:finite presentation},
we may assume that $\mathcal{Z} \hookrightarrow \mathcal{X}$ is of finite presentation.
Let
\[
\overline{\mathcal{X}}:=\mathcal{X} \times_{\Spec \O_L} \Spec \O
\quad \text{and} \quad \overline{\mathcal{Z}}:=\mathcal{Z} \times_{\Spec \O_L} \Spec \O
\]
denote the fiber products.
We have
$
d(\widehat{\overline{\mathcal{Z}}}) \cong d(\widehat{\mathcal{Z}})_K
$
and
$
d(\widehat{\overline{\mathcal{X}}}) \cong d(\widehat{\mathcal{X}})_K.
$
For an element
$\epsilon \in \vert L^{\times} \vert$,
we have
$T(\mathcal{Z}, \epsilon)_K=T(\overline{\mathcal{Z}}, \epsilon)$
in $d(\widehat{\overline{\mathcal{X}}})$.
By \cite[Proposition 1.9.6]{Huber96},
we have
$d(\widehat{\mathcal{Z}}) = Z^{\ad}$
and $d(\widehat{\mathcal{X}}) = X^{\ad}$.
Therefore, the assertion follows from 
Theorem \ref{Theorem:tubular neighborhood direct image}.
\end{proof}

\section{\'Etale cohomology with proper support of adic spaces and nearby cycles}\label{Section:Etale cohomology with compact support of adic spaces and nearby cycles}

In this section,
we study the relation between
the compatibility of the sliced nearby cycles functors with base change and
the bijectivity of specialization maps on stalks of higher direct image sheaves with proper support for adic spaces
by using a comparison theorem of Huber \cite[Theorem 5.7.8]{Huber96}.

\subsection{Analytic adic spaces associated with formal schemes}\label{Analytic adic spaces associated with formal schemes}

In this subsection,
we recall the functor $d(-)$
from a certain category of formal schemes to the category of analytic adic spaces defined in \cite[Section 1.9]{Huber96}.

Following \cite{Huber96},
for a commutative ring $A$ and an element $s \in A$,
let
\[
A(s/s)
\]
denote
the localization $A[1/s]$ equipped with the structure of a Tate ring such that the image $A_0$ of the map
$A \to A[1/s]$
is a ring of definition and
$s A_0$
is an ideal of definition.

We record the following well known results.

\begin{lem}\label{Lemma:quotient is complete}
Let $A$ be a commutative ring endowed with the $\varpi$-adic topology for an element $\varpi \in A$
satisfying the following two properties:
\begin{enumerate}
\renewcommand{\labelenumi}{(\roman{enumi})}
    \item $A$ is $\varpi$-adically complete, i.e.\
    the following natural map is an isomorphism:
    \[
    A \to \widehat{A}:=\plim[n] A/\varpi^nA.
    \]
    \item Let $A\langle X_1, \dotsc, X_n \rangle$ be the $\varpi$-adic completion of
    $A[X_1, \dotsc, X_n]$, called the restricted formal power series ring.
    Then $A\langle X_1, \dotsc, X_n \rangle[1/\varpi]$
    is Noetherian for every $n \geq 0$.
\end{enumerate}
Then the following assertions hold:
\begin{enumerate}
    \item For every ideal $I \subset A\langle X_1, \dotsc, X_n \rangle$,
    the quotient
    $
    A\langle X_1, \dotsc, X_n \rangle/I
    $
    is $\varpi$-adically complete.
    \item Let $B$ be an $A$-algebra
    such that the $\varpi$-adic completion $\widehat{B}$ of $B$ is isomorphic to $A\langle X_1, \dotsc, X_n \rangle$.
    Let $I \subset B$ be an ideal.
    Then,
    the $\varpi$-adic completion $\widehat{B/I}$ of $B/I$ is isomorphic to $\widehat{B}/I\widehat{B}$.
    \item The Tate ring $A(\varpi/\varpi)$ is complete and
    we have for every $n \geq 0$
    \[
    A \langle X_1, \dotsc, X_n \rangle[1/\varpi]\cong A(\varpi/\varpi)\langle X_1, \dotsc, X_n \rangle.
    \]
    Here
    $A(\varpi/\varpi)\langle Y_1, \dotsc, Y_m \rangle$ is the ring defined in \cite[Section 1.1]{Huber96} for the Tate ring $A(\varpi/\varpi)$.
    In particular, the Tate ring $A(\varpi/\varpi)$ is strongly Noetherian.
\end{enumerate}
\end{lem}
\begin{proof}
See \cite[Chapter 0, Proposition 8.4.4]{Fujiwara-Kato} for (1).
The rest of the proposition is an immediate consequence of (1).
We will sketch the proof for the reader's convenience.

(2) By (1), the ring $\widehat{B}/I\widehat{B}$ is $\varpi$-adically complete.
Hence we have
\[
    \widehat{B/I}=\plim[n] (B/I)/\varpi^n
    \cong
    \plim[n] (\widehat{B}/I\widehat{B})/\varpi^n
    \cong
    \widehat{B}/I\widehat{B}.
\]

(3) Let $A_0$ be the image of the map $A \to A[1/\varpi]$.
By (1), the ring $A_0$ is $\varpi$-adically complete, and hence $A(\varpi/\varpi)$ is complete.
It is clear from the definitions that
\[
A_0 \langle Y_1, \dotsc, Y_m \rangle[1/\varpi] \cong A(\varpi/\varpi)\langle Y_1, \dotsc, Y_m \rangle.
\]
Let $N$ be the kernel of the surjection
$B:=A[X_1, \dotsc, X_n] \to A_0[X_1, \dotsc, X_n]$.
By using (2), we have the following exact sequence:
\[
N\otimes_{B}  \widehat{B} \to \widehat{B} \to
A_0 \langle Y_1, \dotsc, Y_m \rangle \to 0.
\]
Since $N[1/\varpi]=0$, we have
$(N\otimes_{B}  \widehat{B})[1/\varpi]=0$, and hence
\[
\widehat{B}[1/\varpi] \cong
A_0 \langle Y_1, \dotsc, Y_m \rangle[1/\varpi].
\]
This completes the proof of (3).
\end{proof}

Let $\mathcal{C}$ be
the category whose objects are formal schemes which are locally isomorphic to
$\Spf A$
for an adic ring $A$ with an ideal of definition $\varpi A$ such that
the pair $(A, \varpi)$ satisfies the conditions in Lemma \ref{Lemma:quotient is complete}.
The morphisms in $\mathcal{C}$ are adic morphisms.
A formal scheme in $\mathcal{C}$ satisfies the condition (S) in \cite[Section 1.9]{Huber96} by Lemma \ref{Lemma:quotient is complete} (3).
In \cite[Proposition 1.9.1]{Huber96}, Huber defined a functor 
\[
d(-)
\]
from $\mathcal{C}$ to the category of analytic adic spaces.
For a formal scheme $\mathscr{X}$ in $\mathcal{C}$,
the adic space
$d(\mathscr{X})$
is equipped with a morphism of ringed spaces
\[
\lambda \colon d(\mathscr{X}) \to \mathscr{X}.
\]
This map is called a specialization map.
If $A$ and $\varpi \in A$ satisfy the conditions in Lemma \ref{Lemma:quotient is complete},
then we have
\[
d(\Spf A)=\Spa (A(\varpi/\varpi), A^{+})
\]
where $A^+$ is the integral closure of $A$ in $A(\varpi/\varpi)=A[1/\varpi]$.
The map $\lambda \colon d(\Spf A) \to \Spf A$ sends
$x \in d(\Spf A)$ to the prime ideal
$
\{ a \in A \, \vert \, \vert a (x) \vert < 1 \} \subset A.
$
If $f \colon \mathscr{X} \to \mathscr{Y}$
is an adic morphism of formal schemes in $\mathcal{C}$,
then the induced morphism
$d(f) \colon d(\mathscr{X}) \to d(\mathscr{Y})$
fits into the following commutative diagram:
\[
\xymatrix{ d(\mathscr{X}) \ar[r]^-{\lambda} \ar[d]^-{d(f)} &  \mathscr{X} \ar[d]^-{f}   \\
d(\mathscr{Y}) \ar[r]^-{\lambda} & \mathscr{Y}.
}
\]

For the sake of completeness,
we include a proof of the following result on the compatibility of the functor $d(-)$ with fiber products.

\begin{prop}\label{Proposition:base change compatibility}
Let $f \colon \mathscr{X} \to \mathscr{Y}$
be a morphism locally of finite type of formal schemes in $\mathcal{C}$.
Let $\mathscr{Z} \to \mathscr{Y}$ be
an adic morphism of formal schemes in $\mathcal{C}$.
Then the morphism
\[
d(\mathscr{X} \times_\mathscr{Y} \mathscr{Z}) \to
d(\mathscr{X}) \times_{d(\mathscr{Y})} d(\mathscr{Z})
\]
induced by the universal property of the fiber product is an isomorphism.
\end{prop}
\begin{proof}
First, we note that
the fiber product 
$
d(\mathscr{X}) \times_{d(\mathscr{Y})} d(\mathscr{Z})
$
exists by \cite[Proposition 1.2.2]{Huber96}
since $d(f) \colon d(\mathscr{X}) \to d(\mathscr{Y})$ is locally of finite type.

We may assume that
$\mathscr{X}=\Spf A$,
$\mathscr{Y}= \Spf B$, and 
$\mathscr{Z}= \Spf C$ are affine,
where $B$ and $C$ satisfy the conditions in Lemma \ref{Lemma:quotient is complete} for some element $\varpi \in B$ and for its image in $C$, respectively.
We may assume further that $A$ is of the form $B\langle X_1, \dotsc, X_n \rangle/I$.
We write
$
D:=A \otimes_B C$.
The source of the morphism in question
is isomorphic to
\[
\Spa ((\widehat{D})(\varpi/\varpi), E^{+})
\]
where
$\widehat{D}$
is the $\varpi$-adic completion of
$D$
and $E^{+}$ is the integral closure of
$\widehat{D}$ in
$(\widehat{D})[1/\varpi]$.
On the other hand,
the target of the morphism in question
is isomorphic to
\[
\Spa (D(\varpi/\varpi), F^+)
\]
where $F^+$ is the integral closure of $D$ in $D[1/\varpi]$.
Let $D_0$ be the image of the map $D \to D[1/\varpi]$.
Clearly,
the completion of $D(\varpi/\varpi)$ is isomorphic to
$(\widehat{D}_0)(\varpi/\varpi)$.
By a similar argument as in the proof of Lemma \ref{Lemma:quotient is complete} (3),
we have
$(\widehat{D})(\varpi/\varpi) \cong (\widehat{D}_0)(\varpi/\varpi)$.
This completes the proof of the proposition since the adic spaces associated with an affinoid ring and its completion are naturally isomorphic (see \cite[Lemma 1.5]{Huber94}).
\end{proof}

A valuation ring $R$ is called a
\textit{microbial valuation ring}
if the field of fractions $L$ of $R$ admits a topologically nilpotent unit $\varpi$ with respect to the valuation topology; see \cite[Definition 1.1.4]{Huber96}.
We equip $R$ with the valuation topology unless explicitly mentioned otherwise.
In this case, the element $\varpi$ is contained in $R$, the ideal $\varpi R$ is an ideal of definition of $R$, and we have $L=R[1/\varpi]$.
The completion
$
\widehat{R}
$
of $R$ is also a microbial valuation ring.

Let $R$ be a complete microbial valuation ring.
It is well known that
\[
R\langle X_1, \dotsc, X_n \rangle[1/\varpi] \cong
L \langle X_1, \dotsc, X_n \rangle
\]
is Noetherian for every $n \geq 0$; see \cite[5.2.6, Theorem 1]{BGR}.
A formal scheme $\mathscr{X}$ locally of finite type over $\Spf R$ is in the category $\mathcal{C}$.

\subsection{\'Etale cohomology with proper support of adic spaces and nearby cycles}\label{Subsection:Etale cohomology with compact support of adic spaces and nearby cycles}

We shall recall a comparison theorem of Huber.
To formulate his result,
we need some preparations.

Let $R$ be a microbial valuation ring with field of fractions $L$.
We assume that $R$ is a strictly Henselian local ring.
Let $\varpi$ be a topologically nilpotent unit in $L$.
Let $\overline{L}$ be a separable closure of $L$ and
let
$\overline{R}$ be the valuation ring of $\overline{L}$ which extends $R$.

We will use the following notation.
For a scheme 
$\mathcal{X}$ over $R$,
we write
\[
\overline{\mathcal{X}}:= \mathcal{X} \times_{\Spec R} \Spec \overline{R} \quad \text{and} \quad
\mathcal{X}':= \mathcal{X} \times_{\Spec R} \Spec R/\varpi R.
\]
Let $\eta \in \Spec R$ and $\overline{\eta} \in \Spec \overline{R}$ be the generic points.
We define
\[
\mathcal{X}_{\eta}:=\mathcal{X} \times_{\Spec R} \eta
\quad \text{and} \quad
\mathcal{X}_{\overline{\eta}}:=\mathcal{X} \times_{\Spec R} \overline{\eta}.
\]
The $\varpi$-adic formal completion of a scheme (or a ring) $\mathcal{X}$ over $R$ is denoted by $\widehat{\mathcal{X}}$.
Let $s \in \Spec R$
be the closed point
and $\mathcal{X}_s$ the special fiber of $\mathcal{X}$.
We will use the same notation for morphisms of schemes over $R$ when there is no possibility of confusion.

We write
\[
S:=\Spa(L, R)= d(\Spf \widehat{R}) \quad \text{and} \quad \overline{S}:=\Spa(\overline{L}, \overline{R}) = d(\Spf \widehat{\overline{R}}).
\]
Let $t \in S$ and $\overline{t} \in \overline{S}$ be the closed points corresponding to the valuation rings $R$ and $\overline{R}$,
respectively.
The pseudo-adic space
$
(\overline{S}, \{ \overline{t} \})
$
is also denoted by $\overline{t}$.
The natural morphism
$\xi \colon \overline{t} \to S$
is a geometric point with support $t \in S$
in the sense of \cite[Definition 2.5.1]{Huber96}.

Let
$f \colon \mathcal{X} \to \Spec R$
be a separated morphism of finite type of schemes.
The induced morphism
\[
d(f) \colon d(\widehat{\mathcal{X}}) \to S
\]
is separated and of finite type;
the separatedness can be checked for example by using Proposition \ref{Proposition:base change compatibility}.
(We often write $d(f)$ instead of $d(\widehat{f})$.)

There is a natural morphism
$d(\widehat{\mathcal{X}}) \to \mathcal{X}_{\eta}$
of locally ringed spaces; see \cite[(1.9.4)]{Huber96}.
An \'etale morphism $Y \to \mathcal{X}_{\eta}$
defines an adic space
$d(\widehat{\mathcal{X}}) \times_{\mathcal{X}_{\eta}} Y$,
which is \'etale over $d(\widehat{\mathcal{X}})$; see \cite[Proposition 3.8]{Huber94} and \cite[Corollary 1.7.3 i)]{Huber96}.
In this way, we get a morphism of \'etale sites
\[
a \colon d(\widehat{\mathcal{X}})_\et \to (\mathcal{X}_{\eta})_\et.
\]

Let $\Lambda$ be a torsion commutative ring.
Let $\mathcal{F}$ be an \'etale sheaf of $\Lambda$-modules on $\mathcal{X}$.
Let
$
\mathcal{F}^a
$
denote
the pull-back of $\mathcal{F}$ by the composition
\[
d(\widehat{\mathcal{X}})_\et \overset{a}{\to} (\mathcal{X}_{\eta})_\et \to \mathcal{X}_\et.
\]
Recall that we have the direct image functor with proper support
\[
Rd(f)_! \colon D^+(d(\widehat{\mathcal{X}}), \Lambda) \to D^+(S, \Lambda)
\]
for $d(f)$ by \cite[Definition 5.4.4]{Huber96}.
We define
$
R^nd(f)_!\mathcal{F}^a:=H^n(Rd(f)_!\mathcal{F}^a).
$
We will describe the stalk
\[
(R^nd(f)_!\mathcal{F}^a)_{\overline{t}}:=\Gamma(\overline{t}, \xi^*R^nd(f)_!\mathcal{F}^a)
\]
at the geometric point $\xi \colon  \overline{t} \to S$ in terms of the sliced nearby cycles functor relative to $f$.
Recall that we defined the sliced nearby cycles functor
\[
R\Psi_{f, \overline{\eta}}:=i^{*}Rj_*j^* \colon D^{+}(\mathcal{X}, \Lambda) \to D^{+}(\mathcal{X}_s, \Lambda)
\]
in Section \ref{Section:Nearby cycles over general bases}.
Here we fix the notation by the following commutative diagram:
\[
\xymatrix{ \mathcal{X}_{\overline{\eta}} \ar[r]^-{j} \ar[d]^-{} & \mathcal{X} \ar[d]^-{}  & \ar[l]_-{i} \mathcal{X}_s \ar[d]^-{} \\
\overline{\eta} \ar[r]_-{} & \Spec R & \ar[l]^-{} s.
}
\]
Now we can state the following result due to Huber:

\begin{thm}[Huber {\cite[Theorem 5.7.8]{Huber96}}]\label{Theorem:Huber comparison}
There is an isomorphism
\[
(R^nd(f)_!\mathcal{F}^a)_{\overline{t}}
\cong
H^n_c(\mathcal{X}_s, R\Psi_{f, \overline{\eta}}(\mathcal{F}))
\]
for every $n$.
This isomorphism is compatible with the natural actions of $\Gal(\overline{L}/L)$ on both sides.
\end{thm}
\begin{proof}
We recall the construction of the isomorphism
since it is a key ingredient in this paper and the construction will make the compatibility of it with specialization maps clear; see the proof of Proposition \ref{Proposition:overconvergent}.

First, we recall the following fact.
Let $\overline{\mathcal{F}}$
be the pull-back of $\mathcal{F}$ to $\overline{\mathcal{X}}$.
By
\cite[(1) in the proof of Proposition 4.2.4]{Huber96},
we see that 
the base change map
\begin{equation}\label{equation:Huber 4.2.4}
    R\Psi_{f, \overline{\eta}}(\mathcal{F}) \to
R\Psi_{\overline{f}, \overline{\eta}}(\overline{\mathcal{F}})
\end{equation}
is an isomorphism,
where
$\overline{f} \colon \overline{\mathcal{X}} \to \Spec \overline{R}$ is the base change of $f$.

There is a factorization $f=g \circ u$
where
$u \colon \mathcal{X} \hookrightarrow \mathcal{P}$ is an open immersion and
$g \colon \mathcal{P} \to \Spec R$
is a proper morphism by Nagata's compactification theorem.
By using the valuation criterion \cite[Corollary 1.3.9]{Huber96},
we see that the morphism
$d(g) \colon d(\widehat{\mathcal{P}}) \to S$
is proper; see also the proof of \cite[Lemma 3.5]{Mieda06}.

Let $q \colon \Spec \overline{R} \to \Spec R$
denote the morphism induced by the inclusion $R \subset \overline{R}$.
The base change of it will be denoted by the same letter $q$
when there is no possibility of confusion.
We have
the following Cartesian diagram:
\[
\xymatrix{ d(\widehat{\overline{\mathcal{X}}}) \ar[r]^-{d(\overline{u})} \ar[d]^-{d(q)} &  d(\widehat{\overline{\mathcal{P}}}) \ar[d]^-{d(q)}   \\
d(\widehat{\mathcal{X}}) \ar[r]^-{d(u)} & d(\widehat{\mathcal{P}}).
}
\]
By \cite[Proposition 2.5.13 i) and Proposition 2.6.1]{Huber96}, we have
\begin{equation}\label{equation: stalk direct image}
    (R^nd(f)_!\mathcal{F}^a)_{\overline{t}}  \cong (R^nd(g)_*d(u)_!\mathcal{F}^a)_{\overline{t}} 
    \cong H^n(d(\widehat{\overline{\mathcal{P}}}), d(q)^*d(u)_!\mathcal{F}^a).
\end{equation}
(See \cite[Section 2.7]{Huber96} for the functor $d(u)_!$.)
By \cite[Proposition 5.2.2 iv)]{Huber96}, we have
$d(q)^*d(u)_!\mathcal{F}^a \cong d(\overline{u})_!d(q)^*\mathcal{F}^a$.
The sheaf
$d(q)^*\mathcal{F}^a$
is isomorphic to the pull-back of
$\overline{\mathcal{F}}$
by the composition
\[
d(\widehat{\overline{\mathcal{X}}})_\et \overset{a}{\to} (\mathcal{X}_{\overline{\eta}})_\et  \to (\overline{\mathcal{X}})_\et.
\]
Therefore,
in view of (\ref{equation:Huber 4.2.4}),
we reduce to the case where $R=\overline{R}$.
The construction given below shows that the desired isomorphism
is $\Gal(\overline{L}/L)$-equivariant.

Let
$\lambda \colon d(\widehat{\mathcal{X}})_\et \to (\mathcal{X}')_\et$
denote the morphism of sites defined by sending
an \'etale morphism
$h \colon Y \to \mathcal{X}'$
to $d(\widetilde{Y}) \to d(\widehat{\mathcal{X}})$
where
$\widetilde{Y} \to \widehat{\mathcal{X}}$
is
an \'etale morphism of formal schemes lifting $h$; see \cite[Lemma 3.5.1]{Huber96}.
Similarly, we have a morphism
$\lambda \colon d(\widehat{\mathcal{P}})_\et \to (\mathcal{P}')_\et$ of sites.
By applying \cite[Corollary 3.5.11 ii)]{Huber96} to the following diagram
\[
\xymatrix{ d(\widehat{\mathcal{X}})_\et \ar[r]^-{d(u)} \ar[d]^-{\lambda} & d(\widehat{\mathcal{P}})_\et \ar[d]^-{\lambda}   \\
(\mathcal{X}')_\et \ar[r]^-{u'} & (\mathcal{P}')_\et,}
\]
we have
$
R\lambda_*d(u)_!\mathcal{F}^a \cong u'_!R\lambda_*\mathcal{F}^a.
$
Moreover, by applying \cite[Theorem 3.5.13]{Huber96} to the following diagram
\[
\xymatrix{ d(\widehat{\mathcal{X}})_\et \ar[r]^-{a} \ar[d]^-{\lambda} & (\mathcal{X}_{\eta})_\et \ar[d]^-{j}   \\
(\mathcal{X}')_\et \ar[r]^-{i'} & \mathcal{X}_\et,}
\]
we have an isomorphism
$
R\lambda_*\mathcal{F}^a \cong i'^*Rj_*j^*\mathcal{F}.
$
So we have 
\begin{align*}
 H^n(d(\widehat{\mathcal{P}}), d(u)_!\mathcal{F}^a) 
    & \cong H^n(\mathcal{P}', R\lambda_*d(u)_!\mathcal{F}^a) \\
    & \cong
    H^n(\mathcal{P}', u'_!R\lambda_*\mathcal{F}^a) \\
    & \cong
    H^n(\mathcal{P}', u'_!i'^*Rj_*j^*\mathcal{F}).
\end{align*}
Together with (\ref{equation: stalk direct image}),
we obtain the following isomorphism
\begin{equation}\label{equation:main formula}
    (R^nd(f)_!\mathcal{F}^a)_{\overline{t}}
\cong H^n(\mathcal{P}', u'_!i'^*Rj_*j^*\mathcal{F}).
\end{equation}
The proper base change theorem for schemes implies that
\[
H^n(\mathcal{P}', u'_!i'^*Rj_*j^*\mathcal{F})
\cong H^n_c(\mathcal{X}_s, R\Psi_{f, \eta}(\mathcal{F})).
\]
This isomorphism completes the construction of the desired isomorphism.
\end{proof}

\subsection{Specialization maps on the stalks of $Rd(f)_!$}\label{Subsection:Specialization maps for higher direct image sheaves with proper support}

In this subsection, we work over
a complete non-archimedean field $K$ with ring of integers $\O$ for simplicity.
We fix a topologically nilpotent unit $\varpi$ in $K$.

For an adic space $X$ over $\Spa(K, \O)$,
we will use the following notation.
For a point $x \in X$,
let $k(x)$ be the residue field  of the local ring $\O_{X, x}$
and $k(x)^+$ the valuation ring corresponding to the valuation $v_x$.
We note that $k(x)^+$ is a microbial valuation ring and
the image of $\varpi$, also denoted by $\varpi$, is a topologically nilpotent unit in $k(x)$.
For a geometric point
$\xi$ of $X$,
let $\Supp(\xi) \in X$ denote the support of it.

We recall strict localizations of analytic adic spaces.
Let
$
\xi \colon s \to X
$
be a geometric point.
The strict localization
\[
X(\xi)
\]
of $X$ at $\xi$ is defined in \cite[Section 2.5.11]{Huber96}.
It is an adic space over $X$
with an $X$-morphism
$s \to X(\xi)$.
We write $x:=\Supp(\xi)$.
By \cite[Proposition 2.5.13]{Huber96},
the strict localization $X(\xi)$ is isomorphic to
\[
\Spa(\overline{k}(x), \overline{k}(x)^+)
\]
over $X$,
where $\overline{k}(x)$ is a separable closure of $k(x)$ and $\overline{k}(x)^+$ is a valuation ring extending $k(x)^+$.
A specialization morphism
$\xi_1 \to \xi_2$
of geometric points of $X$ is by definition a morphism
$X(\xi_1) \to X(\xi_2)$
over $X$,
and such a morphism exists if and only if we have
\[
\Supp(\xi_2) \in \overline{\{ \Supp(\xi_1) \}}.
\]
Let $\mathcal{F}$ be an abelian \'etale sheaf on $X$.
A specialization morphism
$\xi_1 \to \xi_2$
of geometric points of $X$
induces a mapping
\[
\mathcal{F}_{\xi_2} \to \mathcal{F}_{\xi_1}
\]
on the stalks in the usual way; see \cite[(2.5.16)]{Huber96}.

\begin{defn}\label{Definition:overconvergent}
Let $X$ be an adic space over $\Spa(K, \O)$
(or more generally an analytic adic space).
Let $\mathcal{F}$ be an abelian \'etale sheaf on $X$.
For a subset $W \subset X$,
we say that
$\mathcal{F}$
is \textit{overconvergent} on $W$
if,
for every specialization morphism
$\xi_1 \to \xi_2$
of geometric points of $X$
whose supports are contained in
$W$,
the induced map
$
\mathcal{F}_{\xi_2} \to \mathcal{F}_{\xi_1}
$
is bijective.
\end{defn}

Let $\mathcal{X}$ be a scheme of finite type over $\Spec \O$.
We write
$
\mathcal{X}_{K}:=\mathcal{X} \times_{\Spec \O} \Spec K.
$
For an \'etale sheaf $\mathcal{F}$ on $\mathcal{X}$,
let $
\mathcal{F}^a
$
denote the pull-back of $\mathcal{F}$ by the composition
\[
d(\widehat{\mathcal{X}})_\et \overset{a}{\to} (\mathcal{X}_{K})_\et \to \mathcal{X}_\et.
\]
Let $\Lambda$ be a torsion commutative ring.

\begin{prop}\label{Proposition:overconvergent}
Let $\mathcal{X}$ and $\mathcal{Y}$ be separated schemes of finite type over $\O$.
Let
$f \colon \mathcal{X} \to \mathcal{Y}$
be a morphism over $\O$.
Let
$\mathcal{F}$
be
an \'etale sheaf of $\Lambda$-modules on $\mathcal{X}$.
We assume that
the sliced nearby cycles complexes for $f$ and $\mathcal{F}$ are compatible with any base change; see Definition \ref{Definition:base change, unipotent} (1).
Let
$s \in \widehat{\mathcal{Y}}$
be a point.
We consider
the inverse image $\lambda^{-1}(s)$ under the specialization map
\[
\lambda \colon d(\widehat{\mathcal{Y}}) \to \widehat{\mathcal{Y}}.
\]
Then,
the sheaf $R^nd(f)_!\mathcal{F}^a$ is overconvergent on $\lambda^{-1}(s)$
for every $n$.
\end{prop}
\begin{proof}
Let
$\xi_1 \to \xi_2$
be a specialization morphism
of geometric points of $d(\widehat{\mathcal{Y}})$
whose supports are contained in $\lambda^{-1}(s)$.
We write $y_m:= \Supp(\xi_m)$ ($m=1, 2$).
Let $\overline{k}(y_m)$ be a separable closure of $k(y_m)$ and let $\overline{k}(y_m)^+$ be a valuation ring extending $k(y_m)^+$.
We identify
$d(\widehat{\mathcal{Y}})(\xi_m)$
with $\Spa(\overline{k}(y_m), \overline{k}(y_m)^+)$.
Let $R_m$ be the completion of $\overline{k}(y_m)^+$
and we put $U_m:= \Spec R_m$.
The morphism
$\Spa(\overline{k}(y_m), \overline{k}(y_m)^+) \to d(\widehat{\mathcal{Y}})$
induces a natural morphism
\[
q_m \colon U_m \to \mathcal{Y}
\]
over $\Spec \O$
and
the specialization morphism
$d(\widehat{\mathcal{Y}})(\xi_1) \to d(\widehat{\mathcal{Y}})(\xi_2)$
induces a natural $\mathcal{Y}$-morphism
\[
r \colon U_1 \to U_2.
\]
By the assumption,
we have
$
q_m(s_m)=s
$
for the closed point $s_m \in U_m$,
where the image of $s \in \widehat{\mathcal{Y}}$ in $\mathcal{Y}$ is denoted by the same letter.
Let $\overline{s} \to \mathcal{Y}$ be an algebraic geometric point lying above $s$ and
let
$U=\Spec R$ be the strict localization of $\mathcal{Y}$ at $\overline{s}$.
There are local $\mathcal{Y}$-morphisms
$\widetilde{q}_m  \colon U_m \to U$ $(m=1, 2)$ such that
the following diagram commutes:
\[
\xymatrix{ U_1 \ar[rd]_-{\widetilde{q}_1} \ar[rr]^-{r} & & U_2 \ar[ld]^-{\widetilde{q}_2}  \\
& U.&
}
\]
We remark that $r$ is not a local morphism if $y_1 \neq y_2$.
Let $\eta_m$ be the generic point of $U_m$.
Then we have $r(\eta_1)=\eta_2$.
We write
$\eta:=\widetilde{q}_1(\eta_1)=\widetilde{q}_2(\eta_2)$.
Let
$\overline{\eta} \to U$
denote the algebraic geometric point
which is the image of $\eta_2$.
We fix the notation by the following commutative diagrams:
\[
\xymatrix{ \mathcal{X} \times_{\mathcal{Y}} \eta_m \ar[r]^-{j_m} \ar[d]^-{\widetilde{q}_m} & \mathcal{X} \times_{\mathcal{Y}} U_m \ar[d]^-{\widetilde{q}_m}  & \ar[l]_-{i'_m}  \ar[d]^-{\widetilde{q}_m} \mathcal{X} \times_{\mathcal{Y}} \Spec R_m/\varpi R_m &  \ar[l]^-{} \ar[d]^-{\widetilde{q}_m} \mathcal{X} \times_{\mathcal{Y}} s_m \\
\mathcal{X} \times_{\mathcal{Y}} U_{(\overline{\eta})}  \ar[r]^-{j} & \mathcal{X} \times_{\mathcal{Y}} U & \ar[l]_-{i'} \mathcal{X} \times_{\mathcal{Y}} \Spec R/\varpi R & \ar[l]^-{} \mathcal{X} \times_{\mathcal{Y}} \overline{s},
}
\]
\[
\xymatrix{ \mathcal{X} \times_{\mathcal{Y}} \eta_1 \ar[r]^-{j_1} \ar[d]^-{r} & \mathcal{X} \times_{\mathcal{Y}} U_1 \ar[d]^-{r}  & \ar[l]_-{i'_1}  \ar[d]^-{r} \mathcal{X} \times_{\mathcal{Y}} \Spec R_1/\varpi R_1 \\
\mathcal{X} \times_{\mathcal{Y}} \eta_2  \ar[r]^-{j_2} & \mathcal{X} \times_{\mathcal{Y}} U_2 & \ar[l]_-{i'_2} \mathcal{X} \times_{\mathcal{Y}} \Spec R_2/\varpi R_2.
}
\]

There is a factorization $f=g \circ u$
where $u \colon \mathcal{X} \hookrightarrow \mathcal{P}$ is an open immersion and
$g \colon \mathcal{P} \to \mathcal{Y}$
is a proper morphism by Nagata's compactification theorem.
Let
\[
u'_m \colon \mathcal{X} \times_{\mathcal{Y}} \Spec R_m/\varpi R_m \hookrightarrow \mathcal{P}'_m:=\mathcal{P} \times_{\mathcal{Y}} \Spec R_m/\varpi R_m
\]
and
\[
u' \colon \mathcal{X} \times_{\mathcal{Y}} \Spec R/\varpi R \hookrightarrow \mathcal{P}':=\mathcal{P} \times_{\mathcal{Y}} \Spec R/\varpi R
\]
be the morphisms induced by $u$.
Let $\mathcal{F}_{m}$ be the pull-back of $\mathcal{F}$ to $\mathcal{X} \times_{\mathcal{Y}} U_m$ and
$\mathcal{F}_U$ the pull-back of $\mathcal{F}$ to $\mathcal{X} \times_{\mathcal{Y}} U$.
For $m=1, 2$,
by the isomorphism (\ref{equation:main formula}) in the proof of Theorem \ref{Theorem:Huber comparison},
we have
\[
 (R^nd(f)_!\mathcal{F}^a)_{\xi_m}
\cong H^n(\mathcal{P}'_m, (u'_m)_! i'^*_m R{j_m}_* j^*_m \mathcal{F}_m).
\]
Via these isomorphisms,
the map $(R^nd(f)_!\mathcal{F}^a)_{\xi_2} \to (R^nd(f)_!\mathcal{F}^a)_{\xi_1}$
can be identified with the
composition 
\begin{align*}
 \phi  \colon  H^n(\mathcal{P}'_2, (u'_2)_! i'^*_2 R{j_2}_* j^*_2 \mathcal{F}_2) & \to  H^n(\mathcal{P}'_1, (u'_1)_! r^* i'^*_2 R{j_2}_* j^*_2 \mathcal{F}_2) \\
& \to H^n(\mathcal{P}'_1, (u'_1)_! i'^*_1 R{j_1}_* j^*_1 \mathcal{F}_1).
\end{align*}
We have the following commutative diagram:
\[
\xymatrix{ H^n(\mathcal{P}'_2, (u'_2)_! i'^*_2 R{j_2}_* j^*_2 \mathcal{F}_2)  \ar[rr]^-{\phi} & & H^n(\mathcal{P}'_1, (u'_1)_! i'^*_1 R{j_1}_* j^*_1 \mathcal{F}_1)   \\
& \ar[lu]^-{\phi_2} H^n(\mathcal{P}', u'_! i'^* Rj_* j^* \mathcal{F}_U). \ar[ru]_-{\phi_1}&
}
\]
By the proper base change theorem,
the map $\phi_m$ is identified with the map
\[
H^n_c(\mathcal{X} \times_{\mathcal{Y}} \overline{s}, R\Psi_{f_U, \overline{\eta}}(\mathcal{F}_U)) \to H^n_c(\mathcal{X} \times_{\mathcal{Y}} s_m, R\Psi_{f_{U_m}, \eta_m} (\mathcal{F}_m)).
\]
For both $m=1$ and $m=2$,
this map is an isomorphism by our assumption that
the sliced nearby cycles complexes for $f$ and $\mathcal{F}$ are compatible with any base change.
Thus $\phi$ is also an isomorphism.
The proof of the proposition is complete.
\end{proof}

\section{Local constancy of higher direct images with proper support}
\label{Section:Local constancy of higher direct images with proper support for generically smooth morphisms}

In this section, we study local constancy of higher direct images with proper support for generically smooth morphisms of adic spaces whose target is one-dimensional.
We will formulate and prove the results not only for constant sheaves,
but also for non-constant sheaves satisfying certain conditions related to the sliced nearby cycles functors.

Throughout this section,
we fix an algebraically closed complete non-archimedean field $K$ with ring of integers $\O$.

\subsection{Tame sheaves on annuli}\label{Subsection:Tame sheaves on annuli}

In this subsection, we recall two theorems on finite \'etale coverings on annuli and the punctured disc,
which are essentially proved in
\cite{Lutkebohmert93, Ramero05, LS05}.
We do not impose any conditions on the characteristic of $K$.
Since we can not directly apply some results there and
some results are only stated in the case where the base field is of characteristic zero,
we give proofs of the theorems in Appendix \ref{Appendix:finite etale coverings of annuli}.

To state the two theorems,
we need some preparations.
Recall that
we defined
$\B(1)=\Spa(K \langle T \rangle, \O \langle T \rangle)$.
Let
\[
\B(1)^{*}:=\B(1) \backslash \{ 0 \}
\]
be the \textit{punctured disc},
where $0 \in \B(1)$ is
the $K$-rational point corresponding to $0 \in K$.
It is an adic space locally of finite type over
$\Spa (K, \O)$.
We fix a valuation
$ \vert \cdot \vert \colon K \to \R_{\geq 0}$
of rank $1$ such that the topology of $K$ is induced by it.
For elements $a, b \in \vert K^{\times} \vert$ with
$a \leq b \leq 1$,
we define
\begin{align*}
\B(a, b)&:= \{ x \in \B(1) \, \vert \, a \leq  \vert T(x) \vert \leq b  \} \\
&:= \{ x \in \B(1) \, \vert \, \vert \varpi_a(x) \vert \leq  \vert T(x) \vert \leq \vert \varpi_b(x) \vert \},
\end{align*}
which is called an \textit{open annulus}.
Here $\varpi_a, \varpi_b \in K^\times$ are elements such that
$a=\vert \varpi_a \vert$ and $b=\vert \varpi_b \vert$.
It is a rational subset of $\B(1)$,
and hence it is an affinoid open subspace of $\B(1)$.

Let $m$ be a positive integer invertible in $K$.
The finite \'etale morphism
$
\varphi_m \colon \B(1)^{*} \to \B(1)^{*}
$
defined by $T \mapsto T^m$ is called a \textit{Kummer covering} of degree $m$.
For elements $a, b \in \vert K^{\times} \vert$ with
$a \leq b \leq 1$,
the restriction
\[
\varphi_m \colon \B(a^{1/m}, b^{1/m}) \to \B(a, b)
\]
of $\varphi_m$ is also called a Kummer covering of degree $m$.
(We also call a morphism of affinoid adic spaces of finite type over $\Spa(K, \O)$ a Kummer covering if it is isomorphic to
$\varphi_m \colon \B(a^{1/m}, b^{1/m}) \to \B(a, b)$
for $a, b \in \vert K^{\times} \vert$ and some $m$ with $a \leq b \leq 1$.)

In this paper, we use the following notion of tameness for
\'etale sheaves on
one-dimensional smooth adic spaces over $\Spa(K, \O)$.

\begin{defn}\label{Definition:tame sheaf}
Let $X$ be a one-dimensional smooth adic space over $\Spa(K, \O)$.
Let $x \in X$ be a
point which has a proper generalization in $X$,
i.e.\ there exists a point $x' \in X$ with
$x \in \overline{\{ x' \}}$ and $x \neq x'$.
Let
\[
k(x)^{\wedge h+}
\]
be the Henselization of the completion of the valuation ring $k(x)^+$ of $x$.
Let $L(x)$ be a separable closure of the field of fractions $k(x)^{\wedge h}$ of $k(x)^{\wedge h+}$.
It induces a geometric point
$\overline{x} \to X$
with support $x$.
For an \'etale sheaf $\mathcal{F}$ on $X$,
we say that 
$\mathcal{F}$ is \textit{tame} at $x \in X$
if the action of
\[
\Gal(L(x)/k(x)^{\wedge h})
\]
on the stalk
$
\mathcal{F}_{\overline{x}}
$
at the geometric point $\overline{x}$
factors through a finite group $G$ such that $\sharp \, G$ is invertible in $\O$,
where $\sharp \, G$ denotes the cardinality of $G$.
\end{defn}

Now we can formulate the results.

\begin{thm}[{\cite[Theorem 2.2]{Lutkebohmert93}, \cite[Theorem 4.11]{LS05}, \cite[Theorem 2.4.3]{Ramero05}}]\label{Theorem:split into annuli}
Let $f \colon X \to \B(1)^{*}$ be a finite \'etale morphism of adic spaces.
There exists an element $\epsilon \in \vert K^\times \vert$ with $\epsilon \leq 1$ such that,
for all $a, b \in \vert K^{\times} \vert$
with
$a < b \leq \epsilon$,
we have
\[
f^{-1}(\B(a, b)) \cong \coprod^{n}_{i=1} \B(c_i, d_i)
\]
for some elements $c_i, d_i \in \vert K^{\times} \vert$
with
$c_i < d_i \leq 1$ ($1 \leq i \leq n$).
If $K$ is of characteristic zero,
then we can take such an element $\epsilon \in \vert K^\times \vert$ so that
the restriction
\[
\B(c_i, d_i)  \to \B(a, b)
\]
of $f$ to every component $\B(c_i, d_i)$ appearing in the above decomposition is a Kummer covering.
\end{thm}

\begin{thm}\label{Theorem:trivialization of tame sheaf}
Let $a, b \in \vert K^{\times} \vert$ be elements with
$a < b \leq 1$.
Let $\mathcal{F}$ be a locally constant \'etale sheaf with finite stalks on
$\B(a, b)$.
We assume that
the sheaf 
$\mathcal{F}$
is tame at every
$x \in \B(a, b)$ having a proper generalization in $\B(a, b)$, in the sense of Definition \ref{Definition:tame sheaf}.
Let $t \in \vert K^\times \vert$ be an element with $a/b < t^2 < 1$.
Then
the restriction $\mathcal{F} \vert_{\B(a/t, tb)}$ of $\mathcal{F}$ to $\B(a/t, tb)$
is trivialized by
a Kummer covering $\varphi_m$ of degree $m$, i.e.\
the pull-back
\[
\varphi^*_m(\mathcal{F} \vert_{\B(a/t, tb)})
\]
is a constant sheaf.
Moreover, we can assume that the degree $m$ is invertible in $\O$.
\end{thm}

We prove Theorem \ref{Theorem:split into annuli} and Theorem \ref{Theorem:trivialization of tame sheaf} in Appendix \ref{Appendix:finite etale coverings of annuli}.

\begin{rem}\label{Remark:Riemann existence}
If $K$ is of characteristic zero, then Theorem \ref{Theorem:split into annuli} is known as the $p$-adic Riemann existence theorem of L\"utkebohmert \cite{Lutkebohmert93}.
\end{rem}

\subsection{Local constancy of $R^id(f)_!$ for generic smooth morphisms}\label{Subsection:Local constancy of higher direct images with proper support}

As in Section \ref{Section:Etale cohomology with compact support of adic spaces and nearby cycles},
we use the following notation.
Let $\mathcal{X}$ be a scheme of finite type over $\O$.
We write
$
\mathcal{X}_{K}:=\mathcal{X} \times_{\Spec \O} \Spec K.
$
For an \'etale sheaf $\mathcal{F}$ on $\mathcal{X}$,
we denote by
$
\mathcal{F}^a
$
the pull-back of $\mathcal{F}$ by the composition
$
d(\widehat{\mathcal{X}})_\et \overset{a}{\to} (\mathcal{X}_{K})_\et \to \mathcal{X}_\et.
$
(See Section \ref{Subsection:Etale cohomology with compact support of adic spaces and nearby cycles} for the morphism
$a \colon d(\widehat{\mathcal{X}})_\et \to (\mathcal{X}_{K})_\et$.)

Let us introduce the following slightly technical definition.

\begin{defn}\label{Definition:uniform finiteness adapted}
We consider the following diagram:
\[
\xymatrix{  & \mathcal{Z} \ar[d]^-{\pi}  & \\
\mathcal{X} \ar[r]^-{f} & \mathcal{Y},
}
\]
where
\begin{itemize}
    \item $\mathcal{Y}=\Spec A$ is an integral affine scheme of finite type over $\O$
such that
$\mathcal{Y}_K$ is one-dimensional and smooth over $K$,
    \item $f \colon \mathcal{X} \to \mathcal{Y}$ is a separated morphism of finite type, and
    \item $\pi \colon \mathcal{Z} \to \mathcal{Y}$ is a proper surjective morphism
    such that $\mathcal{Z}$ is an integral scheme whose generic fiber $\mathcal{Z}_K$ is smooth over $K$,
    and the base change $\pi_K \colon \mathcal{Z}_K \to \mathcal{Y}_K$ is a finite morphism.
\end{itemize}
Let $n$ be a positive integer invertible in $\O$
and $\mathcal{F}$ a sheaf of $\Z/n\Z$-modules on $\mathcal{X}$.
We say that
$\mathcal{F}$
is \textit{adapted} to the pair $(f, \pi)$
if the following conditions are satisfied:
\begin{enumerate}
    \item The \'etale  sheaf $\mathcal{F}^a$ of $\Z/n\Z$-modules on $d(\widehat{\mathcal{X}})$ is constructible in the sense of \cite[Definition 2.7.2]{Huber96}.
    \item 
    The sliced nearby cycles complexes for $f_{\mathcal{Z}}$ and $\mathcal{F}_{\mathcal{Z}}$ are compatible with any base change.
    \item The sliced nearby cycles complexes for $f_{\mathcal{Z}}$ and $\mathcal{F}_{\mathcal{Z}}$ are unipotent.
\end{enumerate}
\end{defn}
See Definition \ref{Definition:base change, unipotent} for
the terminology used in the conditions (2) and (3).
Here we retain the notation of Section \ref{Section:Nearby cycles over general bases}.
For example $f_{\mathcal{Z}}$ denotes the base change
$
f_{\mathcal{Z}} \colon \mathcal{X} \times_\mathcal{Y} \mathcal{Z} \to \mathcal{Z}
$
of $f$ and $\mathcal{F}_{\mathcal{Z}}$
denotes the pull-back of the sheaf $\mathcal{F}$ to $\mathcal{X} \times_\mathcal{Y} \mathcal{Z}$.

For the proofs of Theorem \ref{Theorem:tubular neighborhood compact support} and Theorem \ref{Theorem:tubular neighborhood direct image},
we need the following proposition, which is a consequence of Theorem \ref{Theorem:uniform base change}.

\begin{prop}\label{Proposition:constant sheaves adapted}
Let $\mathcal{Y}=\Spec A$ be an integral affine scheme of finite type over $\O$
such that
$\mathcal{Y}_K$ is one-dimensional and smooth over $K$.
Let $f \colon \mathcal{X} \to \mathcal{Y}$ be a separated morphism of finite presentation.
Then, there exists
a proper surjective morphism
$\pi \colon \mathcal{Z} \to \mathcal{Y}$
as in Definition \ref{Definition:uniform finiteness adapted}
such that,
for every positive integer $n$ invertible in $\O$,
the constant sheaf
$\Z/n\Z$ on $\mathcal{X}$ is adapted to $(f, \pi)$.
\end{prop}

\begin{proof}
Let $p \geq 0$ be the characteristic of the residue field of $\O$.
Let $\Z_{(p)}$ be the localization of $\Z$ at the prime ideal $(p)$.
We may find a finitely generated $\Z_{(p)}$-subalgebra $A_0$ of $A$
and a separated morphism $f_0 \colon \mathcal{X}_0 \to \Spec A_0$ of finite type such that
the base change
$\mathcal{X}_0 \times_{\Spec A_0} \mathcal{Y}$
is isomorphic to $\mathcal{X}$ over $\mathcal{Y}$.
By applying Corollary \ref{Corollary:uniform base change for constant sheaves} to $f_0$,
we find an alteration
$\pi_0 \colon \mathcal{Z}_0 \to \Spec A_0$
satisfying the properties stated there.
The base change
\[
\pi' \colon \mathcal{Z}':=\mathcal{Z}_0 \times_{\Spec A_0} \mathcal{Y} \to \mathcal{Y}
\]
to $\mathcal{Y}$ is generically finite, proper, and surjective.
By restricting $\pi'$ to an irreducible component $\mathcal{Z}''$ of
$\mathcal{Z}'$ dominating
$\mathcal{Y}$,
we obtain a morphism
$\pi'' \colon \mathcal{Z}'' \to \mathcal{Y}$.
The scheme
$\mathcal{Z}''_K$
is one-dimensional since $\pi''$ is generically finite.
It follows that
$\pi''_K \colon \mathcal{Z}''_K \to \mathcal{Y}_K$
is finite.
Let
$h \colon Z \to \mathcal{Z}''_K$
be the normalization of $\mathcal{Z}''_K$.
There exists a proper surjective morphism
$\mathcal{Z} \to \mathcal{Z}''$ such that
$\mathcal{Z}$ is integral and the base change
$\mathcal{Z}_K \to \mathcal{Z}''_K$
is isomorphic to $h$.
Then
we define $\pi$ as the composition
\[
\pi \colon \mathcal{Z} \to \mathcal{Z}'' \to \mathcal{Y}.
\]
By the construction,
the constant sheaf $\Z/n\Z$ is adapted to $(f, \pi)$ for every positive integer $n$ invertible in $\O$.
\end{proof}

By using the results in Section \ref{Section:Etale cohomology with compact support of adic spaces and nearby cycles},
we prove the following proposition:

\begin{prop}\label{Proposition:overconvergent and unipotent}
Let $f \colon \mathcal{X} \to \mathcal{Y}$
and $\pi \colon \mathcal{Z} \to \mathcal{Y}$
be morphisms as in Definition \ref{Definition:uniform finiteness adapted}.
We have the following diagram:
\[
\xymatrix{  & d(\widehat{\mathcal{Z}}) \ar[d]^-{d(\pi)}  & \\
d(\widehat{\mathcal{X}}) \ar[r]^-{d(f)} & d(\widehat{\mathcal{Y}}).
}
\]
Then the following assertions hold:
\begin{enumerate}
    \item  Let $\mathcal{F}$ be an \'etale sheaf of $\Z/n\Z$-modules on $\mathcal{X}$
adapted to $(f, \pi)$ with $n \in \O^{\times}$.
For every $i$, the sheaf
\[
d(\pi)^*R^id(f)_!\mathcal{F}^a
\]
on $d(\widehat{\mathcal{Z}})$ is tame at every
$z \in d(\widehat{\mathcal{Z}})$ having a proper generalization in $d(\widehat{\mathcal{Z}})$, in the sense of Definition \ref{Definition:tame sheaf}.
    \item Let $y \in d(\widehat{\mathcal{Y}})$ be a $K$-rational point.
    There exists an open subset
    $
    V \subset d(\widehat{\mathcal{Y}})
    $
    containing $y$ such that,
    for every \'etale sheaf $\mathcal{F}$ of $\Z/n\Z$-modules on $\mathcal{X}$ adapted to $(f, \pi)$ with $n \in \O^{\times}$,
    the sheaf
    $
    d(\pi)^*R^id(f)_!\mathcal{F}^a
    $
    is overconvergent on $d(\pi)^{-1}(V)$ for every $i$.
\end{enumerate}
\end{prop}

\begin{proof}
For an \'etale sheaf $\mathcal{F}$ of $\Z/n\Z$-modules on
$\mathcal{X}$ with $n \in \O^\times$,
the pull-back of
$\mathcal{F}^a$ by $d(\widehat{\mathcal{X}_{\mathcal{Z}}}) \to d(\widehat{\mathcal{X}})$
is isomorphic to
$(\mathcal{F}_{\mathcal{Z}})^a$,
and hence,
by using the base change theorem \cite[Theorem 5.4.6]{Huber96} for $Rd(f)_!$,
we have
\[
d(\pi)^*R^id(f)_!\mathcal{F}^a \cong 
R^id(f_\mathcal{Z})_!(\mathcal{F}_{\mathcal{Z}})^a.
\]

(1)
This is an immediate consequence of Theorem \ref{Theorem:Huber comparison}.
Indeed,
let
$z \in d(\widehat{\mathcal{Z}})$ be an element having a proper generalization in $d(\widehat{\mathcal{Z}})$.
Let
\[
R:=k(z)^{\wedge h+}
\]
be the Henselization of the completion of the valuation ring $k(z)^+$ of $z$.
By \cite[Corollary 5.4]{Huber01},
the residue field of $R$ is algebraically closed.
We write $U:= \Spec R$.
Let $L$ be the field of fractions of $R$.
The composite
\[
\Spa(L, R) \to \Spa(k(z), k(z)^+) \to d(\widehat{\mathcal{Z}})
\]
is induced by a natural morphism $q \colon U \to \mathcal{Z}$ of schemes over $\O$.
Let $\overline{L}$ be a separable closure of $L$,
which
induces a geometric point
$\overline{t} \to \Spa(L, R)$
and a geometric point
$\overline{z} \to d(\widehat{\mathcal{Z}})$
in the usual way.

Let $\mathcal{F}$ be an \'etale sheaf of $\Z/n\Z$-modules on $\mathcal{X}$
adapted to $(f, \pi)$ with $n \in \O^{\times}$.
By applying Theorem \ref{Theorem:Huber comparison} to
$f_U \colon \mathcal{X}_U \to U$,
we have $\Gal(\overline{L}/L)$-equivariant isomorphisms
\[
(R^id(f_\mathcal{Z})_!(\mathcal{F}_{\mathcal{Z}})^a)_{\overline{z}}
\cong
(R^id(f_U)_!(\mathcal{F}_{U})^a)_{\overline{t}} \cong
H^i_c((\mathcal{X}_U)_s, R\Psi_{f_U, \overline{\eta}}(\mathcal{F}_{U})),
\]
where $s \in  U$ is the closed point and
$\overline{\eta}=\Spec \overline{L} \to U$ is the algebraic geometric point.
By \cite[Corollary 5.4.8 and Proposition 6.2.1 i)]{Huber96},
the left hand side is a finitely generated $\Z/n\Z$-module.
Moreover,
the action of $\Gal(\overline{L}/L)$ on it factors through a finite group.
Since the complex
$
R\Psi_{f_{U}, \overline{\eta}}(\mathcal{F}_{U})
$
is $\Gal(\overline{L}/L)$-unipotent
and the integer $n$ is invertible in $\O$,
it follows that
the action of
$\Gal(\overline{L}/L)$
on
the right hand side
factors through a finite group $G$ such that $\sharp \, G$ is invertible in $\O$.
This proves (1).

(2)
Since $\pi \colon \mathcal{Z} \to \mathcal{Y}$ is proper,
by \cite[Proposition 1.9.6]{Huber96},
we have
$
d(\widehat{\mathcal{Z}}) \cong d(\widehat{\mathcal{Y}}) \times_{\mathcal{Y}} \mathcal{Z},
$
where
$d(\widehat{\mathcal{Y}}) \times_{\mathcal{Y}} \mathcal{Z}$
is the adic space over
$d(\widehat{\mathcal{Y}})$ associated with
$
d(\widehat{\mathcal{Y}}) \to \mathcal{Y}_K \to \mathcal{Y}
$
and
$\pi \colon \mathcal{Z} \to \mathcal{Y}$; see \cite[Proposition 3.8]{Huber94}.
Since
$\pi_K \colon \mathcal{Z}_K \to \mathcal{Y}_K$
is a finite morphism,
it follows that
\[
d(\widehat{\mathcal{Z}}) \cong d(\widehat{\mathcal{Y}}) \times_{\mathcal{Y}} \mathcal{Z} \cong d(\widehat{\mathcal{Y}}) \times_{\mathcal{Y}_K} \mathcal{Z}_K
\]
is finite over $d(\widehat{\mathcal{Y}})$.
The inverse image
\[
d(\pi)^{-1}(y)=\{ z_1, \dotsc, z_m \}
\]
of $y \in d(\widehat{\mathcal{Y}})$ consists of finitely many $K$-rational points.
Let
\[
\lambda \colon d(\widehat{\mathcal{Z}}) \to \widehat{\mathcal{Z}}
\]
be the specialization map associated with the formal scheme $\widehat{\mathcal{Z}}$.
Since
the inverse image
$\lambda^{-1}(\lambda(z_j))$
of
$\lambda(z_j)$
is a closed constructible subset of
$d(\widehat{\mathcal{Z}})$
and $z_j$ is a $K$-rational point,
there exists an open neighborhood
$
V_j \subset d(\widehat{\mathcal{Z}})
$
of $z_j$
with
\[
V_j \subset \lambda^{-1}(\lambda(z_j))
\]
for every $j$; see Lemma \ref{Lemma:tubular neighborhood small enough}.
Since $d(\pi)$ is a finite morphism,
there is an open neighborhood
$
V \subset d(\widehat{\mathcal{Y}})
$
of $y$
with
$
d(\pi)^{-1}(V) \subset \cup_j V_j.
$
By using Proposition \ref{Proposition:overconvergent},
we see that,
for every \'etale sheaf
$\mathcal{F}$
of $\Z/n\Z$-modules on $\mathcal{X}$ adapted to $(f, \pi)$ with $n \in \O^{\times}$,
the sheaf
$R^id(f_\mathcal{Z})_!(\mathcal{F}_{\mathcal{Z}})^a$
is overconvergent on $d(\pi)^{-1}(V)$.
\end{proof}

We need the following finiteness result due to Huber.

\begin{thm}[{Huber \cite[Theorem 6.2.2]{Huber96}}]\label{Theorem:constructible theorem}
Let $Y$ be an adic space over $\Spa(K, \O)$.
Let $f \colon X \to Y$ be a morphism of adic spaces
which is \textit{smooth}, separated, and quasi-compact.
Let $\mathcal{F}$ be a constructible sheaf of $\Z/n\Z$-modules on $X$ with $n \in \O^\times$.
Then, the sheaf
$
R^if_!\mathcal{F}
$
on $Y$ is a constructible sheaf of $\Z/n\Z$-modules for every $i$.
\end{thm}
\begin{proof}
See \cite[Theorem 6.2.2]{Huber96}.
\end{proof}

\begin{rem}\label{Remark:constructible}
The assertion of
Theorem \ref{Theorem:constructible theorem}
can fail for non-smooth morphisms.
See the introduction of \cite{Huber98a} for details.
See also \cite[Proposition 7.1]{Huber07} for a more general result
for smooth, separated, and quasi-compact morphisms of analytic pseudo-adic spaces.
\end{rem}

As in the previous sections, we write
$\B(1)=\Spa(K \langle T \rangle, \O \langle T \rangle)$.
For an element $\epsilon \in \vert K^{\times} \vert$ with $\epsilon \leq 1$,
we define
\[
\B(\epsilon):= \{ x \in \B(1) \, \vert \, \vert T(x) \vert \leq \epsilon \}
\]
and
$
\B(\epsilon)^* := \B(\epsilon) \backslash \{ 0 \}.
$
Let $X$ be a one-dimensional adic space of finite type over $\Spa(K, \O)$.
We define
an
\textit{open disc}
$V \subset X$
as an open subset $V$ of $X$
equipped with an isomorphism
\[
\phi \colon \B(1) \cong V
\]
over $\Spa(K, \O)$.
For an open disc $V \subset X$,
we write
\[
V(\epsilon):=\phi(\B(\epsilon)) \quad \text{and} \quad
V(\epsilon)^*:=\phi(\B(\epsilon)^*).
\]
Similarly, we write
\[
V(a, b):=\phi(\B(a, b))
\]
for an open annulus $\B(a, b) \subset \B(1)$.

The main result of this section is the following theorem.

\begin{thm}\label{Theorem:uniform local constancy}
Let $f \colon \mathcal{X} \to \mathcal{Y}$
and $\pi \colon \mathcal{Z} \to \mathcal{Y}$
be morphisms as in Definition \ref{Definition:uniform finiteness adapted}.
We assume that
there is an open disc
$V \subset d(\widehat{\mathcal{Y}})$
such that
\[
d(f) \colon d(\widehat{\mathcal{X}}) \to d(\widehat{\mathcal{Y}})
\]
is smooth over $V(1)^*$.
Then there exists an element
$\epsilon_0 \in \vert K^{\times} \vert$
with $\epsilon_0 \leq 1$
such that,
for every \'etale sheaf $\mathcal{F}$ of $\Z/n\Z$-modules on $\mathcal{X}$ adapted to $(f, \pi)$ with $n \in \O^\times$,
the following two assertions hold:
\begin{enumerate}
    \item The restriction
    \[
    (R^id(f)_!\mathcal{F}^a)\vert_{V(\epsilon_0)^*}
    \]
    of $R^id(f)_!\mathcal{F}^a$ to $V(\epsilon_0)^*$ is a locally constant constructible sheaf of $\Z/n\Z$-modules
    for every $i$.
    \item 
    For elements $a, b \in \vert K^{\times} \vert$ with $a < b \leq \epsilon_0$,
    there exists a composition
    \[
    h \colon \B(c^{1/m}, d^{1/m}) \overset{\varphi_m}{\longrightarrow} \B(c, d) \overset{g}{\longrightarrow} V(a, b)
    \]
    of
    a Kummer covering
    $
    \varphi_m
    $
    of degree $m$, where $m$ is invertible in $\O$,
    with
    a finite Galois \'etale morphism
    $g$,
    such that
    the pull-back
    \[
    h^{*}((R^id(f)_!\mathcal{F}^a)\vert_{V(a, b)})
    \]
    is a constant sheaf associated with a finitely generated $\Z/n\Z$-module for every $i$.
    If $K$ is of characteristic zero, then we can take $g$ as a Kummer covering.
\end{enumerate}
\end{thm}

\begin{proof}
Clearly, the first assertion (1) follows from the second assertion (2). 
We shall prove (2).

\textit{Step 1.}
We may assume that,
for the dominant morphism
$\pi \colon \mathcal{Z} \to \mathcal{Y}$,
the separable closure $k(\mathcal{Y})^{\sep}$ of the function field $k(\mathcal{Y})$
of $\mathcal{Y}$ in the function field of $\mathcal{Z}$ is Galois over $k(\mathcal{Y})$.

Indeed,
there is a finite surjective morphism
$Z' \to \mathcal{Z}_K$
from an integral scheme $Z'$ which is smooth over $K$
such that
the separable closure of the function field
$k(\mathcal{Y})$
of $\mathcal{Y}$
in the function field of $Z'$ is Galois over $k(\mathcal{Y})$.
There exists a proper surjective morphism
$\mathcal{Z}' \to \mathcal{Z}$ such that
$\mathcal{Z}'$ is integral and the base change
$\mathcal{Z}'_K \to \mathcal{Z}_K$
is isomorphic to $Z' \to \mathcal{Z}_K$.
We define $\pi'$ as the composition
$\pi' \colon \mathcal{Z}' \to \mathcal{Z} \to \mathcal{Y}$.
If a sheaf $\mathcal{F}$ is adapted to $(f, \pi)$, then it is also adapted to
$(f, \pi')$.
Thus it suffices to prove Theorem \ref{Theorem:uniform local constancy} for $(f, \pi')$.

\textit{Step 2.}
We will choose an appropriate $\epsilon_0 \in \vert K^{\times} \vert$.

Let $W \to \mathcal{Y}_K$ be the normalization of
$\mathcal{Y}_K$ in $k(\mathcal{Y})^{\sep}$.
Then
the induced morphism
$\mathcal{Z}_K \to W$ is finite, radicial, and surjective
and there is a dense open subset
$U \subset \mathcal{Y}_K$
over which $W \to \mathcal{Y}_K$ is a finite Galois \'etale morphism.
Let
\[
W':=d(\widehat{\mathcal{Y}}) \times_{\mathcal{Y}_K} W 
\]
be
the adic space over
$d(\widehat{\mathcal{Y}})$ associated with
$
d(\widehat{\mathcal{Y}}) \to \mathcal{Y}_K
$
and
$W \to \mathcal{Y}_K$.
Let
$
g \colon W' \to d(\widehat{\mathcal{Y}})
$
denote the structure morphism.
The morphism $d(\pi)$
can be written as the composition of
finite morphisms
\[
d(\widehat{\mathcal{Z}}) \overset{\alpha}{\to} W' \overset{g}{\to} d(\widehat{\mathcal{Y}}).
\]

Let
$\epsilon_1 \in \vert K^{\times} \vert$
be an element with $\epsilon_1 \leq 1$
such that
$V(\epsilon_1)^* \subset d(\widehat{\mathcal{Y}})$ is mapped into $U$
under the map
$d(\widehat{\mathcal{Y}}) \to \mathcal{Y}_K$.
Then the restriction
\[
g^{-1}(V(\epsilon_1)^*) \to V(\epsilon_1)^*
\]
is finite and \'etale.
By Theorem \ref{Theorem:split into annuli},
there exists an element $\epsilon_2 \in \vert K^\times \vert$ with $\epsilon_2 \leq \epsilon_1$ such that,
for all $a, b \in \vert K^{\times} \vert$
with
$a < b \leq \epsilon_2$,
we have
\[
g^{-1}(V(a, b)) \cong \coprod^{N}_{j=1} \B(c_j, d_j)
\]
for some elements $c_j, d_j \in \vert K^{\times} \vert$
with
$c_j < d_j \leq 1$ ($1 \leq j \leq N$).
If $K$ is of characteristic zero,
then we can take such $\epsilon_2 \in \vert K^\times \vert$ so that
the restriction
\[
\B(c_j, d_j)  \to V(a, b)
\]
of $g$ to every component $\B(c_j, d_j)$ appearing in the above decomposition is a Kummer covering.
By Proposition \ref{Proposition:overconvergent and unipotent} (2),
there exists an element $\epsilon_3 \in \vert K^\times \vert$ with $\epsilon_3 \leq \epsilon_2$
such that,
for every \'etale sheaf $\mathcal{F}$ of $\Z/n\Z$-modules on $\mathcal{X}$ adapted to $(f, \pi)$ with $n \in \O^{\times}$,
the sheaf
$
d(\pi)^*R^id(f)_!\mathcal{F}^a
$
is overconvergent on $d(\pi)^{-1}(V(\epsilon_3))$ for every $i$.
Let $t \in \vert K^{\times} \vert$ be an element with $t < 1$.
Then we put $\epsilon_0:=t\epsilon_3$.

\textit{Step 3.}
We shall show that $\epsilon_0$ satisfies the condition.

Indeed,
let $\mathcal{F}$ be an \'etale sheaf of $\Z/n\Z$-modules on $\mathcal{X}$ adapted to $(f, \pi)$ with $n \in \O^{\times}$.
Let $a, b \in \vert K^{\times} \vert$
be elements with $a < b \leq \epsilon_0$.
We have
$
g^{-1}(V(ta, b/t)) \cong \coprod^{N}_{i=1} \B(c_j, d_j)
$
for some elements $c_j, d_j \in \vert K^{\times} \vert$
with
$c_j < d_j \leq 1$ ($1 \leq j \leq N$).
We take a component
$\B(c_1, d_1)$
of the decomposition.
The restriction
\[
g \colon \B(c_1, d_1) \to V(ta, b/t)
\]
of $g$ is denoted by the same letter.
By the construction, it is a finite Galois \'etale morphism.

We remark that,
since the morphism
$\mathcal{Z}_K \to W$ is finite, radicial, and surjective,
it follows that
$
\alpha \colon d(\widehat{\mathcal{Z}}) \to  W'
$
is a homeomorphism and,
for every $z \in d(\widehat{\mathcal{Z}})$,
the extension
$k(\alpha(z))^{\wedge} \to k(z)^{\wedge}$
of the completions of the residue fields
is a finite purely inseparable extension, and hence
the extension
$k(\alpha(z))^{\wedge h} \to k(z)^{\wedge h}$
of the Henselizations of these fields is also a finite purely inseparable extension.

Since
$
d(f) \colon d(\widehat{\mathcal{X}}) \to d(\widehat{\mathcal{Y}})
$
is smooth over $V(1)^*$,
the sheaf of $\Z/n\Z$-modules
\[
\mathcal{G}:=g^{*}((R^id(f)_!\mathcal{F}^a)\vert_{V(ta, b/t)})
\]
on $\B(c_1, d_1)$
is constructible by Theorem \ref{Theorem:constructible theorem}.
By the construction,
it is overconvergent on $\B(c_1, d_1)$.
Therefore, by \cite[Lemma 2.7.11]{Huber96},
the sheaf $\mathcal{G}$ is locally constant.
Moreover,
by Proposition \ref{Proposition:overconvergent and unipotent} (1) and the remark above,
the sheaf
$
\mathcal{G}
$
is tame at every
$x \in \B(c_1, d_1)$ having a proper generalization in $\B(c_1, d_1)$.
We have
$
g^{-1}(V(a, b))=\B(c, d)
$
for some elements $c, d \in \vert K^\times \vert$ with $c_1 < c < d  < d_1$.
Hence,
by Theorem \ref{Theorem:trivialization of tame sheaf},
we conclude that
the restriction
of 
$\mathcal{G}$
to
$g^{-1}(V(a, b))=\B(c, d)$
is trivialized by a Kummer covering
$\varphi_m \colon \B(c^{1/m}, d^{1/m}) \to \B(c, d)$ with $m \in \O^\times$.

The proof of Theorem \ref{Theorem:uniform local constancy} is now complete.
\end{proof}

\begin{rem}\label{Remark:finite cover is independent of n}
In Theorem \ref{Theorem:uniform local constancy}, the morphism $g$ can be taken independent of $\mathcal{F}$ although the integer $m$ possibly depends on $\mathcal{F}$.
\end{rem}

For an element $ \epsilon \in \vert K^{\times} \vert$,
we define
\[
\D(\epsilon):= \{ x \in \B(1) \, \vert \, \vert T(x) \vert < \epsilon \}.
\]
This is a closed constructible subset of $\B(1)$.
For later use, we record the following results.

\begin{lem}\label{Lemma:calculation of cohomology}
Let $n$ be a positive integer invertible in $\O$.
\begin{enumerate}
    \item Let $a, b \in \vert K^\times \vert$ be elements with
$a < b \leq 1$.
Let $\mathcal{F}$ be a
locally constant constructible sheaf of $\Z/n\Z$-modules on $\B(a, b)$.
Assume that there exists a finite \'etale morphism
$
h \colon \B(c, d) \to \B(a, b)
$
such that $h$ is a composition of finite Galois \'etale morphisms
and
the pull-back
$
h^{*}\mathcal{F}
$
is a constant sheaf.
Then the following assertions hold:
\begin{enumerate}
    \item  We have
    \[
    H^i_c(\B(b) \backslash \B(a), \mathcal{F}\vert_{\B(b) \backslash \B(a)})=0
    \quad
    \text{and}
    \quad
    H^i_c(\D(b) \cap \B(a, b), \mathcal{F}\vert_{\D(b) \cap \B(a, b)})=0
    \]
    for every $i$.
    \item
    The restriction map
    \[
    H^i(\B(a, b), \mathcal{F})
    \to
    H^i(\D(b)^{\circ} \cap \B(a, b), \mathcal{F})
    \]
    is an isomorphism for every $i$, where $\D(b)^{\circ}$ is the interior of $\D(b)$ in $\B(1)$.
\end{enumerate}
\item Let $\mathcal{F}$ be a
locally constant constructible sheaf of $\Z/n\Z$-modules on $\B(1)^*$.
Assume that
for all $a, b \in \vert K^{\times} \vert$ with $a < b \leq 1$,
there exists a finite \'etale morphism
$
h \colon \B(c, d) \to \B(a, b)
$
such that $h$ is a composition of finite Galois \'etale morphisms
and
the pull-back
$
h^{*}(\mathcal{F}\vert_{\B(a, b)})
$
is a constant sheaf.
Then we have 
\[
H^i_c(\D(1) \backslash \{ 0 \}, \mathcal{F}\vert_{\D(1) \backslash \{ 0 \}})=0
\]
for every $i$.
\end{enumerate}
\end{lem}
\begin{proof}
(1)
After possibly changing the coordinate function of $\B(c, d)$,
we may assume that
$
h^{-1}(\B(a, a)) = \B(c, c)
$
and
$
h^{-1}(\B(b, b)) = \B(d, d).
$
Then, we have
\[
h^{-1}(\B(b) \backslash \B(a)) = \B(d) \backslash \B(c) \quad \text{and} \quad
h^{-1}(\D(b) \cap \B(a, b)) = \D(d) \cap \B(c, d).
\]

We shall show the first equality of (a).
By the spectral sequence in \cite[4.2 ii)]{Huber98a} and the fact that $h$ is a composition of finite Galois \'etale morphisms,
it is enough to show that,
for a constant sheaf $M$ on $\B(d) \backslash \B(c)$ associated with a finitely generated $\Z/n\Z$-module,
we have for every $i$
\begin{equation}\label{equation:compact support annulas formula}
    H^i_c(\B(d) \backslash \B(c), M)=0.
\end{equation}
This is proved in \cite[(II) in the proof of Theorem 2.5]{Huber98a}.
(Although the characteristic of the base field is always assumed to be zero in \cite{Huber98a}, we need not assume that the characteristic of $K$ is zero here.)

We shall show the second equality of (a).
Similarly as above,
it suffices to show that,
for a constant sheaf $M$ on $\D(d) \cap \B(c, d)$ associated with a finitely generated $\Z/n\Z$-module,
we have
$
H^i_c(\D(d) \cap \B(c, d), M)=0
$
for every $i$.
But this follows from (\ref{equation:compact support annulas formula}) since
$\D(d) \cap \B(c, d)$ is isomorphic to $\B(d) \backslash \B(c)$ as a pseudo-adic space over $\Spa(K, \O)$.

We shall prove (b).
By the \v{C}ech-to-cohomology spectral sequences
and by the fact that $h$ is a composition of finite Galois \'etale morphisms,
it is enough to show that,
for a constant sheaf $M$ on $\B(c, d)$ associated with a finitely generated $\Z/n\Z$-module,
the restriction map
\[
H^i(\B(c, d), M)
\to
H^i(\D(d)^{\circ} \cap \B(c, d), M)
\]
is an isomorphism for every $i$.
Let $t \in \vert K^\times \vert$ be an element with
$t <1$.
Then we have
\[
\D(d)^{\circ} \cap \B(c, d) = \bigcup_{m \in \Z_{>0}} \B(c, t^{1/m} d).
\]
(See also \cite[Lemma 1.3]{Huber98a}.)
Moreover
$H^i(\B(c,  t^{1/m} d), M)$ is a finite group for every $i$; see \cite[Proposition 6.1.1]{Huber96}.
Therefore,
by \cite[Lemma 3.9.2 i)]{Huber96},
we obtain that
\[
H^i(\D(d)^{\circ} \cap \B(c, d), M)
\cong 
\plim[m] H^i(\B(c,  t^{1/m} d), M)
\]
for every $i$.
Thus it suffices to prove that,
for every $m \in \Z_{>0}$,
the restriction map
\[
H^i(\B(c,  d), M) \to H^i(\B(c,  t^{1/m} d), M)
\]
is an isomorphism for every $i$.
By \cite[Example 6.1.2]{Huber96},
both sides vanish when $i \geq 2$.
For $i \leq 1$,
the assertion can be proved by using the Kummer sequence.
(See the last paragraph of the proof of Theorem \ref{Theorem:trivialization of tame sheaf} in Appendix \ref{Appendix:finite etale coverings of annuli}.)

(2)
Since
\[
\D(1) \backslash \{ 0 \} = \bigcup_{\epsilon \in \vert K^{\times} \vert, \,  \epsilon < 1} \D(1) \cap \B(\epsilon, 1),
\]
we have
\[
\ilim[\epsilon] H^i_c(\D(1) \cap \B(\epsilon, 1), \mathcal{F}\vert_{\D(1) \cap \B(\epsilon, 1)}) \cong H^i_c(\D(1) \backslash \{ 0 \}, \mathcal{F}\vert_{\D(1) \backslash \{ 0 \}})
\]
by \cite[Proposition 5.4.5 ii)]{Huber96}.
Therefore, the assertion follows from (1).
\end{proof}

\section{Proofs of Theorem \ref{Theorem:tubular neighborhood compact support} and Theorem \ref{Theorem:tubular neighborhood direct image}}\label{Section:proofs of main theorems}

In this section, we shall prove Theorem \ref{Theorem:tubular neighborhood compact support} and Theorem \ref{Theorem:tubular neighborhood direct image}.
Let $K$ be an algebraically closed complete non-archimedean field with ring of integers $\O$.

The main part of the proofs of Theorem \ref{Theorem:tubular neighborhood compact support} and Theorem \ref{Theorem:tubular neighborhood direct image} is contained in the following lemma.

\begin{lem}\label{Lemma:key lemma, tubular neighborhood}
Let $\mathcal{X}$ be a separated scheme of finite type over $\O$ and $\mathcal{Z} \hookrightarrow \mathcal{X}$ a closed immersion defined by one global section $f \in \O_{\mathcal{X}}(\mathcal{X})$.
Let
$
f \colon \mathcal{X} \to \Spec \O[T]
$
be the morphism defined by $T \mapsto f$,
which is also denoted by $f$.
We assume that
there is an element $\epsilon_1 \in \vert K^\times \vert$ 
with $\epsilon_1 \leq 1$
such that 
\[
d(f) \colon d(\widehat{\mathcal{X}}) \to d((\Spec \O[T])^{\wedge})=\B(1).
\]
is smooth over $\B(\epsilon_1)^*=\B(\epsilon_1) \backslash \{ 0 \}$.
Then, there exists an element
$\epsilon_0 \in \vert K^{\times} \vert$
with $\epsilon_0 \leq \epsilon_1$,
such that
the following assertions hold
for every
$\epsilon \in \vert K^{\times} \vert$
with
$\epsilon \leq \epsilon_0$,
every positive integer $n$ invertible in $\O$,
and every integer $i$.
\begin{enumerate}
    \item 
    $H^i_c(S(\mathcal{Z}, \epsilon)\backslash d(\widehat{\mathcal{Z}}), \Z/n\Z)=0$ and
    $H^i_c(T(\mathcal{Z}, \epsilon_0)\backslash T(\mathcal{Z}, \epsilon), \Z/n\Z)=0.$
    \item
    $
    H^i(T(\mathcal{Z}, \epsilon), \Z/n\Z) \overset{\cong}{\to}
    H^i(S(\mathcal{Z}, \epsilon)^{\circ}, \Z/n\Z).
    $
    \item
    $
    H^i(S(\mathcal{Z}, \epsilon), \Z/n\Z)\overset{\cong}{\to}
    H^i(d(\widehat{\mathcal{Z}}), \Z/n\Z).
    $
    \item
    $
    H^i(S(\mathcal{Z}, \epsilon), \Z/n\Z) \overset{\cong}{\to}
    H^i(S(\mathcal{Z}, \epsilon)^{\circ}, \Z/n\Z).
    $
\end{enumerate}
\end{lem}

We first deduce Theorem \ref{Theorem:tubular neighborhood compact support} and Theorem \ref{Theorem:tubular neighborhood direct image} from Lemma \ref{Lemma:key lemma, tubular neighborhood}.
We will prove Lemma \ref{Lemma:key lemma, tubular neighborhood} in Section \ref{Subsection:Proof of the key case}.

\subsection{Cohomological descent for analytic adic spaces}\label{Subsection:cohomological descent for analytic adic spaces}

We will recall some results on cohomological descent for analytic pseudo-adic spaces.
Our basic references are \cite[Expos\'e Vbis]{SGA 4-II}
and \cite[Section 5]{Deligne HodgeIII}.
See also \cite[Section 3]{Mieda06}.

Let $f \colon Y \to X$ be a morphism of finite type of analytic pseudo-adic spaces.
Let
\[
\beta \colon Y_{\bullet}:=\cosq_0(Y/X) \to X
\]
be the augmented simplicial object in the category of analytic pseudo-adic spaces of finite type over $X$
(this category has finite projective limits by \cite[Proposition 1.10.6]{Huber96})
defined as in
\cite[(5.1.4)]{Deligne HodgeIII}.
So $Y_m$ is the
$(m+1)$-times fiber product
$Y \times_X \cdots \times_X Y$ for $m \geq 0$.
As in \cite[(5.1.6)--(5.1.8)]{Deligne HodgeIII},
one can associate to
the \'etale topoi $(Y_m)^{\sim}_{\et}$ ($m \geq 0$)
a topos
$(Y_\bullet)^{\sim}$.
Moreover,
as in \cite[(5.1.11)]{Deligne HodgeIII},
we have a morphism of topoi
\[
(\beta_*, \beta^*) \colon (Y_\bullet)^{\sim} \to X^{\sim}_\et
\]
from $(Y_\bullet)^{\sim}$ to the \'etale topos $X^{\sim}_\et$ of $X$.
We say that $f \colon Y \to X$ is a
\textit{morphism of cohomological descent} for torsion abelian \'etale sheaves
if for every torsion abelian \'etale sheaf $\mathcal{F}$ on $X$,
the natural morphism
\[
\mathcal{F} \to R\beta_*\beta^*\mathcal{F}
\]
in the derived category $D^+(X^{\sim}_\et)$
is an isomorphism.
See \cite[Expos\'e Vbis, Section 2]{SGA 4-II} for details.

As a consequence of the proper base change theorem for analytic pseudo-adic spaces \cite[Theorem 4.4.1]{Huber96},
we have the following proposition.
We formulate it in the generality we need.

\begin{prop}\label{Proposition:proper descent for adic spaces}
Let $f \colon Y \to X$ be a morphism of analytic adic spaces which is proper, of finite type, and surjective. 
Then
for every morphism $Z \to X$ of analytic pseudo-adic spaces,
the base change $f \colon Y \times_X Z \to Z$ is of cohomological descent for torsion abelian \'etale sheaves.
\end{prop}
\begin{proof}
First, we note that the fiber product $Y \times_X Z \to Z$ exists by \cite[Proposition 1.10.6]{Huber96}.
By the proper base change theorem for analytic pseudo-adic spaces
\cite[Theorem 4.4.1 (b)]{Huber96},
it suffices to prove that, for every geometric point
$S \to X$,
the base change $Y \times_X S \to S$ is of cohomological descent for torsion abelian \'etale sheaves.
It is enough to show that there exists a section $S \to Y \times_X S$; see \cite[Expos\'e Vbis, Proposition 3.3.1]{SGA 4-II} for example.
The existence of a section can be easily proved in our case:
By the properness of $f$ and \cite[Corollary 1.3.9]{Huber96}, we may assume that $S$ is of rank $1$.
Then it is well known.
\end{proof}

For future reference,
we deduce the following corollaries from Proposition \ref{Proposition:proper descent for adic spaces}.

\begin{cor}\label{Corollary:cohomological descent spectral sequence compact support}
Let $\mathcal{X}$ be a separated scheme of finite type over $\O$.
Let
$\beta_0 \colon \mathcal{Y} \to \mathcal{X}$
be a proper surjective morphism.
We put
$
\beta \colon \mathcal{Y}_{\bullet}=\cosq_0(\mathcal{Y}/\mathcal{X}) \to \mathcal{X}.
$
Let $Z \subset d(\widehat{\mathcal{X}})$ be a taut locally closed subset.
Then we have the following spectral sequence:
\[
E^{i, j}_{1}=H^j_c(Z_i, \Z/n\Z) \Rightarrow H^{i+j}_c(Z, \Z/n\Z),
\]
where $Z_i$ is the inverse image of $Z$ under the morphism
$d(\beta_i) \colon d(\widehat{\mathcal{Y}}_i) \to d(\widehat{\mathcal{X}})$.
\end{cor}
\begin{proof}
By Nagata's compactification theorem,
there exists a proper scheme $\mathcal{P}$ over $\Spec \O$
with a dense open immersion
$u \colon \mathcal{X} \hookrightarrow \mathcal{P}$ over $\Spec \O$.
Moreover,
there is the following Cartesian diagram of schemes:
\[
\xymatrix{ \mathcal{Y} \ar[r]^-{u'} \ar[d]^-{\beta_0} & \mathcal{Q} \ar[d]^-{\beta'_0} \\
\mathcal{X} \ar[r]^-{u}  & \mathcal{P}
}
\]
where $u'$ is an open immersion and $\beta'_0$ is a proper surjective morphism.
The morphism
$d(\beta'_0) \colon d(\widehat{\mathcal{Q}}) \to d(\widehat{\mathcal{P}})$
is proper, of finite type, and surjective; see \cite[Lemma 3.5]{Mieda06} (although the base field is assumed to be a discrete valuation field in \cite{Mieda06}, the same proof works).
We put
$
\beta' \colon \mathcal{Q}_{\bullet}=\cosq_0(\mathcal{Q}/\mathcal{P}) \to \mathcal{P}.
$
We have
a taut locally closed embedding
$j \colon Z \to d(\widehat{\mathcal{P}})$.
Let $j_m \colon Z_m \to d(\widehat{\mathcal{Q}}_m)$
be the pull-back of $j$ by $d(\beta'_m)$.
Then we have
$H^i_c(Z, \Z/n\Z)=H^i(d(\widehat{\mathcal{P}}), j_!\Z/n\Z)$
and
\[
H^i_c(Z_m, \Z/n\Z)=H^i(d(\widehat{\mathcal{Q}}_m), (j_m)_!\Z/n\Z)
=H^i(d(\widehat{\mathcal{Q}}_m), d(\beta'_m)^*j_!\Z/n\Z).
\]
Therefore the assertion follows from Proposition \ref{Proposition:proper descent for adic spaces} and \cite[Expos\'e Vbis, Proposition 2.5.5]{SGA 4-II}.
\end{proof}

\begin{cor}\label{Corollary:cohomological descent direct image}
Let $\beta_0 \colon Y \to X$ be a morphism of analytic adic spaces which is proper, of finite type, and surjective.
We put
$\beta \colon Y_{\bullet}=\cosq_0(Y/X) \to X$.
Let $i \colon Z \hookrightarrow W$ be an inclusion of locally closed subsets of $X$.
For $m \geq 0$, let $i_m \colon Z_m \hookrightarrow W_m$ be the pull-back of $i$ by
$\beta_m \colon Y_m \to X$.
Let
\[
R\Gamma(W, \Z/n\Z) \to R\Gamma(Z, \Z/n\Z) \to \mathcal{K} \to
\]
and
\[
R\Gamma(W_m, \Z/n\Z) \to R\Gamma(Z_m, \Z/n\Z) \to \mathcal{K}_m \to
\]
be distinguished triangles.
Let $\rho \geq -1$ be an integer.
If $\tau_{\leq (\rho-m)}\mathcal{K}_m = 0$
for every $0 \leq m \leq \rho+1$,
then we have
$\tau_{\leq \rho}\mathcal{K}=0$.
\end{cor}
\begin{proof}
The assertion can be proved by a similar method used in \cite[Lemma 4.1]{Orgogozo06}.
We shall give a brief sketch here.

Let
$\beta' \colon W_{\bullet}=\cosq_0(W_0/W) \to W$
(resp.\
$\beta'' \colon Z_{\bullet}=\cosq_0(Z_0/Z) \to Z$)
be the base change of
$\beta$ to $W$ (resp.\ to $Z$).
The morphism $i$ induces a morphism
$
i_\bullet \colon (Z_{\bullet})^\sim \to (W_{\bullet})^\sim
$
of topoi.
For the sheaf $\Z/n\Z$ on $W$,
we have
$\Z/n\Z \cong R\beta'_*\beta'^*\Z/n\Z$
by Proposition \ref{Proposition:proper descent for adic spaces},
and we obtain isomorphisms
\[
Ri_*i^* R\beta'_*\beta'^*\Z/n\Z \cong Ri_*R\beta''_*(i_{\bullet})^*\beta'^*\Z/n\Z
\cong R\beta'_*R(i_{\bullet})_*(i_{\bullet})^*\beta'^*\Z/n\Z
\]
by the proper base change theorem \cite[Theorem 4.4.1 (b)]{Huber96} and by a spectral sequence as in \cite[(5.2.7.1)]{Deligne HodgeIII} (see also \cite[Tag 0D7A]{StacksProject}).
Thus,
by applying $R\beta'_*$ to
the following distinguished triangle
\[
\beta'^*\Z/n\Z \to R(i_{\bullet})_*(i_{\bullet})^*\beta'^*\Z/n\Z \to \Delta \to,
\]
we have the following distinguished triangle
\[
\Z/n\Z \to Ri_*i^*\Z/n\Z \to R\beta'_*\Delta \to.
\]
This implies that
$
R\Gamma((W_{\bullet})^\sim, \Delta)\cong
R\Gamma(W, R\beta'_*\Delta) \cong \mathcal{K}.
$
By
\cite[Expos\'e Vbis, Corollaire 1.3.12]{SGA 4-II},
we have
$R\Gamma(W_m, \Delta \vert_{W_m}) \cong \mathcal{K}_m$
for every $m \geq 0$,
where $\Delta \vert_{W_m}$ denotes the restriction of $\Delta$ to $(W_m)^\sim_\et$.
Now the assertion follows from the following spectral sequence (cf.\ \cite[(5.2.3.2)]{Deligne HodgeIII} and \cite[Tag 09WJ]{StacksProject}):
\[
E^{p, q}_1=H^q(W_p, \Delta \vert_{W_p}) \Rightarrow
H^{p+q}((W_{\bullet})^\sim, \Delta).
\]
\end{proof}

\subsection{Reduction to the key case} \label{Subsection:Reduction to the key case}

In this subsection, we deduce Theorem \ref{Theorem:tubular neighborhood compact support} and Theorem \ref{Theorem:tubular neighborhood direct image} from Lemma \ref{Lemma:key lemma, tubular neighborhood}.
A theorem of de Jong \cite[Theorem 4.1]{deJong96} will again play a key role.

\begin{lem}\label{Lemma:reduction to V(f) case compact support}
To prove Theorem \ref{Theorem:tubular neighborhood compact support}, it suffices to prove the following statement
$\textbf{P}_c(i)$ for every integer $i$.
\begin{enumerate}
    \item[$\textbf{P}_c(i)$:]
    Let $\mathcal{X}$ be a separated scheme of finite type over $\O$ and $\mathcal{Z} \hookrightarrow \mathcal{X}$ a closed immersion defined by one global section $f \in \O_{\mathcal{X}}(\mathcal{X})$.
Then there exists an element
$\epsilon_0 \in \vert K^{\times} \vert$,
such that for every
$\epsilon \in \vert K^{\times} \vert$
with
$\epsilon \leq \epsilon_0$ and every positive integer $n$ invertible in $\O$,
we have
\[
H^j_c(S(\mathcal{Z}, \epsilon)\backslash d(\widehat{\mathcal{Z}}), \Z/n\Z)=0
\quad \text{and} \quad
H^j_c(T(\mathcal{Z}, \epsilon_0)\backslash T(\mathcal{Z}, \epsilon), \Z/n\Z)=0
\]
for every $j \leq i$.
\end{enumerate}
\end{lem}

\begin{proof}
\textit{Step 1.}
To prove Theorem \ref{Theorem:tubular neighborhood compact support},
it suffices to prove the following statement
$\textbf{P}'_c(i)$ for every $i$.
\begin{enumerate}
    \item[$\textbf{P}'_c(i)$:]
    Let $\mathcal{X}$ be a separated scheme of finite type over $\O$ and $\mathcal{Z} \hookrightarrow \mathcal{X}$ a closed immersion of finite presentation.
Then there exists an element
$\epsilon_0 \in \vert K^{\times} \vert$,
such that for every
$\epsilon \in \vert K^{\times} \vert$
with
$\epsilon \leq \epsilon_0$ and every positive integer $n$ invertible in $\O$,
we have
\[
H^j_c(S(\mathcal{Z}, \epsilon)\backslash d(\widehat{\mathcal{Z}}), \Z/n\Z)=0
\quad
\text{and}
\quad
H^j_c(T(\mathcal{Z}, \epsilon_0)\backslash T(\mathcal{Z}, \epsilon), \Z/n\Z)=0
\]
for every $j \leq i$.
\end{enumerate}

Indeed,
let $\mathcal{X}$ be a separated scheme of finite type over $\O$ and $\mathcal{Z} \hookrightarrow \mathcal{X}$ a closed immersion of finite presentation.
By applying \cite[Remark 5.5.11 iv)]{Huber96} to the following diagram
\[
S(\mathcal{Z}, \epsilon)\backslash d(\widehat{\mathcal{Z}}) \hookrightarrow
S(\mathcal{Z}, \epsilon)
\hookleftarrow d(\widehat{\mathcal{Z}}),
\]
we have the following long exact sequence:
\[
\begin{split}
\cdots \to H^i_c(S(\mathcal{Z}, \epsilon)\backslash d(\widehat{\mathcal{Z}}), \Z/n\Z) \to & H^i_c(S(\mathcal{Z}, \epsilon), \Z/n\Z)  \\
\to H^i_c(d(\widehat{\mathcal{Z}}), \Z/n\Z) \to & H^{i+1}_c(S(\mathcal{Z}, \epsilon)\backslash d(\widehat{\mathcal{Z}}), \Z/n\Z) \to \cdots.
\end{split}
\]
We note that
$
S(\mathcal{Z}, \epsilon)\backslash d(\widehat{\mathcal{Z}}) = (d(\widehat{\mathcal{X}})\backslash d(\widehat{\mathcal{Z}}))\backslash Q(\mathcal{Z}, \epsilon).
$
Hence we have a similar spectral sequence for
the diagram
\[
Q(\mathcal{Z}, \epsilon)
\hookrightarrow
d(\widehat{\mathcal{X}})\backslash d(\widehat{\mathcal{Z}})
\hookleftarrow S(\mathcal{Z}, \epsilon)\backslash d(\widehat{\mathcal{Z}}).
\]
Moreover,
for elements $\epsilon, \epsilon' \in \vert K^\times \vert$ with
$\epsilon \leq \epsilon'$,
we have a similar spectral sequence for
the diagram
\[
T(\mathcal{Z}, \epsilon) \hookrightarrow
T(\mathcal{Z}, \epsilon')
\hookleftarrow T(\mathcal{Z}, \epsilon')\backslash T(\mathcal{Z}, \epsilon).
\]
By \cite[Proposition 5.5.8]{Huber96},
there exists an integer $N$,
which is independent of $n$ and $\epsilon, \epsilon' \in \vert K^\times \vert$,
such that we have
\[
H^{i}_c(S(\mathcal{Z}, \epsilon) \backslash d(\widehat{\mathcal{Z}}), \Z/n\Z)=0
\quad \text{and} \quad
H^i_c(T(\mathcal{Z}, \epsilon')\backslash T(\mathcal{Z}, \epsilon), \Z/n\Z)=0
\]
for every $i \geq N$.
Our claim follows from these results.

\textit{Step 2.}
We suppose that $\textbf{P}_c(i)$ holds for every $i$.
We will prove $\textbf{P}'_c(i)$ by induction on $i$.
The assertion holds trivially for $i=-1$.
We assume that $\textbf{P}'_c(i_0-1)$ holds.
First, we claim that, to prove
$\textbf{P}'_c(i_0)$,
we may assume that
$\mathcal{X}$ is integral.

We may assume that $\mathcal{X}$ is flat over $\O$.
Then every irreducible component of $\mathcal{X}$ dominates $\Spec \O$,
and hence $\mathcal{X}$ has finitely many irreducible components.
Let $\mathcal{X}'$ be the disjoint union of the irreducible components of $\mathcal{X}$.
Then $\mathcal{X}' \to \mathcal{X}$ is proper and surjective.
By $\textbf{P}'_c(i_0-1)$ and Corollary \ref{Corollary:cohomological descent spectral sequence compact support}, it suffices to prove
$\textbf{P}'_c(i_0)$ for $\mathcal{X}'$ and $\mathcal{Z} \times_{\mathcal{X}} \mathcal{X}'$.
By considering each component of $\mathcal{X}'$ separately, our claim follows.

\textit{Step 3.}
We assume that $\mathcal{X}$ is integral.
We may assume further that $\mathcal{Z}$ is not equal to
$\mathcal{X}$.
Let $\mathcal{Y} \to \mathcal{X}$ be the blow-up of
$\mathcal{X}$ along $\mathcal{Z}$, which is proper and surjective.
By $\textbf{P}'_c(i_0-1)$ and Corollary \ref{Corollary:cohomological descent spectral sequence compact support},
it suffices to prove $\textbf{P}'_c(i_0)$ for $\mathcal{Y}$ and
$\mathcal{Z} \times_{\mathcal{X}} \mathcal{Y}$.
Consequently, to prove $\textbf{P}'_c(i_0)$, we may assume further that
$\mathcal{Z} \hookrightarrow \mathcal{X}$ is locally defined by one function.

Finally,
let
$\mathcal{X}=\bigcup_{\alpha \in I} \mathcal{U}_\alpha$
be a finite affine covering
such that
$\mathcal{Z} \cap \mathcal{U}_\alpha \hookrightarrow \mathcal{U}_\alpha$
is defined by one global section in $\O_{\mathcal{X}}(\mathcal{U}_\alpha)$ for every $\alpha \in I$.
We have the following spectral sequence by \cite[Remark 5.5.12 iii)]{Huber96}:
\[
E^{i, j}_{1}=\bigoplus_{J \subset I, \, \sharp J = -i+1} H^j_c(S(\mathcal{Z}_J, \epsilon) \backslash d(\widehat{\mathcal{Z}}_J), \Z/n\Z) \Rightarrow H^{i+j}_c(S(\mathcal{Z}, \epsilon) \backslash d(\widehat{\mathcal{Z}}), \Z/n\Z).
\]
Here we write $\mathcal{U}_J:= \cap_{\alpha \in J}\mathcal{U}_\alpha$ and
$\mathcal{Z}_J := \mathcal{Z} \times_{\mathcal{X}} \mathcal{U}_J \hookrightarrow \mathcal{U}_J$.
We have a similar spectral sequence for
$T(\mathcal{Z}, \epsilon')\backslash T(\mathcal{Z}, \epsilon)$.
Since $\textbf{P}_c(i)$ holds for every $i$ by hypothesis,
it follows that $\textbf{P}'_c(i_0)$ holds for
$\mathcal{X}$ and $\mathcal{Z} \hookrightarrow \mathcal{X}$.
\end{proof}

\begin{lem}\label{Lemma:reduction to V(f) direct image}
To prove Theorem \ref{Theorem:tubular neighborhood direct image},
it suffices to prove the following statement
$\textbf{P}(i)$ for every integer $i$.
\begin{enumerate}
    \item[$\textbf{P}(i)$:]
    Let $\mathcal{X}$ be a separated scheme of finite type over $\O$ and $\mathcal{Z} \hookrightarrow \mathcal{X}$ a closed immersion defined by one global section $f \in \O_{\mathcal{X}}(\mathcal{X})$.
We consider the following distinguished triangles:
\[
R\Gamma(T(\mathcal{Z}, \epsilon), \Z/n\Z) \to R\Gamma(S(\mathcal{Z}, \epsilon)^{\circ}, \Z/n\Z) \to \mathcal{K}_1(\epsilon, n) \to,
\]
\[
R\Gamma(S(\mathcal{Z}, \epsilon), \Z/n\Z) \to R\Gamma(d(\widehat{\mathcal{Z}}), \Z/n\Z) \to \mathcal{K}_2(\epsilon, n) \to,
\]
\[
R\Gamma(S(\mathcal{Z}, \epsilon), \Z/n\Z) \to R\Gamma(S(\mathcal{Z}, \epsilon)^{\circ}, \Z/n\Z) \to \mathcal{K}_3(\epsilon, n) \to.
\]
Then there exists an element
$\epsilon_0 \in \vert K^{\times} \vert$,
such that for every
$\epsilon \in \vert K^{\times} \vert$
with
$\epsilon \leq \epsilon_0$ and every positive integer $n$ invertible in $\O$,
we have 
$
\tau_{\leq i}\mathcal{K}_m(\epsilon, n)=0
$
for every $m \in \{ 1, 2, 3 \}$.
\end{enumerate}
\end{lem}
\begin{proof}
Let $\mathcal{X}$ be a separated scheme of finite type over $\O$ and $\mathcal{Z} \hookrightarrow \mathcal{X}$ a closed immersion of finite presentation.
By \cite[Corollary 2.8.3]{Huber96},
there exists an integer $N$,
which is independent of $\epsilon$ and $n$,
such that
cohomology groups
$
H^i(T(\mathcal{Z}, \epsilon), \Z/n\Z),
$
$
H^i(S(\mathcal{Z}, \epsilon), \Z/n\Z),
$
and
$H^i(d(\widehat{\mathcal{Z}}), \epsilon), \Z/n\Z)$
vanish for every $i \geq N$.
Let $t \in \vert K^\times \vert$ be an element with
$t <1$.
Then we have
\[
S(\mathcal{Z}, \epsilon)^{\circ} = \bigcup_{m \in \Z_{>0}} T(\mathcal{Z}, t^{1/m}\epsilon);
\]
see \cite[Lemma 1.3]{Huber98a}.
Therefore, by \cite[Lemma 3.9.2 i)]{Huber96},
the same holds for the cohomology group
$H^i(S(\mathcal{Z}, \epsilon)^{\circ}, \Z/n\Z)$.

We can prove the assertion by the same argument as in the proof of Lemma \ref{Lemma:reduction to V(f) case compact support}
by using the results remarked above instead of
\cite[Proposition 5.5.8]{Huber96} and
by using
Corollary \ref{Corollary:cohomological descent direct image}
instead of
Corollary \ref{Corollary:cohomological descent spectral sequence compact support}.
\end{proof}

We will now prove the desired statement:

\begin{lem}\label{Lemma:reduction to generic smooth case}
To prove Theorem \ref{Theorem:tubular neighborhood compact support} and Theorem \ref{Theorem:tubular neighborhood direct image},
it suffices to prove Lemma \ref{Lemma:key lemma, tubular neighborhood}.
\end{lem}
\begin{proof}
We suppose that
Lemma \ref{Lemma:key lemma, tubular neighborhood} holds.
By Lemma \ref{Lemma:reduction to V(f) case compact support} and Lemma \ref{Lemma:reduction to V(f) direct image},
it suffices to prove that
$\textbf{P}_c(i)$ and $\textbf{P}(i)$ hold for every $i$.
Let us show the assertions by induction on $i$.
The assertions $\textbf{P}_c(-1)$ and $\textbf{P}(-2)$ hold trivially.
Assume that $\textbf{P}_c(i-1)$ and $\textbf{P}(i-1)$ hold.
Let $\mathcal{X}$ be a separated scheme of finite type over $\O$ and
$\mathcal{Z} \hookrightarrow \mathcal{X}$
a closed immersion defined by one global section $f \in \O_{\mathcal{X}}(\mathcal{X})$.
We shall show
that $\textbf{P}_c(i)$
and
$\textbf{P}(i)$ hold for $\mathcal{X}$ and
$\mathcal{Z}$.
As in the proof of Lemma \ref{Lemma:reduction to V(f) case compact support}, we may assume that $\mathcal{X}$ is integral.

First, we prove the assertions in the case where $K$ is of characteristic zero.
By \cite[Theorem 4.1]{deJong96},
there is an integral alteration
$Y \to \mathcal{X}_K$
such that $Y$ is smooth over $K$.
By Nagata's compactification theorem,
there is a proper surjective
morphism
$\mathcal{Y} \to \mathcal{X}$
such that
$\mathcal{Y}_K \cong Y$ over
$\mathcal{X}_K$ and $\mathcal{Y}$ is integral.
By the induction hypothesis,
Corollary \ref{Corollary:cohomological descent spectral sequence compact support},
and Corollary \ref{Corollary:cohomological descent direct image},
it suffices to prove
$\textbf{P}_c(i)$
and
$\textbf{P}(i)$ for $\mathcal{Y}$ and
$\mathcal{Z} \times_{\mathcal{X}} \mathcal{Y}$.
(As we have already seen in the proof of Corollary
\ref{Corollary:cohomological descent spectral sequence compact support}, the morphism
$d(\widehat{\mathcal{Y}}) \to d(\widehat{\mathcal{X}})$
is proper and surjective.)
Therefore, we may assume that $\mathcal{X}_K$ is smooth over $K$.
Let
\[
f \colon \mathcal{X} \to \Spec \O[T]
\]
be the morphism defined by $T \mapsto f$.
Since $K$ is of characteristic zero,
there is an open dense subset $W \subset \Spec K[T]$
such that $f_K$ is smooth over $W$.
It follows from \cite[Proposition 1.9.6]{Huber96} that there exists
an open subset $V \subset \B(1)$ whose complement
consists of finitely many $K$-rational points of $\B(1)$
such that 
$d(f) \colon d(\widehat{\mathcal{X}}) \to \B(1)$ is smooth over $V$.
Thus
$\textbf{P}_c(i)$
and $\textbf{P}(i)$ hold for $\mathcal{X}$ and
$\mathcal{Z}$ since we suppose that
Lemma \ref{Lemma:key lemma, tubular neighborhood} holds.

Let us now suppose that $K$ is of characteristic $p >0$.
As above, let
$
f \colon \mathcal{X} \to \Spec \O[T]
$
be the morphism defined by $T \mapsto f$.
If $f$ is not dominant,
then $\textbf{P}_c(i)$
and
$\textbf{P}(i)$ hold trivially for
$\mathcal{X}$ and
$\mathcal{Z}$.
Thus we may assume that $f$ is dominant.
By applying \cite[Theorem 4.1]{deJong96} to the underlying reduced subscheme of
$
\mathcal{X} \times_{\Spec \O[T]} \Spec K(T^{1/p^{\infty}}),
$
where $K(T^{1/p^{\infty}})$ is the perfection of $K(T)$,
we find an alteration
\[
g_K \colon Y \to \mathcal{X} \times_{\Spec \O[T]} \Spec K(T^{1/p^{N}})
\]
such that $Y$ is integral and smooth over $K(T^{1/p^{N}})$
for some integer $N \geq 0$.
By Nagata's compactification theorem,
there is a proper surjective morphism
\[
g \colon \mathcal{Y} \to \mathcal{X} \times_{\Spec \O[T]} \Spec \O[T^{1/p^{N}}]
\]
whose base change to $\Spec K(T^{1/p^{N}})$ is isomorphic to $g_K$.
As above,
it suffices to
prove
$\textbf{P}_c(i)$
and
$\textbf{P}(i)$ for $\mathcal{Y}$ and
$\mathcal{Z} \times_{\mathcal{X}} \mathcal{Y}$.

Let $f'$ be the image of $T^{1/p^{N}} \in \O[T^{1/p^{N}}]$ in
$\O_\mathcal{Y}(\mathcal{Y})$
and
let $\mathcal{Z}' \hookrightarrow \mathcal{Y}$ be the closed subscheme defined by $f'$.
Then we have $(f')^{p^N}=f$, where the image of
$f$ in $\O_\mathcal{Y}(\mathcal{Y})$ is also denoted by $f$.
Hence the closed immersion
$
d(\widehat{\mathcal{Z}}') \hookrightarrow d((\mathcal{Z} \times_{\mathcal{X}} \mathcal{Y})^\wedge)
$
is a homeomorphism and
we have
\[
S(\mathcal{Z}', \epsilon)=S(\mathcal{Z} \times_{\mathcal{X}} \mathcal{Y}, \epsilon^{p^N})
\quad 
\text{and}
\quad
T(\mathcal{Z}', \epsilon)=T(\mathcal{Z} \times_{\mathcal{X}} \mathcal{Y}, \epsilon^{p^N})
\]
for every $\epsilon \in \vert K^\times \vert$.
Thus,
it suffices to prove that
$\textbf{P}_c(i)$
and
$\textbf{P}(i)$ hold for $\mathcal{Y}$ and
$\mathcal{Z}'$; see \cite[Proposition 2.3.7]{Huber96}.
By the construction,
the generic fiber of the morphism
$
\mathcal{Y} \to \Spec \O[T]
$
defined by $T \mapsto f'$ is smooth over $K(T)$.
Therefore,
as in the case of characteristic zero,
Lemma \ref{Lemma:key lemma, tubular neighborhood}
implies
$\textbf{P}_c(i)$
and
$\textbf{P}(i)$ for $\mathcal{Y}$ and
$\mathcal{Z}'$.

The proof of Lemma \ref{Lemma:reduction to generic smooth case} is complete.
\end{proof}

\subsection{Proof of the key case}\label{Subsection:Proof of the key case}

In this subsection, we prove Lemma \ref{Lemma:key lemma, tubular neighborhood} and finish the proofs of Theorem \ref{Theorem:tubular neighborhood compact support} and Theorem \ref{Theorem:tubular neighborhood direct image}.

\begin{proof}[\bf  Proof of Lemma \ref{Lemma:key lemma, tubular neighborhood}]
We may assume that $\mathcal{X}$ is flat over $\Spec \O$.
Then $\mathcal{X}$ is of finite presentation over $\Spec \O$ by \cite[Premi\`ere partie, Corollaire 3.4.7]{Raynaud-Gruson}.
By
Proposition \ref{Proposition:constant sheaves adapted}
and
Theorem \ref{Theorem:uniform local constancy},
there exists an element $\epsilon_0 \in \vert K^\times \vert$
with $\epsilon_0 \leq \epsilon_1$
such that,
for all
$a, b \in \vert K^{\times} \vert$ with $a < b \leq \epsilon_0$
and every positive integer $n$ invertible in $\O^\times$,
there exists a finite \'etale morphism
\[
h \colon \B(c, d) \to \B(a, b)
\]
such that $h$ is a composition of finite Galois \'etale morphisms
and
the pull-back
\[
h^{*}((R^id(f)_!\Z/n\Z)\vert_{\B(a, b)})
\]
is a constant sheaf associated with a finitely generated $\Z/n\Z$-module for every $i$.
We shall show that $\epsilon_0$ satisfies the desired properties.
Let $n$ be a positive integer invertible in $\O$
and $\epsilon \in \vert K^{\times} \vert$ an element 
with
$\epsilon \leq \epsilon_0$.

(1)
We have
\[
S(\mathcal{Z}, \epsilon)\backslash d(\widehat{\mathcal{Z}}) = d(f)^{-1}(\D(\epsilon) \backslash \{ 0 \}) \quad \text{and} \quad T(\mathcal{Z}, \epsilon_0)\backslash T(\mathcal{Z}, \epsilon) = d(f)^{-1}(\B(\epsilon_0) \backslash \B(\epsilon)).
\]
Keeping the base change theorem \cite[Theorem 5.4.6]{Huber96} for $Rd(f)_!$ in mind,
we have the following spectral sequences by \cite[Remark 5.5.12 i)]{Huber96}:
\[
E^{i, j}_{2}=H^i_c(\D(\epsilon) \backslash \{ 0 \}, (R^jd(f)_!\Z/n\Z)\vert_{\D(\epsilon) \backslash \{ 0 \}}) \Rightarrow H^{i+j}_c(S(\mathcal{Z}, \epsilon)\backslash d(\widehat{\mathcal{Z}}), \Z/n\Z),
\]
\[
E^{i, j}_{2}=H^i_c(\B(\epsilon_0) \backslash \B(\epsilon), (R^jd(f)_!\Z/n\Z)\vert_{\B(\epsilon_0) \backslash \B(\epsilon)}) \Rightarrow H^{i+j}_c(T(\mathcal{Z}, \epsilon_0)\backslash T(\mathcal{Z}, \epsilon), \Z/n\Z).
\]
Hence the assertion (1) follows from
Lemma \ref{Lemma:calculation of cohomology}.

(2)
We have
\[
T(\mathcal{Z}, \epsilon) = d(f)^{-1}(\B(\epsilon)) \quad \text{and} \quad S(\mathcal{Z}, \epsilon)^{\circ} = d(f)^{-1}(\D(\epsilon)^{\circ}).
\]
Be the Leray spectral spectral sequences,
it suffices to prove that
the restriction map
\[
H^i(\B(\epsilon), R^jd(f)_*\Z/n\Z) \to H^i(\D(\epsilon)^{\circ}, R^jd(f)_*\Z/n\Z)
\]
is an isomorphism for all $i, j$.
Let $\epsilon' \in \vert K^\times \vert$ be an element with $\epsilon' < \epsilon$.
Then
$
\{ \B(\epsilon', \epsilon), \D(\epsilon)^{\circ} \}
$
is an open covering of $\B(\epsilon)$.
By the \v{C}ech-to-cohomology spectral sequences,
it is enough to prove that the restriction map
\[
H^i(\B(\epsilon', \epsilon), R^jd(f)_*\Z/n\Z) \to H^i(\D(\epsilon)^{\circ} \cap \B(\epsilon', \epsilon), R^jd(f)_*\Z/n\Z)
\]
is an isomorphism for all $i, j$.

The inverse image
$d(f)^{-1}(\B(\epsilon', \epsilon))$
has finitely many connected components.
It is enough to show that,
for every connected component
$W \subset d(f)^{-1}(\B(\epsilon', \epsilon))$
and
the restriction
$
g \colon W \to \B(\epsilon', \epsilon)
$
of $d(f)$,
the restriction map
\[
H^i(\B(\epsilon', \epsilon), R^{j}g_*\Z/n\Z) \to H^i(\D(\epsilon)^{\circ} \cap \B(\epsilon', \epsilon), R^{j}g_*\Z/n\Z)
\]
is an isomorphism for all $i, j$.
The morphism $g$ is of pure dimension $N$ for some integer $N \geq 0$.
(See \cite[Section 1.8]{Huber96} for
the definition of the dimension of a morphism of adic spaces.)
Since
$g$
is smooth
and
$R^ig_!\Z/n\Z$
is a locally constant constructible sheaf of $\Z/n\Z$-modules,
Poincar\'e duality \cite[Corollary 7.5.5]{Huber96} implies that
\[
R^{j}g_*(\Z/n\Z(N))\cong (R^{2N-j}g_!\Z/n\Z)^\vee
\]
for every $j$,
where $(N)$ denotes the Tate twist and $()^\vee$ denotes the $\Z/n\Z$-dual.
(Here we use the fact that $\Z/n\Z$ is an injective $\Z/n\Z$-module.)
The right hand side satisfies the assumption of Lemma \ref{Lemma:calculation of cohomology} (1), and hence the assertion follows from the lemma.

(3)
We have
$S(\mathcal{Z}, \epsilon)=d(f)^{-1}(\D(\epsilon))$.
Let
\[
d(f)' \colon S(\mathcal{Z}, \epsilon) \to \D(\epsilon)
\]
be the base change of $d(f)$ to $\D(\epsilon)$.
We write
\[
\mathcal{F}_j:=R^jd(f)'_*\Z/n\Z.
\]
We claim that restriction map
$
H^i(\D(\epsilon), \mathcal{F}_j) \to  H^i(\{ 0 \}, \mathcal{F}_j\vert_{ \{ 0 \}})
$
is an isomorphism for all $i, j$.
Since we have by \cite[Example 2.6.2]{Huber96}
\[
H^i(\{ 0 \}, \mathcal{F}_j\vert_{ \{ 0 \}})= 
\begin{cases}
H^j(d(\widehat{\mathcal{Z}}), \Z/n\Z) \quad &(i=0) \\
0 \quad &(i \neq 0),
\end{cases}
\]
the claim and the Leray spectral spectral sequence for $d(f)'$
imply the assertion (3).

We prove the claim.
Since $\D(\epsilon)$ and $\{ 0 \}$ are proper over $\Spa(K, \O)$,
we have
\[
H^i(\D(\epsilon), \mathcal{F}_j)=H^i_c(\D(\epsilon), \mathcal{F}_j)
\quad
\text{and}
\quad
H^i(\{ 0 \}, \mathcal{F}_j\vert_{ \{ 0 \}})=H^i_c(\{ 0 \}, \mathcal{F}_j\vert_{ \{ 0 \}}),
\]
and hence
it suffices to prove that
$
H^i_c(\D(\epsilon)\backslash \{ 0 \}, \mathcal{F}_j\vert_{\D(\epsilon)\backslash \{ 0 \}}) = 0
$
for all $i, j$.
Moreover,
by \cite[Proposition 5.4.5 ii)]{Huber96},
it suffices to prove that,
for any $\epsilon' \in \vert K^\times \vert$ with $\epsilon' < \epsilon$,
we have
$
H^i_c(\D(\epsilon)\cap \B(\epsilon', \epsilon), \mathcal{F}_j\vert_{\D(\epsilon)\cap \B(\epsilon', \epsilon)}) = 0
$
for all $i, j$.

Let $\epsilon' \in \vert K^\times \vert$ be an element with $\epsilon' < \epsilon$.
Let $W$ and $g$ be as in the proof of (2).
Let
\[
g' \colon g^{-1}(\D(\epsilon)\cap \B(\epsilon', \epsilon)) \to \D(\epsilon)\cap \B(\epsilon', \epsilon)
\]
be the base change of $g$.
It suffices to prove that
$
H^i_c(\D(\epsilon)\cap \B(\epsilon', \epsilon), R^jg'_*\Z/n\Z)
= 0
$
for all $i, j$.
By the base change theorem \cite[Theorem 5.4.6]{Huber96} for $Rg_!$,
we have
\[
(R^{j}g_!\Z/n\Z)\vert_{\D(\epsilon)\cap \B(\epsilon', \epsilon)} \cong R^jg'_!\Z/n\Z.
\]
In particular the right hand side is a locally constant constructible sheaf of $\Z/n\Z$-modules.
As in the proof of (2),
Poincar\'e duality \cite[Corollary 7.5.5]{Huber96} for $g'$ then implies that
\[
R^{j}g'_*(\Z/n\Z(N))\cong (R^{2N-j}g'_!\Z/n\Z)^\vee,
\]
and the assertion follows from Lemma \ref{Lemma:calculation of cohomology} (1).

(4)
Similarly as above,
it suffices to prove that the restriction map
\[
H^i(\D(\epsilon), \mathcal{F}_j) \to H^i(\D(\epsilon)^{\circ}, \mathcal{F}_j)
\]
is an isomorphism for all $i, j$.
Let $\epsilon' \in \vert K^\times \vert$ be an element with $\epsilon' < \epsilon$.
As in the proof of (2),
it suffices to prove that the restriction map
\[
H^i(\D(\epsilon)\cap \B(\epsilon', \epsilon), \mathcal{F}_j) \to
H^i(\D(\epsilon)^{\circ}\cap \B(\epsilon', \epsilon), \mathcal{F}_j)
\]
is an isomorphism for all $i, j$.
By the proof of (3),
the sheaf
$\mathcal{F}_j\vert_{\D(\epsilon) \cap \B(\epsilon', \epsilon)}$
is a locally constant constructible sheaf of $\Z/n\Z$-modules.
Hence the assertion follows from the proof of \cite[Lemma 2.5]{Huber98b}.
(In \cite{Huber98b}, the characteristic of the base field is always assumed to be zero.
However \cite[Lemma 2.5]{Huber98b} holds in positive characteristic without changing the proof.)

The proof of Lemma \ref{Lemma:key lemma, tubular neighborhood} is complete.
\end{proof}

Theorem \ref{Theorem:tubular neighborhood compact support} and Theorem \ref{Theorem:tubular neighborhood direct image} now follow from Lemma \ref{Lemma:key lemma, tubular neighborhood} and Lemma \ref{Lemma:reduction to generic smooth case}.

\appendix

\section{Finite \'etale coverings of annuli}\label{Appendix:finite etale coverings of annuli}

In this appendix,
we prove Theorem \ref{Theorem:split into annuli} and Theorem \ref{Theorem:trivialization of tame sheaf}.
We retain the notation of Section \ref{Subsection:Tame sheaves on annuli}.
In particular,
we fix an algebraically closed complete non-archimedean field $K$ with ring of integers $\O$.
We will follow the methods given in Ramero's paper \cite{Ramero05}.

Following \cite{Ramero05}, we will use the following notation in this appendix.
Recall that for a morphism of finite type
$
\Spa(A, A^+) \to \Spa(K, \O),
$
where $(A, A^+)$ is a complete affinoid ring,
the ring $A^+$ coincides with the ring $A^{\circ}$ of power-bounded elements of $A$; see \cite[Lemma 4.4]{Huber94} and \cite[Section 1.2]{Huber96}.
We often omit $A^+$ and abbreviate $\Spa(A, A^+)$ to
$
\Spa(A).
$
If $A$ is reduced,
then $A^{\circ}$ is topologically of finite type over $\O$,
i.e.\
$
A^{\circ}\cong \O \langle T_1, \dotsc, T_n \rangle/I
$
for some ideal $I \subset \O \langle T_1, \dotsc, T_n \rangle$ by \cite[6.4.1, Corollary 4]{BGR}.
Let $\mathfrak{m} \subset \O$ be the maximal ideal and 
let $\kappa:=\O/\mathfrak{m}$ be the residue field.
The quotient
\[
A^{\sim}:= A^{\circ}/\mathfrak{m}A^{\circ}
\]
is a finitely generated algebra over $\kappa$.
We note that the ideal
$\mathfrak{m}A^{\circ}$
coincides with the set of topologically nilpotent elements of $A$.
In particular, the ring $A^{\sim}$ is reduced.

\subsection{Open annuli in the unit disc}\label{Subsection:Open annuli in the unit disc}

We recall some basic properties of
open annuli in the unit disc $\B(1)=\Spa(K \langle T \rangle)$.

Let $a, b \in \vert K^{\times} \vert$ be elements with
$a \leq b \leq 1$.
Recall that we defined
\begin{align*}
\B(a, b)&:= \{ x \in \B(1) \, \vert \, a \leq  \vert T(x) \vert \leq b \} \\
&:= \{ x \in \B(1) \, \vert \, \vert \varpi_a(x) \vert \leq  \vert T(x) \vert \leq \vert \varpi_b(x) \vert \},
\end{align*}
where $\varpi_a, \varpi_b \in K^\times$ are elements such that
$a=\vert \varpi_a \vert$ and $b=\vert \varpi_b \vert$.
We have
\[
\B(a, b) \cong \Spa (K\langle T, T_a, T_b \rangle/(T_aT-\varpi_a, T-\varpi_bT_b))
\]
as an adic space over $\B(1)$.
The adic space $\B(a, b)$ is isomorphic to $\B(a/b, 1)$ as an adic space over $\Spa(K)$.
We write
$
A(a, b):=\O_{\B(1)}(\B(a, b)).
$

We will focus on the following points of
the unit disc $\B(1)$.
Let
$r \in \vert K^{\times} \vert$
be an element with $r \leq 1$.

\begin{itemize}
    \item Let
\[
\eta(r)^{\flat} \colon K \langle T \rangle \to \vert K^{\times} \vert, \quad \sum_{i \geq 0}a_i T^i \mapsto \max_{i \geq 0}\{ \vert a_i \vert r^i \}
\]
be the Gauss norm of radius $r$ centered at $0$.
The corresponding point $\eta(r)^{\flat} \in \B(1)$ is denoted by the same letter.
    \item Let $\langle \delta \rangle$ be an infinite cyclic group with generator $\delta$.
We equip
$\vert K^{\times} \vert \times \langle \delta \rangle$ with a total order
such that
\[
(s, \delta^m) < (t, \delta^n) \Longleftrightarrow s < t, \quad \text{or} \quad s=t \quad \text{and} \quad m>n.
\]
So we have
$(1, \delta)<(1, 1)$
and
$(s, 1) < (1, \delta)$ for every $s \in \vert K^{\times} \vert$ with $s < 1$.
We identify $\delta$ with $(1, \delta)$ and $r$ with $(r, 1)$.
The valuation
\[
\eta(r) \colon K \langle T \rangle \to \vert K^{\times} \vert \times \langle \delta \rangle, \quad \sum_{i \geq 0}a_i T^i \mapsto \max_{i \geq 0}\{ \vert a_i \vert r^i\delta^i \}
\]
gives a point $\eta(r) \in  \B(1)$.
The point $\eta(r)$ is a specialization of $\eta(r)^{\flat}$,
i.e.\ we have
\[
\eta(r) \in \overline{\{ \eta(r)^{\flat} \}}.
\]
    \item Similarly, if $r < 1$,
the valuation
\[
\eta(r)' \colon K \langle T \rangle  \to \vert K^{\times} \vert \times \langle \delta \rangle, \quad \sum_{i \geq 0}a_i T^i \mapsto \max_{i \geq 0}\{ \vert a_i \vert r^i\delta^{-i} \}
\]
gives a point $\eta(r)' \in  \B(1)$.
The point $\eta(r)'$ is a specialization of $\eta(r)^{\flat}$.
\end{itemize}

We can use the points $\eta(r)$ and $\eta(r)'$ to describe the closure of an annulus $\B(a, b)$.

\begin{ex}\label{Example:complement closure}
Let $a, b \in \vert K^{\times} \vert$ be elements with
$a \leq b \leq 1$.
\begin{enumerate}
    \item For an element $r \in \vert K^{\times} \vert$
with $a \leq r \leq b$,
we have
$\eta(r)^{\flat} \in \B(a, b)$.
    If $a < r \leq b$,
    we have
$\eta(r) \in \B(a, b)$.
Similarly,
if $a \leq r < b$,
we have
$\eta(r)' \in \B(a, b)$.
    \item We assume that $a \leq b < 1$.
    Let $\B(a, b)^c$
be the closure of $\B(a, b)$ in $\B(1)$.
Then we have
$
\B(a, b)^c \backslash \B(a, b)= \{ \eta(a), \eta(b)' \}.
$
In particular, the complement $\B(a, b)^c \backslash \B(a, b)$ consists of two points.
    \item We have some kind of converse to (2).
We define
$
\D(1):=\{ x \in \B(1) \, \vert \, \vert T(x) \vert < 1 \},
$
which is a closed subset of $\B(1)$.
Let $X \subset \B(1)$ be a connected affinoid open subset contained in $\D(1)$.
Let $X^c$ be the closure of $X$ in $\B(1)$.
If the complement
$X^c \backslash X$ consists of two points,
then there exists an isomorphism
\[
X \cong \B(a, 1)
\]
of adic spaces over $\Spa(K)$ for some element $a \in \vert K^{\times} \vert$ with
$a \leq 1$.
This can be proved by using \cite[9.7.2, Theorem 2]{BGR}.
\end{enumerate}
\end{ex}

We recall the following example from \cite{Ramero05}, which is useful to study finite \'etale coverings of $\B(a, b)$.

\begin{ex}[{\cite[Example 2.1.12]{Ramero05}}]\label{Example:morphism defined by Ramero}
We assume that $a < b$.
Let
\[
\Psi \colon \B(a, b)=\Spa(A(a, b)) \to \B(1)=\Spa(K \langle S \rangle)
\]
be the morphism over $\Spa(K)$ defined by the following homomorphism
\[
\psi \colon \O \langle S \rangle \to \O\langle T_a, T_b \rangle/(T_aT_b-\varpi_a/\varpi_b) \cong A(a, b)^{\circ}, \quad S \mapsto T_a+T_b.
\]
The homomorphism $\psi$ makes $A(a, b)^{\circ}$ into a free
$\O \langle S \rangle$-module of rank $2$.
\end{ex}

\begin{rem}\label{Remark:finite etale morphism}
In the rest of this section,
we shall study finite \'etale coverings of
$\B(a, b)$.
We recall the following fact from \cite[Example 1.6.6 ii)]{Huber96}, which we will use freely:
Let $X$ be an affinoid adic space of finite type over $\Spa(K, \O)$.
Let $Y \to X$ be a finite \'etale morphism of adic spaces.
Then $Y$ is affinoid and the induced morphism
$
\Spec \O_Y(Y) \to \Spec \O_X(X)
$
of schemes
is finite and \'etale.
This construction gives an equivalence of categories between the category of adic spaces which are finite and \'etale over $X$ and the category of schemes which are finite and \'etale over $\Spec \O_X(X)$.
\end{rem}

In the rest of this subsection,
we give two lemmas about the connected components of a finite \'etale covering of $\B(a, b)$.

\begin{lem}\label{Lemma:the number of connected components}
We assume that $a < b$.
Let $f \colon X \to \B(a, b)$
be a finite \'etale morphism of adic spaces.
For every $t \in \vert K^{\times} \vert$ with $a/b < t^2 <1$,
there exists
an element
$s_0 \in \vert K^{\times} \vert$ with
$t< s_0 \leq 1$
such that
every connected component of
$f^{-1}(\B(a/s_0, s_0 b))$
remains connected after restricting to $\B(a/s, s b)$
for every $s \in \vert K^{\times} \vert$ with $t < s \leq s_0$.
\end{lem}
\begin{proof}
The number of the connected components of
$f^{-1}(\B(a/s, s b))$
increases with decreasing $s$
and
it is bounded above by the degree of $f$ (i.e.\ the rank of $\O_X(X)$ as an $A(a, b)$-module)
for every $s \in \vert K^{\times} \vert$ with $a/b < s^2 \leq 1$.
The assertion follows from these properties.
\end{proof}

\begin{lem}\label{Lemma:the number of prime ideals}
We assume that $a < b$.
Let $f \colon X \to \B(a, b)$
be a finite \'etale morphism of adic spaces.
We write
$B:=\O_X(X)$
and consider the composition
\[
\O \langle S \rangle \overset{\psi}{\to} A(a, b)^{\circ} \to B^\circ,
\]
where $\psi$ is the homomorphism defined in Example \ref{Example:morphism defined by Ramero} and the second homomorphism is the one induced by $f$.
Let
$
\{ \mathfrak{q}_1, \dotsc, \mathfrak{q}_n  \} \subset \Spec B^{\circ}
$
be the set of the prime ideals of $B^{\circ}$ lying above the maximal ideal
$
\mathfrak{m}\O \langle S \rangle+S \O \langle S \rangle \subset \O \langle S \rangle.
$
Then,
for every $t \in \vert K^{\times} \vert$ with $a/b < t^2 <1$,
the adic space $f^{-1}(\B(a/t, t b))$ has at least $n$ connected components.
\end{lem}
\begin{proof}
This lemma is proved in the proof of \cite[Theorem 2.4.3]{Ramero05}.
We recall the arguments for the reader's convenience.

We define $g$ as the composition
\[
g \colon X \overset{f}{\to} \B(a, b) \overset{\Psi}{\to} \B(1)=\Spa (K \langle S \rangle).
\]
Let $t \in  \vert K^\times \vert$ be an element with $a/b < t^2 <1$.
We define
$
\B(t):= \{ x \in \B(1) \, \vert \, \vert S(x) \vert \leq t \}.
$
Since
$\Psi^{-1}(\B(t))= \B(a/t, tb)$,
it is enough to show that
$g^{-1}(\B(t))$
has at least $n$ connected components.
The ring
$\O[[S]]$
is a Henselian local ring.
Since $B^{\circ}$ is a free
$\O \langle S \rangle$-module of finite rank by \cite[Proposition 2.3.5]{Ramero05},
we have a decomposition
\[
B^{\circ} \otimes_{\O \langle S \rangle} \O[[S]] \cong R_1 \times \cdots \times R_n,
\]
where $R_1, \dotsc, R_n$ are local rings, which are free $\O[[S]]$-modules of finite rank.
Since the natural homomorphism
$
\O \langle S \rangle=\O_{\B(1)}(\B(1))^{\circ} \to \O_{\B(t)}(\B(t))^{\circ}
$
factors through $\O \langle S \rangle \to \O[[S]]$,
we have
\[
g^{-1}(\B(t)) \cong \Spa(B_1) \coprod \cdots \coprod \Spa(B_n),
\]
where $B_i:=R_i \otimes_{\O[[S]]} \O_{\B(t)}(\B(t))$.
This proves our claim since $\Spa(B_i)$ is non-empty for every $1 \leq i \leq n$.
\end{proof}

\subsection{Discriminant functions and finite \'etale coverings of open annuli}\label{Subsection:Discriminant functions and finite etale coverings of open annuli}

Let $a, b \in \vert K^{\times} \vert$
be elements
with
$a < b \leq 1$.
Let
$
f \colon X \to \B(a, b)
$
be a finite \'etale morphism of adic spaces.
Let us briefly recall the definition of the \textit{discriminant function}
\[
\delta_f \colon [-\log b, -\log a] \to \R_{\geq 0}
\]
associated with $f$ following \cite{Ramero05}.

Let $\O^+_X$ be the subsheaf of $\O_X$
defined by
\[
\O^+_X(U)=\{ g \in \O_X(U) \, \vert \, \vert g(x) \vert \leq 1 \  \text{for every} \ x \in U \}
\]
for every open subset $U \subset X$.
For an element
$r \in \vert K^{\times} \vert$
with $a \leq r \leq b$,
let
\[
\mathscr{A}(r)^{\flat}:=(f_* \O^+_X)_{\eta(r)^{\flat}} 
\]
be the stalk of $f_* \O^+_X$ at the point $\eta(r)^{\flat}$.
The maximal ideal of the stalk
$\O_{\B(a, b), \eta(r)^{\flat}}$
at $\eta(r)^{\flat}$
is zero,
in other words, we have
$
\O_{\B(a, b), \eta(r)^{\flat}}=k(\eta(r)^{\flat}).
$
Hence
there is a natural homomorphism
$
k(\eta(r)^{\flat})^+ \to \mathscr{A}(r)^{\flat}.
$
By applying \cite[Proposition 2.3.5]{Ramero05} to
the restriction $f^{-1}(\B(r, r)) \to \B(r, r)$ of $f$,
we see that $\mathscr{A}(r)^{\flat}$ is a free
$k(\eta(r)^{\flat})^+$-module of finite rank.
Then we can define the valuation
\[
v_{\eta(r)^{\flat}}(\mathfrak{d}^{\flat}_f(r)) \in \R_{> 0}
\]
of the discriminant
$\mathfrak{d}^{\flat}_f(r) \in k(\eta(r)^{\flat})^+$
of $\mathscr{A}(r)^{\flat}$ over
$k(\eta(r)^{\flat})^+$ (in the sense of \cite[Section 2.1]{Ramero05})
and we define
\[
\delta_f \colon [-\log b, -\log a] \cap -\log \vert K^\times \vert \to \R_{\geq 0}, \quad -\log r \mapsto - \log( v_{\eta(r)^{\flat}}(\mathfrak{d}^{\flat}_f(r))) \in \R_{\geq 0}.
\]
See \cite[2.3.12]{Ramero05} for details.

\begin{thm}[Ramero {\cite[Theorem 2.3.25]{Ramero05}}]\label{Theorem:discriminant function}
The function $\delta_f$
extends uniquely to a continuous, piecewise linear, and convex function
\[
\delta_f \colon [-\log b, -\log a]\to \R_{\geq 0}.
\]
Moreover, the slopes of $\delta_f$ are integers.
\end{thm}
\begin{proof}
See \cite[Theorem 2.3.25]{Ramero05}.
\end{proof}

Many basic properties of the discriminant function $\delta_f$ were studied in detail in \cite{Ramero05}.
Here, we are interested in the case where $\delta_f$ is linear.

\begin{prop}[Ramero \cite{Ramero05}]\label{Proposition:reduction ordinary double point}
Let
$
f \colon X=\Spa(B, B^{\circ}) \to \B(a, b)
$
be a finite \'etale morphism of adic spaces
with a complete affinoid ring $(B, B^{\circ})$.
Define $g$ as the composition
\[
g \colon X \overset{f}{\to} \B(a, b) \overset{\Psi}{\to} \B(1)=\Spa (K \langle S \rangle),
\]
where $\Psi$ is the morphism defined in Example \ref{Example:morphism defined by Ramero}.
The map $g$ induces a homomorphism
$
\O \langle S \rangle \to B^\circ.
$
Let us suppose that the following two conditions hold:
\begin{itemize}
    \item  There is only one prime ideal $\mathfrak{q} \subset B^{\circ}$
lying above the maximal ideal
$
\mathfrak{m}\O \langle S \rangle+S \O \langle S \rangle \subset \O \langle S \rangle.
$
    \item The discriminant function $\delta_f$ is linear.
\end{itemize}
Then the inverse image
$g^{-1}(\eta(1))$
consists of two points, or equivalently,
the inverse images
$f^{-1}(\eta(a)')$ and $f^{-1}(\eta(b))$ both consist of one point.
Moreover, the closed point
\[
x \in \Spec B^{\sim}=\Spec B^{\circ}/\mathfrak{m}B^\circ
\]
corresponding to the prime ideal $\mathfrak{q}$ is an ordinary double point.
\end{prop}

\begin{proof}
The first assertion is a consequence of \cite[(2.4.4) in the proof of Theorem 2.4.3]{Ramero05}.
The assumption that the characteristic of the base field is zero in \textit{loc.\ cit.} is not needed here.
Moreover, the morphism $f$ need not be Galois.

The second assertion is claimed in \cite[Remark 2.4.8]{Ramero05} (at least when $K$ is of characteristic zero) without proof.
Indeed, the hard parts of the proof were already done in \cite{Ramero05}.
We shall explain how to use the results in \textit{loc.\ cit.} to deduce the second assertion.

Before giving the proof of the second assertion,
let us prepare some notation.
We write $A:= K \langle S \rangle$.
For the points $\eta(1), \eta(1)^\flat \in \B(1)$,
we write
\[
(k(1), k(1)^+):=(k(\eta(1)), k(\eta(1))^+) \quad \text{and} \quad (k(1^\flat), k(1^\flat)^+):=(k(\eta(1)^\flat), k(\eta(1)^\flat)^+).
\]
Note that
$
\O_{\B(1), \eta(1)}=k(1)
$
and
$
\O_{\B(1), \eta(1)^{\flat}}=k(1^{\flat}).
$
We have a natural inclusion
\[
k(1)^+ \to k(1^{\flat})^+.
\]
The residue field
$k(1^{\flat})^+/\mathfrak{m}k(1^{\flat})^+$
of $k(1^{\flat})^+$
is naturally isomorphic to the field of fractions
$\kappa(S)$
of
$\kappa[S]$.
The image of $k(1)^+$ in $\kappa(S)$ is the localization
$\kappa[S]_{(S)}$ of
$\kappa[S]$ at the maximal ideal $(S) \subset \kappa[S]$.
More precisely, we have the following commutative diagram:
\[
\xymatrix{ A^\sim=A^\circ/\mathfrak{m}A^\circ \ar[r]_-{}  \ar[d]^-{\cong}& k(1)^+/\mathfrak{m}k(1)^+ \ar[r]_-{} \ar[d]^-{\cong} &  k(1^{\flat})^+/\mathfrak{m}k(1^{\flat})^+ \ar[d]^-{\cong}   \\
\kappa[S] \ar[r]_-{} & \kappa[S]_{(S)} \ar[r]_-{} & \kappa(S).
}
\]
Let
\[
\mathscr{B}(1)^{+}:=(g_* \O^+_X)_{\eta(1)} 
\]
be the stalk of $g_* \O^+_X$ at the point $\eta(1) \in \B(1)$.
We have a map
\[
i \colon B^{\circ}\otimes_{A^{\circ}}k(1)^+ \to \mathscr{B}(1)^{+}.
\]
The target and the source of $i$ are both free $k(1)^+$-modules of finite rank by \cite[Proposition 2.3.5 and Lemma 2.2.17]{Ramero05}, and clearly $i$ becomes an isomorphism after tensoring with $k(1)$.
In particular, the map $i$ is injective.
By \cite[Proposition 2.3.5]{Ramero05} again,
it follows that $i$ becomes an isomorphism after tensoring with $k(1^{\flat})^+$.
Consequently, we have inclusions
\[
B^{\circ}\otimes_{A^{\circ}}k(1)^+ \hookrightarrow
\mathscr{B}(1)^{+} \hookrightarrow B^{\circ}\otimes_{A^{\circ}} k(1^{\flat})^+,
\]
\[
B^{\circ}\otimes_{A^{\circ}}(k(1)^+/\mathfrak{m}k(1)^+) \hookrightarrow
\mathscr{B}(1)^{+}/\mathfrak{m}\mathscr{B}(1)^{+} \hookrightarrow B^{\circ}\otimes_{A^{\circ}} (k(1^{\flat})^+/\mathfrak{m}k(1^{\flat})^+).
\]

We shall prove the second assertion.
First, we prove that the ring 
$\mathscr{B}(1)^{+}/\mathfrak{m}\mathscr{B}(1)^{+}$
is integrally closed in
$B^{\circ}\otimes_{A^{\circ}} (k(1^{\flat})^+/\mathfrak{m}k(1^{\flat})^+)$.
Let
\[
(k(1)^{\wedge}, k(1)^{\wedge+}) \quad \text{and} \quad (k(1^\flat)^{\wedge}, k(1^\flat)^{\wedge+})
\]
be the completions with respect to the valuation topologies.
By \cite[Lemma 1.1.10 iii)]{Huber96}, we have
$
k(1)^{\wedge} \overset{\cong}{\to} k(1^\flat)^{\wedge}.
$
We have the following commutative diagram:
\[
\xymatrix{  \mathscr{B}(1)^{+} \otimes_{k(1)^+} k(1)^{\wedge+} \ar[r]_-{} \ar[d]^-{} &  B^{\circ}\otimes_{A^{\circ}} k(1^{\flat})^{\wedge+} \ar[d]^-{}   \\
\mathscr{B}(1)^{+} \otimes_{k(1)^+} k(1)^{\wedge} \ar[r]^-{\cong} & \mathscr{B}(1)^{+} \otimes_{k(1)^+} k(1^\flat)^{\wedge},
}
\]
where the vertical maps and the top horizontal map are injective.
By using \cite[Proposition 1.3.2 (iii)]{Ramero05},
we see that
$\mathscr{B}(1)^{+} \otimes_{k(1)^+} k(1)^{\wedge+}$
is integrally closed in
$\mathscr{B}(1)^{+} \otimes_{k(1)^+} k(1)^{\wedge}$,
and hence integrally closed in
$B^{\circ}\otimes_{A^{\circ}} k(1^{\flat})^{\wedge+}$.
This implies that
$\mathscr{B}(1)^{+}/\mathfrak{m}\mathscr{B}(1)^{+}$ is integrally closed in $B^{\circ}\otimes_{A^{\circ}} (k(1^{\flat})^+/\mathfrak{m}k(1^{\flat})^+)$.

By the assumptions, the ring
\[
R:=B^{\circ}\otimes_{A^{\circ}}(k(1)^+/\mathfrak{m}k(1)^+) \cong B^{\sim} \otimes_{\kappa[S]} \kappa[S]_{(S)}
\]
is the local ring 
of $\Spec B^{\sim}$ at the closed point $x \in \Spec B^{\sim}$.
Moreover, the ring
\[
B^{\circ}\otimes_{A^{\circ}} (k(1^{\flat})^+/\mathfrak{m}k(1^{\flat})^+) \cong B^{\sim} \otimes_{\kappa[S]} \kappa(S)
\]
is the total ring of fractions of $R$.
Hence the ring
$\mathscr{B}(1)^{+}/\mathfrak{m}\mathscr{B}(1)^{+}$ is the normalization of $R$.
By using \cite[(2.4.4) in the proof of Theorem 2.4.3]{Ramero05},
one can show that
there are exactly two maximal ideals of $\mathscr{B}(1)^{+}/\mathfrak{m}\mathscr{B}(1)^{+}$
and the length of
$\mathscr{B}(1)^{+}/\mathfrak{m}\mathscr{B}(1)^{+}$
as an $R$-module is one.
In other words, the closed point $x$ is an ordinary double point.

The proof of Proposition \ref{Proposition:reduction ordinary double point} is complete.
\end{proof}

We deduce the following result from Proposition \ref{Proposition:reduction ordinary double point},
which is used in the proof of Theorem \ref{Theorem:split into annuli} (in the case where $K$ is of positive characteristic).

\begin{prop}\label{Proposition:split into annuli}
Let
$
f \colon X \to \B(a, b)
$
be a finite \'etale morphism of adic spaces.
We assume that the discriminant function $\delta_f$ is linear.
Let $t \in \vert K^{\times} \vert$ be an element with $a/b < t^2 <1$.
Then
we have
\[
f^{-1}(\B(a/t, t b)) \cong \coprod^{n}_{i=1} \B(c_i, 1)
\]
for some elements $c_i \in \vert K^{\times} \vert$
with
$c_i < 1$ ($1 \leq i \leq n$).
\end{prop}

\begin{proof}
By Lemma \ref{Lemma:the number of connected components},
without loss of generality,
we may assume that every connected component
of $X$ remains connected after restricting to $\B(a/s, s b)$
for every $s \in \vert K^{\times} \vert$ with $t < s \leq 1$.
Let $X_1, \dotsc, X_m$ be the connected components of $X$
and let $f_i \colon X_i \to \B(a, b)$ be the restriction of $f$.
By Theorem \ref{Theorem:discriminant function},
each discriminant function $\delta_{f_i}$ associated with $f_i$
is a continuous, piecewise linear, and convex function.
Since
$
\delta_f= \sum^m_{i=1} \delta_{f_i},
$
it follows that $\delta_{f_i}$ is linear for every $i$.
Thus we may further assume that $X$ is connected.

Let $(B, B^\circ)$ be a complete affinoid ring such that
$X=\Spa(B, B^\circ)$.
Define $g$ as the composition
\[
g \colon X \overset{f}{\to} \B(a, b) \overset{\Psi}{\to} \B(1)=\Spa (K \langle S \rangle).
\]
By Lemma \ref{Lemma:the number of prime ideals},
there is only one prime ideal $\mathfrak{q} \subset B^{\circ}$
lying above the maximal ideal
$
\mathfrak{m}\O \langle S \rangle+S \O \langle S \rangle \subset \O \langle S \rangle.
$
Let
$
x \in \Spec B^{\sim}
$
be the closed point corresponding to the prime ideal $\mathfrak{q}$,
which is an ordinary double point by Proposition \ref{Proposition:reduction ordinary double point}.
Let
\[
\lambda \colon X=d(\Spf(B^{\circ})) \to \Spf(B^{\circ})
\]
be the specialization map associated with the formal scheme $\Spf(B^{\circ})$; see Section \ref{Subsection:Etale cohomology with compact support of adic spaces and nearby cycles}.
By the proof of \cite[Proposition 2.3]{BL85},
the interior $\lambda^{-1}(x)^{\circ}$ of the inverse image
$\lambda^{-1}(x)$
in $X$
is isomorphic to
the interior
$\D(d, 1)^{\circ}$
of
$\D(d, 1)$
in $\B(1)$
for some element $d \in \vert K^\times \vert$ with $d < 1$
as an adic space over $\Spa(K)$,
where
\[
\D(d, 1):=\{ x \in \B(1)=\Spa(K\langle T \rangle) \, \vert \, d < \vert T(x) \vert < 1 \} \subset \B(1).
\]
We fix such an isomorphism.
For every $s \in \vert K^{\times} \vert$ with $t \leq s < 1$,
we have
\[
f^{-1}(\B(a/s, sb))=g^{-1}(\B(s)) \subset g^{-1}(\D(1)^{\circ}) = \lambda^{-1}(x)^{\circ} \cong \D(d, 1)^{\circ} \subset \B(1).
\]
Thus we may consider $f^{-1}(\B(a/s, sb))$ as
a connected affinoid open subset of $\B(1)$.

We write $X_t:=f^{-1}(\B(a/t, tb))$.
Let $X^c_t$ be the closure of $X_t$ in $\B(1)$, which is contained in $g^{-1}(\D(1)^{\circ})$.
In view of Example \ref{Example:complement closure} (3),
to prove the assertion,
it suffices to prove that
the complement
$
X^{c}_t \backslash X_t
$
consists of exactly two points.
The map
$f$ induces a map
\[
f' \colon X^{c}_t \backslash X_t \to \B(a/t, tb)^{c} \backslash \B(a/t, tb)
=\{ \eta(a/t), \eta(tb)' \}.
\]
We prove that $f'$ is bijective.
Since $f$ is surjective and specializing by \cite[Lemma 1.4.5 ii)]{Huber96},
it follows that the map $f'$ is surjective.
To show that the map $f'$ is injective,
it suffices to prove the following claim:

\begin{claim}\label{Claim:inverse image consists one element}
The inverse images
$f^{-1}(\eta(a/t))$ and $f^{-1}(\eta(t b)')$
both consist of one point.
\end{claim}
\begin{proof}
Recall that we assume that $X$ remains connected after restricting to $\B(a/s, s b)$
for every $s \in \vert K^{\times} \vert$ with $t < s \leq 1$.
Thus,
by Lemma \ref{Lemma:the number of prime ideals} and Proposition \ref{Proposition:reduction ordinary double point},
the inverse images
$f^{-1}(\eta(a/s)')$ and $f^{-1}(\eta(s b))$ both consist of one point
for every $s \in \vert K^{\times} \vert$ with $t < s \leq 1$.
This fact implies that
\[
f^{-1}(\B(s_1b, s_2b))
\]
is connected for every $s_1, s_2 \in \vert K^{\times} \vert$ with
$t \leq s_1 < s_2 \leq 1$.
(Indeed, if it is not connected,
then there exist at least two points mapped to $\eta(s_2 b)$.)
By applying Lemma \ref{Lemma:the number of prime ideals} and Proposition \ref{Proposition:reduction ordinary double point} to
\[
f^{-1}(\B(tb, b)) \to \B(tb, b),
\]
we see that $f^{-1}(\eta(tb)')$ consists of one point.
The same arguments show that $f^{-1}(\eta(a/t))$ consists of one point.
\end{proof}

The proof of Proposition \ref{Proposition:split into annuli} is complete.
\end{proof}

\begin{rem}\label{Remark:LS}
Here we prove
Proposition \ref{Proposition:reduction ordinary double point} and Proposition \ref{Proposition:split into annuli}
in the context of adic spaces as in \cite{Ramero05}.
It is probably possible to prove these results by using the methods of \cite{Lutkebohmert93, LS05}.
\end{rem}

We now give a proof of Theorem \ref{Theorem:split into annuli}.

\begin{proof}[\bf  Proof of Theorem \ref{Theorem:split into annuli}]
Let $f \colon X \to \B(1)^{*}$ be a finite \'etale morphism.
Clearly, the discriminant functions on open annuli constructed in Theorem \ref{Theorem:discriminant function} can be glued to
a continuous, piecewise linear, and convex function
\[
\delta_f \colon [0, \infty) \to \R_{\geq 0}.
\]
Moreover, the slopes of $\delta_f$ are integers.
By \cite[Lemma 2.1.10]{Ramero05},
the function $\delta_f$ is bounded above by some positive real number (depending only on the degree of $f$).
It follows that there exists an element $\epsilon_0 \in \vert K^\times \vert$ with $\epsilon_0 \leq 1$ such that
the restriction of $\delta_f$
to $[-\log \epsilon_0, \infty)$ is constant.
Let $t \in \vert K^\times \vert$ be an element with $t < 1$.
We put $\epsilon := t\epsilon_0$.
Then,
for elements $a, b \in \vert K^{\times} \vert$
with
$a < b \leq \epsilon$,
we have
\[
f^{-1}(\B(a, b)) \cong \coprod^{n}_{i=1} \B(c_i, d_i)
\]
for some elements $c_i, d_i \in \vert K^{\times} \vert$
with
$c_i < d_i \leq 1$ ($1 \leq i \leq n$)
by Proposition \ref{Proposition:split into annuli}.
If $K$ is of characteristic zero,
after replacing $\epsilon$ by a smaller one,
we can easily show that the restriction
$
\B(c_i, d_i) \to \B(a, b)
$
of $f$ is a Kummer covering by using \cite[Claim 2.4.5]{Ramero05}.
(See also the proofs of \cite[Theorem 2.2]{Lutkebohmert93} and \cite[Theorem 2.4.3]{Ramero05}.)
\end{proof}

\subsection{Galois coverings and discriminant functions}\label{Subsection:Galois coverings and discriminant functions}

Let $a, b \in \vert K^{\times} \vert$
be elements
with
$a < b \leq 1$.
Let
\[
f \colon X=\Spa(B, B^\circ) \to \B(a, b)
\]
be 
a finite \'etale morphism of adic spaces with a complete affinoid ring $(B, B^\circ)$.
Let
\[
G:= \Aut(X/\B(a, b))^{\circ} \cong \Aut(B/A(a, b))
\]
be the opposite of the group of $\B(a, b)$-automorphisms on $X$, or equivalently, the group of $A(a, b)$-automorphisms of $B$.
We assume that $f$ is Galois, i.e.\ $A(a, b)$ coincides with the ring $B^G$ of $G$-invariants.
(This is equivalent to saying that the finite \'etale morphism
$\Spec B \to \Spec A(a, b)$ of schemes is Galois; see Remark \ref{Remark:finite etale morphism}.)
In this case, we call $G$ the Galois group of $f$.

We assume that $X$ is connected.
Let $r \in \vert K^{\times} \vert$
be an element
with $a < r \leq b$ and
let $x \in f^{-1}(\eta(r))$ be an element.
Let
\[
\Stab_x:=\{ g \in G \, \vert \, g(x)=x \}
\]
be the stabilizer of $x$.
Let $k(x)^{\wedge h +}$
(resp.\ $k(r)^{\wedge h+}$)
be the Henselization of the completion of the valuation ring
$k(x)^{+}$
(resp.\ $k(\eta(r))^{+}$).
Let $k(x)^{\wedge h}$ and $k(r)^{\wedge h}$ be the fields of fractions of
$k(x)^{\wedge h +}$ and $k(r)^{\wedge h+}$, respectively.
Then by \cite[5.5]{Huber01} the extension of fields
\[
k(r)^{\wedge h} \to k(x)^{\wedge h}
\]
is finite and Galois, and we have a natural isomorphism
\[
\Stab_x \overset{\cong}{\to} \Gal(k(x)^{\wedge h}/k(r)^{\wedge h}).
\]
In \cite{Huber01},
Huber defined higher ramification subgroups and
the Swan character
of
the Galois group
$\Gal(k(x)^{\wedge h}/k(r)^{\wedge h})$.
In \cite{Ramero05},
Ramero investigated the relation between the discriminant functions and the Swan characters.
We are interested in the case where
all higher ramification subgroups and
the Swan character
of
$\Gal(k(x)^{\wedge h}/k(r)^{\wedge h})$ are trivial.
All we need is the following lemma:

\begin{lem}[{\cite[Lemma 3.3.10]{Ramero05}}]\label{Lemma:Swan character}
Let
$
f \colon X \to \B(a, b)
$
be 
a finite Galois \'etale morphism such that $X$ is connected.
We assume that
$
\sharp  \Stab_x
$
is invertible in $\O$ for every $r \in \vert K^{\times} \vert$
with $a < r \leq b$ and every $x \in f^{-1}(\eta(r))$.
Then
the discriminant function
\[
\delta_f \colon [-\log b, -\log a]\to \R_{\geq 0}
\]
associated with $f$ is constant.
\end{lem}
\begin{proof}
This follows from the second equality of \cite[Lemma 3.3.10]{Ramero05}.
Indeed, under the assumption,
we have
$\mathrm{Sw}^{\natural}_x=0$
for the Swan character
$\mathrm{Sw}^{\natural}_x$
attached to $x \in f^{-1}(\eta(r))$ defined in \cite[Section 3.3]{Ramero05}.
\end{proof}

Finally, we prove Theorem \ref{Theorem:trivialization of tame sheaf}.

\begin{proof}[\bf  Proof of Theorem \ref{Theorem:trivialization of tame sheaf}]
In fact, we will show that
if
a locally constant \'etale sheaf $\mathcal{F}$ with finite stalks on
$\B(a, b)$
is tame at $\eta(r) \in \B(a, b)$
for every $r \in \vert K^{\times} \vert$
with $a < r \leq b$,
then,
for every $t \in \vert K^\times \vert$ with $a/b < t^2 < 1$,
there exists an integer $m$ invertible in $\O$
such that
the restriction $\mathcal{F} \vert_{\B(a/t, tb)}$
is trivialized by a Kummer covering $\varphi_m$.

There is a finite Galois \'etale morphism
$f \colon X \to \B(a, b)$
such that $X$ is connected and $f^*\mathcal{F}$ is a constant sheaf.
Let $G$ be the Galois group of $f$.
By replacing $X$ by a quotient of it by a subgroup of $G$
(this makes sense by Remark \ref{Remark:finite etale morphism}),
we may assume that the induced homomorphism
\[
\rho \colon G \to \Aut(\Gamma(X, f^*\mathcal{F}))
\]
is injective.
Let $t \in \vert K^\times \vert$ be an element with $a/b < t^2 < 1$.
By Lemma \ref{Lemma:the number of connected components},
we may assume that $X$ remains connected after restricting to $\B(a/s, s b)$
for every $s \in \vert K^{\times} \vert$ with $t < s \leq 1$.

We claim that
$
\sharp \Stab_x
$
is invertible in $\O$
for every $r \in \vert K^{\times} \vert$
with $a < r \leq b$ and every $x \in f^{-1}(\eta(r))$.
Let $L(r)$ be a separable closure of $k(x)^{\wedge h}$.
It induces a geometric point
$\overline{x} \to X$
with support $x$.
Let $\overline{r} \to \B(a, b)$ denote the composition $\overline{x} \to X \to \B(a, b)$.
Since $f^*\mathcal{F}$ is a constant sheaf, we have the following identifications
\[
\Gamma(X, f^*\mathcal{F}) \cong (f^{*}\mathcal{F})_{\overline{x}} \cong \mathcal{F}_{\overline{r}}.
\]
Recall that we have
$
\Stab_x \cong \Gal(k(x)^{\wedge h}/k(r)^{\wedge h}).
$
Via these identifications,
the action of $\Stab_x \subset G$ on $\Gamma(X, f^*\mathcal{F})$
is compatible
with the action of
$\Gal(L(r)/k(r)^{\wedge h})$ on $\mathcal{F}_{\overline{r}}$.
Since $\mathcal{F}$ is tame at $\eta(r)$
and $\rho$ is injective,
it follows that
$
\sharp \Gal(k(x)^{\wedge h}/k(r)^{\wedge h})
$
is invertible in $\O$.
This proves our claim.

By
Lemma \ref{Lemma:Swan character},
it follows that the discriminant function $\delta_f$ is constant.
By Lemma \ref{Lemma:the number of prime ideals} and Proposition \ref{Proposition:reduction ordinary double point},
there is exactly one point $x$ in $f^{-1}(\eta(b))$.
Therefore the Galois group $G$ is isomorphic to
$
\Stab_x \cong \Gal(k(x)^{\wedge h}/k(b)^{\wedge h}).
$

Theorem \ref{Theorem:trivialization of tame sheaf}
now follows from \cite[Theorem 2.11]{Lutkebohmert93}.
Alternatively, we can argue as follows.
By \cite[Proposition 2.5, Corollary 2.7, and Corollary 5.4]{Huber01},
we see that $G$ is a cyclic group.
Let us write $G \cong \Z/m\Z$.
We consider $f \colon X \to \B(a, b)$ as a $\Z/m\Z$-torsor.
As in the proof of \cite[Theorem 2.4.3]{Ramero05},
since the Picard group of $\B(a, b)$ is trivial,
the Kummer sequence gives an isomorphism
\[
A(a, b)^{\times}/(A(a, b)^{\times})^m \cong H^{1}(\B(a, b), \Z/m\Z).
\]
Since $m$ is invertible in $\O$,
the left hand side is a cyclic group of order $m$ generated by the coordinate function $T \in A(a, b)^{\times}$.
It follows that every $\Z/m\Z$-torsor over $\B(a, b)$ is the disjoint union of Kummer coverings.
This completes the proof of Theorem \ref{Theorem:trivialization of tame sheaf}.
\end{proof}

\subsection*{Acknowledgements}
The author would like to thank his adviser Tetsushi Ito for his support and guidance, and for bringing the paper \cite{LS05} to the attention of the author.
The author would like to thank Teruhisa Koshikawa for
many helpful suggestions and useful conversations.
The author learned a lot from him.
The author would like to thank Yoichi Mieda for helpful comments.
The author also would like to thank Ryota Mikami for useful discussions.
The work of the author was supported by JSPS Research Fellowships for Young Scientists KAKENHI Grant Number 18J22191.

\end{document}